\documentclass[hidelinks,onefignum,onetabnum]{siamart251216}

\usepackage{amsfonts}
\usepackage{amsopn}
\usepackage{array}
\usepackage{booktabs}
\usepackage{etoolbox}
\usepackage{graphicx}
\usepackage{longtable}
\usepackage{multirow}
\usepackage{needspace}
\usepackage{tabularx}

\usepackage{algorithm}
\usepackage{algorithmic}

\usepackage[T1]{fontenc}

\ifpdf
  \DeclareGraphicsExtensions{.pdf,.png,.jpg,.eps}
\else
  \DeclareGraphicsExtensions{.eps}
\fi
\graphicspath{{figures/}}

\newcommand{\rev}[1]{#1}

\DeclareMathOperator{\blockdiag}{blockdiag}

\makeatletter
\long\def\LT@makecaption#1#2#3{%
  \LT@mcol\LT@cols c{%
    \hbox to\z@{\hss
      \parbox[t]{\textwidth}{%
        \def\@captype{table}%
        \@makecaption{#1{#2}}{#3}}%
      \hss}}\\}
\makeatother

\newsiamremark{remark}{Remark}
\crefname{remark}{Remark}{Remarks}
\Crefname{remark}{Remark}{Remarks}

\makeatletter
\long\def\NSDE@cref@first#1#2\NSDE@cref@stop{#1}
\long\def\NSDE@cref@second#1#2#3\NSDE@cref@stop{#2}
\def\cref@getref#1#2{%
  \expandafter\let\expandafter#2\csname r@#1@cref\endcsname
  \expandafter\expandafter\expandafter\def
    \expandafter\expandafter\expandafter#2%
    \expandafter\expandafter\expandafter{%
      \expandafter\NSDE@cref@first#2\NSDE@cref@stop}}
\def\cpageref@getref#1#2{%
  \expandafter\let\expandafter#2\csname r@#1@cref\endcsname
  \expandafter\expandafter\expandafter\def
    \expandafter\expandafter\expandafter#2%
    \expandafter\expandafter\expandafter{%
      \expandafter\NSDE@cref@second#2\NSDE@cref@stop}}
\makeatother

\headers
  {Finite-Budget Convergence and Chaos-Aware Stopping}
  {G. A. Antonucci, R. A. Hauser, D. Samaddar, and J. C. Buchanan}

\title{%
  \rev{Finite-Budget Convergence and Chaos-Aware Stopping Criteria for Time-Parallel Methods}%
  \thanks{%
    Submitted to the editors on \today.
    \funding{%
      This work has been (part-)funded by the EPSRC Energy Programme [grant number EP/W006839/1]. Raphael Hauser was supported by the Hong Kong Innovation and Technology Commission (InnoHK Project CIMDA) and by EPSRC grant EP/Y028872/1 Mathematical Foundations of Intelligence: An ``Erlangen Programme'' for AI. To obtain further information on the data and models underlying this paper please contact \email{PublicationsManager@ukaea.uk}.
    }%
  }%
}

\author{%
  Giancarlo A. Antonucci%
  \protect\footnotemark[2]%
  \protect\hspace{3.5pt}%
  \protect\footnotemark[3]%
  \and
  Raphael A. Hauser%
  \thanks{%
    Mathematical Institute, University of Oxford, Andrew Wiles Building, Woodstock Road, Oxford, OX2 6GG, United Kingdom of Great Britain and Northern Ireland (\email{giancarlo.antonucci@maths.ox.ac.uk}, \email{raphael.hauser@maths.ox.ac.uk})
  }%
  \and
  Debasmita Samaddar%
  \thanks{%
    United Kingdom Atomic Energy Authority, Culham Campus, Abingdon, OX14 3DB, United Kingdom of Great Britain and Northern Ireland (\email{giancarlo.antonucci@ukaea.uk}, \email{debasmita.samaddar@ukaea.uk}, \email{james.buchanan@ukaea.uk})
  }%
  \and
  James C. Buchanan%
  \protect\footnotemark[3]%
}

\ifpdf
  \hypersetup{
    pdftitle  = {Finite-Budget Convergence and Chaos-Aware Stopping Criteria for Time-Parallel Methods},
    pdfauthor = {G. A. Antonucci, R. A. Hauser, D. Samaddar, and J. C. Buchanan},
  }
\fi

\begin{document}

\maketitle

\begin{abstract}
  Time-parallel algorithms, such as Parareal, are well understood for linear problems, but their convergence analysis for nonlinear, chaotic systems remains limited. This paper introduces a theoretical framework for analysing time-decomposition methods as local contraction mappings with a finite iteration budget. We derive a finite-budget guarantee linking the initial error, convergence rate, and iteration count through a geometric outer--inner-ball condition. We apply this framework to Parareal, deriving explicit estimates for the convergence factor on nonlinear problems. A frozen-Jacobian expansion separates tableau mismatch from Jacobian drift; actual Runge--Kutta maps admit short-window contraction under uniform neighbourhood hypotheses. First-batch defects estimate the initial error, and sampled actual-map derivatives give practical budget predictions covering all six tested weighted-error crossings; guarantees require uniform bounds. For chaos-aware stopping, we introduce a \emph{proximity function} that weights the interface defects (the jumps between chunks) and bounds weighted trajectory error. Numerical experiments on the logistic, Lorenz-63, and Lorenz-96 systems assess early acceptance. With accurate coarse propagation, the chaotic tests typically take two or three fine batches, and accurate prefixes can extend beyond exact chunks. In an independent Lorenz-63 test with inadequate coarse propagation, Lyapunov stopping reduces energy-estimation error relative to coarse-only at lower modelled cost than standard stopping and meets the prescribed adequacy test. Its gain over the zero-correction first fine batch is not significant, however, and a less costly intermediate-step serial solver meets the accuracy limit. An independent Lorenz-96 test finds no clear energy-error improvement. Statistical fidelity is assessed separately from weighted acceptance.
\end{abstract}



\section{Introduction}
\label{sec:intro}

Time-parallel integration methods aim to overcome the sequential bottleneck of time stepping. The main families are shooting, waveform relaxation, multigrid, and direct methods~\cite{gander_50_2015,gander_time_2024}. Parareal~\cite{lions_resolution_2001} is a shooting method with established convergence analyses for linear and nonlinear problems~\cite{gander_analysis_2007,staff_stability_2005,gander_nonlinear_2008,bal_convergence_2005}; fine--coarse phase mismatch limits its performance on hyperbolic problems~\cite{ruprecht_wave_2018}.

For chaotic dynamics, early discrepancies amplify exponentially, so a global pointwise stopping target can demand nearly one correction sweep per chunk. Samaddar et al.\ proposed judging turbulent simulations statistically~\cite{samaddar_parallelization_2010}; Vargas et al.\ modified the nonlinear coarse correction and coarsening of multigrid reduction in time~\cite{falgout_parallel_2014} to obtain speedup on chaotic systems~\cite{vargas_multigrid_2023}.

In this paper, we present a framework for analysing time-parallel algorithms for nonlinear systems. The framework links local contraction estimates and initial-error bounds to a prescribed tolerance and a fixed, finite computational budget. Practical efficiency requires the number of correction sweeps to be small compared with the number of chunks. Parareal is algebraically exact after $N$ sweeps; our outer--inner-ball construction gives a smaller sufficient budget when the contraction and initial-error bounds permit it.

We apply this framework to Parareal, deriving explicit estimates for its convergence factor on nonlinear problems. The transport--source bound and first-batch defects supply the contraction and initial-error inputs. Uniform neighbourhood bounds yield guarantees, whereas sampled actual-map derivatives provide practical predictions. The frozen-Jacobian expansion separates tableau mismatch from Jacobian drift. For chaotic systems, we also introduce a chaos-aware stopping criterion based on the \emph{proximity function}, a weighted measure of interface defects, and prove two-sided bounds relating it to trajectory error in the same norm.

Using interface errors, marginal distributions, and statistical functionals, we compare the accepted trajectories with fine and coarse-only baselines, the first fine batch without corrections, and an intermediate-step serial solver. These comparisons distinguish correction progress from statistical benefit. Accuracy and cost for moving windows remain open questions (\Cref{sec:future_work}).

\section{Problem formulation and fixed-point structure}
\label{sec:generalities}

Consider an autonomous initial-value problem defined over a finite time domain $[0, T]$:
\begin{equation}
  \label{eq:ivp}
  u'(t) = f(u(t)),
  \qquad u(0) = u_0,
  \qquad f : \mathbb{R}^d \to \mathbb{R}^d,
  \quad 0 \le t \le T.
\end{equation}
We divide this domain into $N$ time chunks of equal size $\Delta T = T/N$, with endpoints $T_n = n\Delta T$. Parareal combines a fine solver $\mathcal{F}$, accurate but computationally expensive, with a less costly coarse solver $\mathcal{G}$ that provides the sequential coarse correction. Their Runge--Kutta steps are $h$ and $H = \xi h$, with $L$ coarse steps per chunk, so $\Delta T = LH = L\xi h$. Their tableaux are $(A_{\mathcal{F}}, b_{\mathcal{F}}, c_{\mathcal{F}})$ and $(A_{\mathcal{G}}, b_{\mathcal{G}}, c_{\mathcal{G}})$, with consistency $b^\top \mathbf{1} = 1$ and row sums $A\mathbf{1} = c$. \Cref{fig:timegrid} depicts the grids.

\begin{figure}[!hbp]
  \centering
  \includegraphics[width=\textwidth]{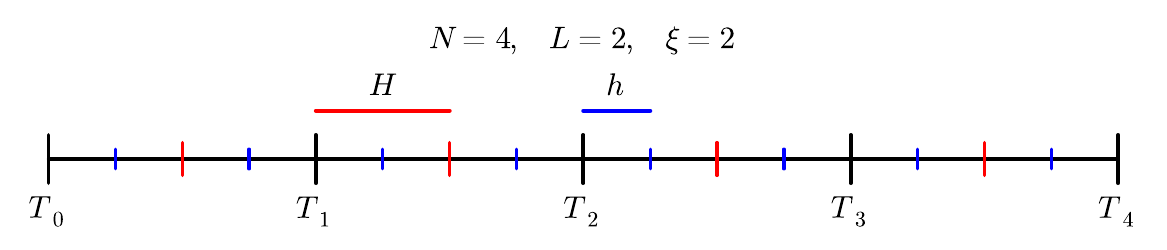}
  \caption{Fine steps $h$, coarse steps $H = \xi h$, and chunks $\Delta T = LH$ are different scales. Refinement at fixed chunk length doubles $L$ as $h, H$ are halved; shortening at fixed $L, \xi, N$ instead reduces the window length.}
  \label{fig:timegrid}
\end{figure}

\paragraph{Notation and assumptions}
We use double bars for norms, including state norms and their induced matrix norms, and single bars for scalar absolute values. On the affine space $\mathcal{X} = \{U = (U_0, \ldots, U_N) : U_0 = u_0\}$, the corresponding trajectory norm is
\begin{equation}
  \label{eq:trajectory_norm}
  \begin{gathered}
    \|U - V\|_p = \left(\sum_{n = 1}^N \|U_n - V_n\|^p\right)^{1/p},
    \qquad 1 \le p < \infty, \\
    \|U - V\|_\infty = \max_{1 \le n \le N} \|U_n - V_n\|.
  \end{gathered}
\end{equation}
We use a fixed norm on $\mathcal{X}$ unless stated otherwise, suppressing its subscript when no ambiguity arises. In that norm, $B_\rho(V) = \{U \in \mathcal{X} : \|U - V\| < \rho\}$ and $\overline{B}_\rho(V) = \{U \in \mathcal{X} : \|U - V\| \le \rho\}$. All numerical tests use Euclidean state norms and $p = 2$; \Cref{SM-sec:notation} collects the symbols.

We fix an open state neighbourhood $\mathcal{N} \subseteq \mathbb{R}^d$ containing the reference trajectory and its fine interfaces, and use the following assumptions.
\begin{itemize}
  \item[(A1)] The IVP is well posed on $[0, T]$, and $f \in C^{p_* + 2}(\mathcal{N})$, where $p_* = \max\{p_{\mathcal{F}}, p_{\mathcal{G}}\}$ is the larger Runge--Kutta order.
  \item[(A2)] The chosen stage solutions define single-valued $C^1$ chunk maps $\mathcal{F}, \mathcal{G}$ on $\mathcal{N}$. For implicit methods, this includes a locally unique differentiable branch of stage solutions.
\end{itemize}
Assumption (A2) defines the discrete iteration and its derivatives; (A1) supports its interpretation as an IVP approximation and the Runge--Kutta estimates. Each local result specifies its trajectory neighbourhood and quantitative bounds, and all state and stage evaluations it uses must lie in $\mathcal{N}$; invariance is proved where required, rather than assumed for arbitrary starting trajectories.

\subsection{The discrete target and finite-step exactness}

Under (A2), the serial fine solution is defined recursively by $U_0^* = u_0$, $U_n^* = \mathcal{F}(U_{n - 1}^*)$, provided these states remain in $\mathcal{N}$. Equivalently, we seek the zero of the stacked residual
\begin{equation}
  \label{eq:rootfinding}
  \mathcal{E}_{\mathcal{F}}(U) =
  \begin{bmatrix}
    U_0 - u_0 \\
    U_1 - \mathcal{F}(U_0) \\
    \vdots \\
    U_N - \mathcal{F}(U_{N - 1})
  \end{bmatrix}
  = 0,
  \qquad U = [U_0^\top, \ldots, U_N^\top]^\top.
\end{equation}
We call its blocks $d_n(U) = U_n - \mathcal{F}(U_{n - 1})$, $1 \le n \le N$, the \emph{defects}. On $\mathcal{X}$ the first block vanishes, and the residual consists of these defects. Its Jacobian is block-bidiagonal and lower triangular. Serial time stepping uses forward substitution: computing $U_1$ from $U_0$, then $U_2$ from $U_1$, and so forth. Time-parallel methods can be viewed as iterative schemes designed to solve this nonlinear algebraic system without strictly sequential propagation.

To separate a general fact from the special structure of Parareal, first consider an iteration $U^{s + 1} = \Psi(U^s)$ on an invariant set $\mathcal{D} \subseteq \mathcal{X}$.
\begin{lemma}[Finite termination and uniqueness]
  \label{lem:finite_termination}
  Let $\Psi : \mathcal{D} \to \mathcal{D}$, with $U^* \in \mathcal{D}$. If some integer $M \ge 1$ satisfies $\Psi^M(V) = U^*$ for every $V \in \mathcal{D}$, then $U^*$ is the unique fixed point of $\Psi$.
\end{lemma}
\begin{proof}
  Since $\Psi(U^*) \in \mathcal{D}$, we have $\Psi(U^*) = \Psi(\Psi^M(U^*)) = \Psi^M(\Psi(U^*)) = U^*$. If $V$ is fixed, then $V = \Psi^M(V) = U^*$.
\end{proof}
The hypothesis requires recovery by a common iteration count: each trajectory reaches $U^*$ within $M$ steps and stays there. A single visit does not suffice.

Parareal starts with the coarse pass $U_n^0 = \mathcal{G}(U_{n - 1}^0)$ and applies $U^{s + 1} = \Phi(U^s) = U^s + v(U^s)$, where
\begin{equation}
  \label{eq:parareal_update}
  \Phi_0(U) = u_0,
  \qquad \Phi_n(U) = \mathcal{G}(\Phi_{n - 1}(U)) + \mathcal{F}(U_{n - 1}) - \mathcal{G}(U_{n - 1}).
\end{equation}
The fine terms are computed in parallel and the new coarse terms sequentially. To recover the approximate-Newton viewpoint, define $\mathcal{E}_{\mathcal{G}}$ as in \Cref{eq:rootfinding}, with $\mathcal{G}$ in place of $\mathcal{F}$. Newton's step and the Parareal correction respectively satisfy
\begin{equation}
  \label{eq:newton_vs_parareal}
  \begin{gathered}
    U_{\mathrm{Newton}} = U - [D\mathcal{E}_{\mathcal{F}}(U)]^{-1}\mathcal{E}_{\mathcal{F}}(U), \\
    \mathcal{E}_{\mathcal{G}}(\Phi(U)) = \mathcal{E}_{\mathcal{G}}(U) - \mathcal{E}_{\mathcal{F}}(U).
  \end{gathered}
\end{equation}
Thus Parareal uses the coarse nonlinear residual as a preconditioner. For an affine coarse map, this reduces to replacing $D\mathcal{E}_{\mathcal{F}}$ by $D\mathcal{E}_{\mathcal{G}}$ in Newton's step; for nonlinear $\mathcal{G}$, the residual identity is the precise interpretation~\cite{dembo_inexact_1982,gander_analysis_2007}.

\begin{proposition}[Finite-step exactness]
  \label{prop:unique_fixed_point}
  Under (A2), for any $U^0 \in \mathcal{X}$ whose required propagation states remain in $\mathcal{N}$, $U_n^s = U_n^*$ whenever $n \le s$. Consequently $U^N = U^*$ in exact arithmetic, and $U^*$ is the unique fixed point of $\Phi$.
\end{proposition}
\begin{proof}
  We prove the prefix assertion by induction on $s$. It holds for $s = 0$ because $U_0^0 = u_0$. Suppose it holds at sweep $s$. For $1 \le n \le s + 1$, the old input $U_{n - 1}^s$ is exact. The new input $U_{n - 1}^{s + 1}$ is also exact: this holds at $n = 1$ and follows successively in $n$. The coarse terms in \Cref{eq:parareal_update} therefore cancel, giving $U_n^{s + 1} = \mathcal{F}(U_{n - 1}^*) = U_n^*$. Hence $U^N = U^*$. Substitution shows that $U^*$ is fixed. Any other fixed point would remain unchanged for $N$ sweeps, yet equal $U^*$ by the prefix result, proving uniqueness as in \Cref{lem:finite_termination}.
\end{proof}

For early acceptance, the same update describes what remains beyond the exact prefix. Subtracting $\mathcal{F}(U_{n - 1}^{s + 1})$ from \Cref{eq:parareal_update} gives
\begin{equation}
  \label{eq:tail_identity}
  d_n^{s + 1} = \bigl(\mathcal{F} - \mathcal{G}\bigr)(U_{n - 1}^{s})
  - \bigl(\mathcal{F} - \mathcal{G}\bigr)(U_{n - 1}^{s + 1}).
\end{equation}
Thus the tail is a fine-map pseudo-orbit whose interface defects equal the change in the fine--coarse discrepancy between sweeps, not generally a coarse local truncation error. Exact recovery after $N$ sweeps and earlier acceptance have different targets.

The lemma is independent of Parareal and applies to other iterations once its finite-termination hypothesis has been established, including the two-level MGRIT realization of Parareal with F-relaxation~\cite{falgout_parallel_2014}. We do not assume this property for PFASST, whose iterations use deferred-correction sweeps~\cite{emmett_efficient_2012}. The contraction framework below does not require finite termination and can be applied to such methods on their full iteration space, including internal stage values, once its hypotheses are verified.

\subsection{Physical interpretation of the target}
\label{sec:physical_target}

Assumption (A1) supplies the well-posed continuous target; with consistent, stable time stepping it supports convergence to that target on a fixed interval as $h \to 0$. For chaotic dynamics, however, an error of size $\eta$ can grow on the scale $\eta e^{\lambda t}$, producing a practical predictability horizon of order $\lambda^{-1} \log(1/\eta)$~\cite{lorenz_reply_2008}. This separates the fine discrete orbit from the exact IVP orbit, while leaving both Parareal's algebraic exactness in \Cref{prop:unique_fixed_point} and finite-time convergence as $h \to 0$ intact. Convergence on fixed time intervals does not by itself establish accurate long-time statistics.

For the attractor discussion only, (A3) assumes a compact absorbing set and an appropriate positively invariant numerical neighbourhood. Under the additional hypotheses of numerical-attractor approximation theorems, upper semicontinuity places the numerical attractor near the continuous one~\cite[Chapter 7]{stuart_dynamical_1996}, but agreement of invariant measures or finite-time statistics requires more than this setwise conclusion. A $\delta$-pseudo-orbit satisfies $\|V_n - \mathcal{F}(V_{n - 1})\| \le \delta$; shadowing compares it with some nearby true orbit, which may have a different initial condition and, for flows, may require a reparametrization of time~\cite{pilyugin_shadowing_1999,vanvleck_shadowing_1995}. \Cref{fig:shadowing} illustrates this distinction. We assume no shadowing theorem for the Lorenz-63 tests; their statistical accuracy is assessed empirically in \Cref{sec:statistics,sec:lorenz,sec:lorenz96}.

\begin{figure}[!htbp]
  \centering
  \includegraphics[width=\textwidth]{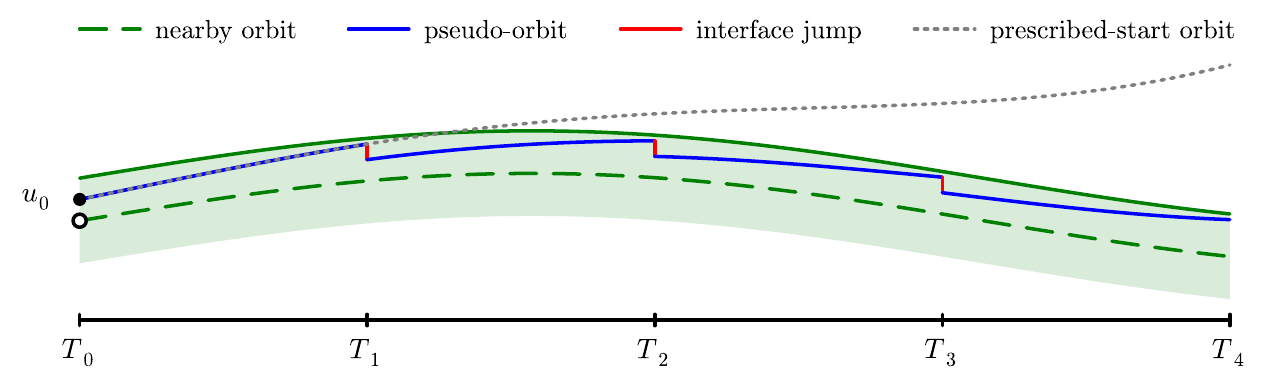}
  \caption{Conditional shadowing comparison: a pseudo-orbit remains near an orbit with a different initial condition while the prescribed-start orbit separates. Interface jumps measure local inconsistency; the geometric comparison alone supplies no statistical guarantee.}
  \label{fig:shadowing}
\end{figure}

\section{A framework for finite-budget convergence}
\label{sec:framework}

Practical time-parallel integration requires convergence to a specified tolerance within a finite computational budget. In any fixed norm, we ask how accurately to initialize the iteration to meet that budget. For Parareal, useful budgets are below its $N$-sweep exactness limit. Here $s$ counts correction sweeps and $K_F$ counts parallel fine batches; \Cref{sec:experiments} considers their cost.

\Needspace{7\baselineskip}
\subsection{Rates and local contraction}
\label{sec:rates}

\begin{definition}[Q-convergence]
  \label{def:convergence_rate}
  For a convergent sequence with nonzero errors $e^s = U^s - U^*$ eventually, define
  \begin{equation}
    \label{eq:q_ratio}
    Q_q = \limsup_{s \to \infty}\frac{\|e^{s + 1}\|}{\|e^s\|^q},
    \qquad q \ge 1.
  \end{equation}
  The sequence is Q-linearly convergent if $Q_1 < 1$, Q-superlinearly convergent if $Q_1 = 0$, and of at least Q-order $q > 1$ if $Q_q < \infty$~\cite{ortega_iterative_1970}.
\end{definition}
Because finite termination makes these ratios eventually undefined, we do not assign Parareal an asymptotic Q-order. Instead, the useful hypothesis for a finite budget is a contraction bound valid at \emph{every} iterate in a specified neighbourhood.

\begin{theorem}[Sufficient conditions for local contraction]
  \label{thm:ostrovski}
  Let $\Phi = I + v$ satisfy $\Phi(U^*) = U^*$, with $v$ continuously differentiable on a neighbourhood in $\mathcal{X}$ of the closed ball $\Omega = \overline{B}_\varrho(U^*)$. If
  \begin{equation}
    \label{eq:local_contraction}
    \sup_{U \in \Omega}\|I + Dv(U)\| = \sup_{U \in \Omega}\|D\Phi(U)\| \le \beta < 1,
  \end{equation}
  then every sequence starting at $U^0 \in \Omega$ remains in $\Omega$ and satisfies $\|U^s - U^*\| \le \beta^s\|U^0 - U^*\|$ for every $s \ge 0$. Moreover, $U^*$ is the only fixed point in $\Omega$.
\end{theorem}
\begin{proof}
  The segment from $U^*$ to $U$ lies in $\Omega$. The integral mean-value inequality applied to
  \begin{equation}
    \label{eq:mean_value}
    \Phi(U) - \Phi(U^*) = \int_0^1 D\Phi\bigl(U^* + t(U - U^*)\bigr)(U - U^*)\, dt
  \end{equation}
  gives $\|\Phi(U) - U^*\| \le \beta\|U - U^*\|$. This proves invariance and the error bound by induction. If $V \in \Omega$ is fixed, the same inequality gives $\|V - U^*\| \le \beta\|V - U^*\|$, so $V = U^*$.
\end{proof}
This is the standard contraction argument~\cite{ortega_iterative_1970}; the derivative and identity operator in \Cref{eq:local_contraction} act on trajectory perturbations with zero initial block. Assumption (A2) gives Parareal's differentiability when all old and new evaluation states lie in $\mathcal{N}$; contraction additionally requires the derivative bound in \Cref{eq:local_contraction}.

\subsection{Basins of attraction and outer and inner balls}
\label{sec:acceptance_geometry}

To explain the geometry behind the budget, consider a continuous iteration $\Phi : \mathcal{D} \to \mathcal{D}$ on an invariant domain $\mathcal{D} \subseteq \mathcal{X}$ containing $U^*$. Its asymptotic basin is
\begin{equation}
  \label{eq:basin}
  \mathcal{A}(U^*) = \{U^0 \in \mathcal{D} : \Phi^s(U^0) \to U^*\text{ as }s \to \infty\}.
\end{equation}
For a finitely terminating iteration this basin is all of $\mathcal{D}$; for other iterations it may be a proper subset. In particular, Parareal's finite-step exactness leaves no fractal asymptotic basin within such a domain. The practical question is instead which starts reach the target within a small budget.

The defect landscapes in \Cref{SM-fig:ErrorContoursLorenz} motivate bounding a ball within the acceptance set without inferring its geometry from a plot.

\Needspace{10\baselineskip}
For an integer budget $s \ge 0$, let $0 < r \le \varrho$, where $r$ is the target tolerance and $\varrho$ is the radius of the contraction ball $\Omega = \overline{B}_\varrho(U^*) \subseteq \mathcal{D}$ in \Cref{thm:ostrovski}. We call $\overline{B}_r(U^*)$ the \emph{inner ball} and $B_{R(s)}(U^*)$ the \emph{outer ball}, with the certified outer radius
\begin{equation}
  \label{eq:rho_bound_linear}
  R(s) = \min\{\varrho, r\beta^{-s}\},
  \qquad 0 < \beta < 1.
\end{equation}
The inner ball represents the target zone; the outer ball bounds the allowable coarse-start error. Every start in the outer ball reaches error less than $r$ within $s$ sweeps; its closure reaches error at most $r$. At fixed budget, $\beta \approx 1$ makes this sufficient initial-error bound close to $r$; a smaller $\beta$ allows a larger bound, capped by $\varrho$. Once $r\beta^{-s} \ge \varrho$, enlarging $R$ requires contraction on a larger ball.

The outer ball lies in the finite-budget acceptance set
\begin{equation}
  \label{eq:acceptance}
  \mathcal{A}_r(U^*;s) = \{U^0 \in \mathcal{D} : \|\Phi^s(U^0) - U^*\| \le r\}.
\end{equation}
Denote by $R_{\max}(s) \ge R(s)$ the supremal radius of centred balls contained in this set, possibly infinite. The acceptance region may extend beyond the certified outer ball; our guarantees use only $R(s)$.

\subsection{Linking initial error to sweep budget}
\label{sec:budget}

We can now derive an explicit link between the required coarse accuracy, the iteration count, and the convergence rate. Reversing the outer-radius bound gives a sufficient budget from $\|e^0\| \le E_0 \le \varrho$:
\begin{equation}
  \label{eq:budget_box}
  s_{\mathrm{pred}}(E_0) =
  \begin{cases}
    0, & E_0 \le r, \\
    1, & E_0 > r, \ \beta = 0, \\
    \lceil \log(E_0/r)/\log(1/\beta) \rceil,
    & E_0 > r, \ 0 < \beta < 1.
  \end{cases}
\end{equation}
Every $s \ge s_{\mathrm{pred}}$ gives $\|e^s\| \le r$. The case $\beta = 0$ also gives $R(0) = r$ and $R(s) = \varrho$ for $s \ge 1$.

For $q > 1$, a uniform recursion $\|e^{j + 1}\| \le \beta\|e^j\|^q$ on an invariant neighbourhood gives
\begin{equation}
  \label{eq:higher_order_recursion}
  \|e^s\| \le \beta^{(q^s - 1)/(q - 1)}\|e^0\|^{q^s}.
\end{equation}
This finite-step statement needs more than asymptotic Q-order; \Cref{SM-sec:higher_order} derives its outer radius. Our working prediction uses the linear-contraction bound.

The prediction requires both a contraction factor, derived for Parareal in \Cref{sec:parareal_bounds}, and an initial-error bound. All coarse-pass defects are available after the first fine batch; \Cref{sec:certify} converts them into $E_0$ and combines both inputs in \Cref{alg:certify}.

\section{Explicit convergence bounds for Parareal}
\label{sec:parareal_bounds}

We derive explicit bounds for the convergence factor $\beta$ for the Parareal algorithm applied to nonlinear problems. We first use the unweighted trajectory norm; \Cref{sec:weighted_bound} adapts the estimate to the stopping criterion.

\subsection{The general bound: transport times source}
\label{sec:general_bound}

\begin{theorem}[Parareal derivative bound]
  \label{thm:parareal_rate_beta}
  Under (A2), let $\Phi$ be the Parareal map, with all old and new evaluation states in $\mathcal{N}$. For $U \in \Omega \subset \mathcal{X}$, define
  \begin{equation}
    \label{eq:AD_blocks}
    A_j(U) = D\mathcal{G}(\Phi_j(U)),
    \qquad D_j(U) = D\mathcal{F}(U_j) - D\mathcal{G}(U_j).
  \end{equation}
  The derivative on perturbations with zero initial block is strictly lower triangular, with
  \begin{equation}
    \label{eq:block_jacobian}
    (D\Phi(U))_{nm} = A_{n - 1}(U)\cdots A_{m + 1}(U)D_m(U),
    \quad 1 \le m < n \le N,
  \end{equation}
  and zero otherwise, with empty products equal to the identity. If $\|A_j(U)\| \le \Lambda_{\mathcal{G}}$ and $\|D_j(U)\| \le S$ uniformly for $U \in \Omega$, then for every $1 \le p \le \infty$
  \begin{equation}
    \label{eq:beta_unweighted}
    \sup_{U \in \Omega}\|D\Phi(U)\|_p \le
    \beta := S\sum_{i = 0}^{N - 2}\Lambda_{\mathcal{G}}^i.
  \end{equation}
  The sum is empty and equals zero when $N = 1$. On a closed ball about $U^*$, a bound $\beta < 1$ therefore supplies the contraction hypothesis of \Cref{thm:ostrovski}.
\end{theorem}
\begin{proof}
  Differentiating \Cref{eq:parareal_update} gives $\delta\Phi_n = A_{n - 1}\delta\Phi_{n - 1} + D_{n - 1}\delta U_{n - 1}$. Stacking the free blocks $1, \ldots, N$ gives the factorization
  \begin{equation}
    \label{eq:transport_source}
    I + Dv = D\Phi =
    \underbrace{\begin{bmatrix}
      I & & & \\
      A_1 & I & & \\
      \vdots & \ddots & \ddots & \\
      A_{N - 1}\cdots A_1 & \cdots & A_{N - 1} & I
    \end{bmatrix}}_{\text{coarse transport}}
    \underbrace{\begin{bmatrix}
      0 & & & \\
      D_1 & 0 & & \\
      & \ddots & \ddots & \\
      & & D_{N - 1} & 0
    \end{bmatrix}}_{\text{fine--coarse mismatch}},
  \end{equation}
  with blank entries zero and the dependence on $U$ suppressed. This yields \Cref{eq:block_jacobian}. The mismatch already shifts by one block, leaving at most $N - 2$ coarse factors in a nonzero block. A block shift with blocks bounded by $a$ has induced $\ell^p$ norm at most $a$, so summing the subdiagonal bounds $S\Lambda_{\mathcal{G}}^i$ proves the estimate.
\end{proof}
The transport sum is $(\Lambda_{\mathcal{G}}^{N - 1} - 1)/(\Lambda_{\mathcal{G}} - 1)$ for $\Lambda_{\mathcal{G}} \ne 1$, and $N - 1$ otherwise. Changing the coarse solver changes both the mismatch $S$ and its transport. Coarse derivative bounds must cover the new states $\Phi_j(U)$ as well as any old evaluation states; coarse-pass samples alone cannot supply a uniform bound.

\begin{remark}[One source of conservativeness]
  \label[remark]{rem:transport}
  Replacing $\|A_{n - 1}\cdots A_{n - i}\|$ by $\Lambda_{\mathcal{G}}^{i}$ discards cancellation and alignment information. Uniform coarse-product bounds give the sharper estimate
  \begin{equation}
    \label{eq:transport_products}
    \begin{gathered}
      \Gamma_i = \sup_{U \in \Omega}\max_n\|A_{n - 1}\cdots A_{n - i}\|,
      \qquad \Gamma_0 = 1, \\
      \sup_{U \in \Omega}\|D\Phi(U)\|_p \le S\sum_{i = 0}^{N - 2}\Gamma_i.
    \end{gathered}
  \end{equation}
  Products along one trajectory are samples of these suprema.
\end{remark}

\subsection{Runge--Kutta source terms}
\label{sec:source_term}

The frozen-Jacobian expansion separates tableau mismatch from Jacobian drift, without resolving higher-order cancellations in the actual nonlinear RK maps. Actual-map derivative estimates then establish short-window contraction; sampled versions give the budget predictions. Write $J(t) = Df(u(t))$ and set
\begin{equation}
  \label{eq:jacobian_products}
  \begin{gathered}
    F_{\mathrm{fr}}' = \prod_{i = 1}^{\xi L}\mathcal{R}_{\mathcal{F}}(X_i), \\
    X_i = hJ(T_n + (i - 1)h),
  \end{gathered}
  \qquad
  \begin{gathered}
    G_{\mathrm{fr}}' = \prod_{j = 1}^{L}\mathcal{R}_{\mathcal{G}}(Y_j), \\
    Y_j = \xi hJ(T_n + (j - 1)\xi h).
  \end{gathered}
\end{equation}
Here later times multiply on the left, and
\begin{equation}
  \label{eq:stability_function}
  \mathcal{R}(Z) = I + (b^\top \otimes I)(I - A \otimes Z)^{-1}(\mathbf{1} \otimes Z).
\end{equation}
For autonomous linear problems this represents the true derivatives~\cite[Section~IV.3]{hairer_solving_1996}, whereas nonlinear Runge--Kutta stages generally sample different Jacobians.

\begin{theorem}[Frozen-Jacobian expansion]
  \label{thm:parareal_beta_second_term}
  Under (A1), on reference trajectories in $\mathcal{N}$, assume $\|J(t)\| \le \mu$, $\|J'(t)\| \le \nu$, with $\mu > 0$, $\nu \ge 0$, and sufficient smoothness for a uniform third-order remainder. Hold $L, \xi$ fixed, with $\xi$ an integer. If
  \begin{equation}
    \label{eq:step_size_condition}
    h\mu\max\{\|A_{\mathcal{F}}\|_\infty, \xi\|A_{\mathcal{G}}\|_\infty\} < 1,
  \end{equation}
  where $\|A\|_\infty = \max_i\sum_j|a_{ij}|$, then the maps in \Cref{eq:jacobian_products} satisfy
  \begin{equation}
    \label{eq:parareal_explicit_bound}
    \|F_{\mathrm{fr}}' - G_{\mathrm{fr}}'\| \le L\xi(h\mu)^2(C_1 + C_2) + O(h^3),
  \end{equation}
  where
  \begin{equation}
    \label{eq:C1C2}
    C_1 = \left|b_{\mathcal{F}}^\top c_{\mathcal{F}} - \xi b_{\mathcal{G}}^\top c_{\mathcal{G}} + \frac{\xi - 1}{2}\right|,
    \qquad C_2 = \frac{\xi - 1}{2}\frac{\nu}{\mu^2}.
  \end{equation}
  The estimate is uniform on a family of trajectories when its hypotheses and remainder bounds hold uniformly there.
\end{theorem}
\begin{proof}
  Put $a_{\mathcal{F}} = b_{\mathcal{F}}^\top c_{\mathcal{F}}$ and $a_{\mathcal{G}} = b_{\mathcal{G}}^\top c_{\mathcal{G}}$. The Neumann series in \Cref{eq:stability_function}, consistency, and $A\mathbf{1} = c$ give $\mathcal{R}(Z) = I + Z + (b^\top c)Z^2 + O(h^3)$. The stated smoothness and fixed step counts make the remainders uniform.

  In coarse block $j$, write $i = \xi(j - 1) + m + 1$, $0 \le m < \xi$. Integrating $J'$ between the coarse and fine nodes gives $X_i = Y_j/\xi + E_i$, with $\|E_i\| \le mh^2\nu$. Thus the linear difference $\Delta_j = \sum_{i\text{ in block }j}X_i - Y_j$ satisfies
  \begin{equation}
    \label{eq:frozen_block_drift}
    \|\Delta_j\| \le h^2\nu\sum_{m = 0}^{\xi - 1}m
    = h^2\nu\,\frac{\xi(\xi - 1)}{2}.
  \end{equation}
  Let $B_j$ be the product of the $\xi$ fine stability factors in this block, with later factors on the left. Each squared term is $Y_j^2/\xi^2 + O(h^3)$, as is each ordered pair from distinct fine factors. There are $\xi$ squared terms and $\xi(\xi - 1)/2$ pairs. Consequently,
  \begin{equation}
    \label{eq:frozen_block_difference}
    B_j - \mathcal{R}_{\mathcal{G}}(Y_j)
    = \Delta_j + \left(\frac{a_{\mathcal{F}}}{\xi} - a_{\mathcal{G}}
    + \frac{\xi - 1}{2\xi}\right)Y_j^2 + O(h^3).
  \end{equation}
  This difference is $O(h^2)$. Telescoping the $L$ block products multiplies each difference on either side by $I + O(h)$, so summing \Cref{eq:frozen_block_difference} changes the result only by $O(h^3)$. Use \Cref{eq:frozen_block_drift} and $\sum_j\|Y_j^2\| \le L(\xi h\mu)^2$, then factor out $L\xi(h\mu)^2$ to obtain $C_1 + C_2$. The supplement expands the individual product terms in \Cref{SM-sec:second_term}.
\end{proof}

\begin{remark}[Scope of the expansion]
  \label[remark]{rem:fixed_deltaT}
  For two methods of order at least two, $b^\top c = 1/2$ and $C_1 = 0$; for explicit- or implicit-Euler pairs $C_1 = (\xi - 1)/2$. Thus $C_1$ combines the tableau coefficients with the cross terms from composing fine steps. The drift term $C_2$ measures variation sampled within a coarse step by the fine model but not by the coarse model; it is not a sharp estimate for high-order methods that resolve stage variation. The Neumann condition is sufficient for this expansion, not a necessary condition for numerical stability or stage solvability. For explicit Euler the expansion is exact per step; its remainder comes only from products.
\end{remark}
Our main limit holds $L, \xi, N$ fixed while $h$ decreases, shortening the window $T = NL\xi h$. In the frozen model the transport sum tends to $N - 1$ and the source estimate is second order: when its leading coefficient is nonzero, halving $h$ reduces the leading estimate of $\beta$ by a factor of about four. To establish contraction of the implemented methods, however, we also need the derivatives of their actual chunk maps.

\Needspace{15\baselineskip}
\begin{proposition}[Actual Runge--Kutta maps]
  \label{prop:actual_rk}
  Under (A1)--(A2), restrict to a common bounded set of initial states whose exact trajectories and numerical stages remain in a compact subset of $\mathcal{N}$. Assume the chosen stage branches are uniformly regular, so the methods of orders $p_{\mathcal{F}}, p_{\mathcal{G}} \ge 1$ have uniform derivative-error estimates there. At fixed $L, \xi$, with a fixed number of steps in each shrinking chunk,
  \begin{equation}
    \label{eq:actual_source_short}
    S = \sup\|D\mathcal{F} - D\mathcal{G}\|
    = O(h^{p_{\mathcal{F}} + 1} + H^{p_{\mathcal{G}} + 1}).
  \end{equation}
  Alternatively, refinement at fixed $\Delta T$ gives
  \begin{equation}
    \label{eq:actual_source}
    S \le C_{\mathcal{F}}h^{p_{\mathcal{F}}} + C_{\mathcal{G}}H^{p_{\mathcal{G}}}.
  \end{equation}
\end{proposition}
\begin{proof}
  Apply the methods to the variational equation together with the state equation. The resulting tangent maps are the derivatives of the numerical maps. A fixed number of steps retains the local orders $p_{\mathcal{F}} + 1, p_{\mathcal{G}} + 1$; a fixed integration time gives the global orders $p_{\mathcal{F}}, p_{\mathcal{G}}$. Insert the derivative of the exact flow between $D\mathcal{F}$ and $D\mathcal{G}$ and use the triangle inequality~\cite{hairer_nonstiff_1993,stuart_dynamical_1996}.
\end{proof}

\begin{corollary}[Contraction on short windows]
  \label{cor:short_windows}
  Fix $L, \xi, N$ and $\varrho > 0$. Suppose the hypotheses of \Cref{thm:parareal_rate_beta,prop:actual_rk} hold uniformly for all sufficiently small $h$ on closed balls $\Omega_h = \overline{B}_\varrho(U_h^*)$ about the corresponding fine trajectories. Assume all required evaluation states and stages, including the coarse pass, lie in one compact subset of $\mathcal{N}$. Then one may choose $\Lambda_{\mathcal{G}} = 1 + O(h)$, and
  \begin{equation}
    \label{eq:short_window_beta}
    \beta \le S\, [N - 1 + O(h)]
    = O(h^{p_{\mathcal{F}} + 1} + H^{p_{\mathcal{G}} + 1}),
    \qquad H = \xi h.
  \end{equation}
  Consequently, for all sufficiently short windows $T = NL\xi h$, Parareal contracts on $\Omega_h$, and the coarse initial trajectory lies in that ball.
\end{corollary}
\begin{proof}
  Consistency and uniform stage regularity give $D\mathcal{G} = I + O(h)$ on a fixed number of steps. The transport sum is therefore $N - 1 + O(h)$, and \Cref{eq:actual_source_short} makes $\beta < 1$ for sufficiently small $h$. The state-error estimates similarly give $\|U_h^0 - U_h^*\|_p = O(h^{p_{\mathcal{F}} + 1} + H^{p_{\mathcal{G}} + 1})$ over the fixed number of steps in the window. This eventually falls below $\varrho$, so \Cref{thm:ostrovski} applies.
\end{proof}

With $L, \xi, N$ fixed, we can estimate $\beta$ and shorten the window until $\beta < 1$. The coarse start also eventually enters the outer ball for each fixed target $0 < r \le \varrho$. Constants may depend on $N$, so no uniform-in-$N$ scaling conclusion follows.

A common threshold along a trajectory requires hypotheses uniform over starting states, including a positive ball radius; (A3) alone is insufficient. Moving windows also require a global analysis of endpoint errors. At fixed $\Delta T = 1$, coarse refinement on eight common states gives fitted source slopes 4.06 for RK4 on Lorenz-63 and 1.08 for implicit Euler on the logistic problem (\Cref{SM-tab:refinement}). These support \Cref{eq:actual_source} while the fine error is smaller; they test fixed-chunk refinement, not the shrinking-window limit.

\section{Chaos-aware stopping criteria}
\label{sec:proximity}

Standard checks such as $\|U^s - U^{s - 1}\|_2 \le \epsilon$~\cite{samaddar_parallelization_2010,aubanel_scheduling_2011} require a convergence estimate to control error. In our chaotic tests, agreement over the full horizon often requires nearly $N$ sweeps, although an accurate prefix appears earlier; beyond the practical predictability horizon, reproducing $U^*$ does not establish physical accuracy (\Cref{sec:physical_target}).

To address this, we introduce a \emph{proximity function} that weights interface defects, emphasizing early errors, which have longer to amplify. It bounds weighted trajectory error without computing $U^*$; statistical fidelity is assessed separately (\Cref{sec:statistics}).

\subsection{The proximity function and its implementation}

Using the defects of \Cref{eq:rootfinding}, which measure jumps at the time interfaces, we define, for $w \ge 1$,
\begin{equation}
  \label{eq:proximity_function}
  \psi(U) = \frac{1}{N}\sum_{n = 1}^N w^{-n}\|d_n(U)\|.
\end{equation}
Here $N\psi$ is the weighted $\ell^1$ residual norm. With $W_n = w^{-n}I_d$, the trajectory norm is
\begin{equation}
  \label{eq:weighted_norm}
  \|U - V\|_{W, p} = \left(\sum_{n = 1}^N\bigl(w^{-n}\|U_n - V_n\|\bigr)^p\right)^{1/p},
  \qquad w \ge 1,
\end{equation}
with the maximum for $p = \infty$. The choice $w = 1$ recovers the plain norm; fixed $w$ permits the budget analysis of \Cref{sec:framework}. Related defect objectives occur in optimization-based time decomposition~\cite{maday_parareal_2002,rao_adjointbased_2014,chen_adjoint_2015}. \Cref{alg:parareal_psi} reuses fine batches for defects and corrections and reports termination status.

\begin{algorithm}[tbp]
  \caption{Parareal on stored chunk starts with weighted-defect acceptance and explicit status.}
  \label{alg:parareal_psi}
  \begin{algorithmic}[1]
    \REQUIRE $u_0, N, \mathcal{F}, \mathcal{G}, \epsilon$, sweep limit $s_b \le N - 1$, and fixed or adaptive $w$
    \STATE Initialize chunk starts $U_0^0, \ldots, U_{N - 1}^0$ by the coarse pass; cache its outputs. Set $s = 0$, $K_F = 0$.
    \LOOP
      \STATE If a required state or propagation is nonfinite, return status F with the current counts.
      \STATE Evaluate $F_n = \mathcal{F}(U_{n - 1}^s)$ in parallel, reusing the exact prefix. Increment $K_F$; return F if an output is nonfinite.
      \STATE Complete the returned vector with $U_N^s = F_N$.
      \STATE Update adaptive $w$ by \Cref{alg:secant} when previous-batch data exist; evaluate the interior defects in $\psi$.
      \STATE Return the fine trajectories and counts with status E if $s = N - 1$; otherwise C if $\psi(U^s) \le \epsilon$; otherwise B if $s = s_b$.
      \STATE Save the old starts and fine outputs. Set $U_0^{s + 1} = u_0$ and assign $U_n^{s + 1} = F_n$ for $1 \le n \le s + 1$.
      \STATE For $s + 1 < n < N$, compute $G_n^{\mathrm{new}} = \mathcal{G}(U_{n - 1}^{s + 1})$ and set $U_n^{s + 1} = F_n + (G_n^{\mathrm{new}} - G_n)$; update the coarse cache.
      \STATE Set $s \gets s + 1$.
    \ENDLOOP
  \end{algorithmic}
\end{algorithm}

Returned fine trajectories start at the stored interfaces, with $U_N = \mathcal{F}(U_{N - 1})$ giving $d_N = 0$; normalization remains $1/N$. Direct prefix assignment and the parenthesized coarse difference preserve cancellation. Finite returns have $K_F = s + 1$; exact completion takes $N$ batches and $N - 1$ corrections. Standard stopping tests every stored interface:
\begin{equation}
  \label{eq:relative_stop}
  \|U_n^s - U_n^{s - 1}\|_2
  \le \epsilon_a + \epsilon_r\max\{\|U_n^s\|_2, \|U_n^{s - 1}\|_2\},
  \qquad 1 \le n < N,
  \quad s \ge 1.
\end{equation}
We use $\epsilon_r = \epsilon$, $\epsilon_a = 10^{-3}\epsilon$, and first test after one correction; the absolute term safeguards states near zero.

\subsubsection{Lyapunov weighting}
\label{sec:lyapunov_weighting}
The growth scale $e^{\lambda t}$ from the maximal Lyapunov exponent $\lambda$ suggests $w = \max\{1, e^{\lambda\Delta T}\}$. For $\lambda > 0$, $w^{-n} = e^{-\lambda T_n}$ discounts this growth; the floor prevents amplification of later defects when $\lambda < 0$. An asymptotic exponent gives a scale, not a uniform amplification bound, and nonpositive exponents allow transient growth.

A local alternative uses the maximal finite-time exponent. For the numerical chunk map, in the Euclidean norm, define
\begin{equation}
  \label{eq:ftle_chunk}
  \lambda_{\Delta T}^{\mathcal{F}}(U)
  = \frac{1}{\Delta T}\log\sigma_{\max}(D\mathcal{F}(U)).
\end{equation}
Here $\log 0 = -\infty$. Differentiate the chunk map, or approximate the exact-flow counterpart by integrating $\dot{M} = Df(u(t))M$, $M(0) = I_d$, and computing $\sigma_{\max}(M(\Delta T))$. This adds $d$ tangent directions per chunk; one vector gives only a directional estimate.

\subsubsection{Lipschitz weighting and its secant estimate}
\label{sec:lipschitz_weighting}
A uniform fine-map Lipschitz bound on the required neighbourhood gives
\begin{equation}
  \label{eq:lipschitz_weight}
  \exp\bigl(\lambda_{\Delta T}^{\mathcal{F}}(U)\,\Delta T\bigr)
  = \|D\mathcal{F}(U)\|_2 \le \Lambda_{\mathcal{F}},
  \qquad w = \max\{1, \Lambda_{\mathcal{F}}\}.
\end{equation}
This worst-case choice gives $\theta = \Lambda_{\mathcal{F}}/w \le 1$, with equality when $\Lambda_{\mathcal{F}} \ge 1$.

\Cref{alg:secant} reuses fine outputs, guarding against unresolved denominators, including exact-prefix pairs. Each exact-arithmetic secant is at most a Lipschitz bound covering its pair, so sampling cannot certify an upper bound.

\begin{algorithm}[!hbp]
  \caption{Guarded secant weight, initialized at $w = 1$.}
  \label{alg:secant}
  \begin{algorithmic}[1]
    \REQUIRE consecutive interfaces $U^s, U^{s - 1}$, their fine outputs, previous $w$, $\tau_{\mathrm{sec}} = 10^{-10}$
    \STATE For each available same-chunk pair $n$, set $q_n = \|U_n^s - U_n^{s - 1}\|$ and $b_n = \max\{1, \|U_n^s\|, \|U_n^{s - 1}\|\}$.
    \STATE Retain the indices $\mathcal{I} = \{n : q_n > \tau_{\mathrm{sec}}b_n\}$.
    \STATE For $n \in \mathcal{I}$, compute $c_n = \|\mathcal{F}(U_n^s) - \mathcal{F}(U_n^{s - 1})\|/q_n$; discard indices with nonfinite $c_n$.
    \STATE Return $\max\{1, \max_{n \in \mathcal{I}}c_n\}$ if $\mathcal{I}$ is nonempty; otherwise return the previous $w$.
  \end{algorithmic}
\end{algorithm}

Adaptive weights change the norm during a run, whereas budget prediction uses a fixed norm. Increasing $w$ relaxes the acceptance condition at fixed defects.

\subsection{Contraction in the weighted norm}
\label{sec:weighted_bound}

Applied to the derivative bound, the same weights divide each coarse transport factor by $w$.
\begin{corollary}[Weighted Parareal derivative bound]
  \label{cor:weighted_beta}
  Under the hypotheses of \Cref{thm:parareal_rate_beta}, for any fixed $w \ge 1$ and $1 \le p \le \infty$,
  \begin{equation}
    \label{eq:beta_weighted}
    \sup_{U \in \Omega}\|D\Phi(U)\|_{W, p} \le
    \beta_W := \frac{S}{w}\sum_{i = 0}^{N - 2}\left(\frac{\Lambda_{\mathcal{G}}}{w}\right)^i.
  \end{equation}
  On a closed ball in this norm, $\beta_W < 1$ gives the hypothesis of \Cref{thm:ostrovski}.
\end{corollary}
\begin{proof}
  Put $W = \blockdiag(w^{-1}I_d, \ldots, w^{-N}I_d)$. Conjugating multiplies block $(n, m)$ in \Cref{eq:block_jacobian} by $w^{-(n - m)}$. The proof of \Cref{thm:parareal_rate_beta} then applies with source $S/w$ and transport $\Lambda_{\mathcal{G}}/w$. With the product bounds of \Cref{rem:transport}, the sharper estimate is $(S/w)\sum_{i = 0}^{N - 2}\Gamma_i/w^i$.
\end{proof}

For uniform bounds covering the weighted neighbourhood, $w > \Lambda_{\mathcal{G}}$ gives $\beta_W \le S/(w - \Lambda_{\mathcal{G}})$ independently of $N$; at equality, $\beta_W = (N - 1)S/w$.

\subsection{Equivalence of proximity and weighted error}
\label{sec:weighted_error}

The following bounds relate $\psi$, computed from current interfaces and fine outputs, to the weighted error without requiring knowledge of $U^*$.
\Needspace{4\baselineskip}
\begin{theorem}[Equivalence of proximity and weighted error]
  \label{thm:equivalence}
  Under (A2), let $U, U^* \in \mathcal{X}$, with the segments joining $U_{n - 1}$ and $U_{n - 1}^*$ contained in $\mathcal{N}$, and suppose $\|D\mathcal{F}\| \le \Lambda_{\mathcal{F}}$ on those segments. Put $\theta = \Lambda_{\mathcal{F}}/w$. For $1 \le p < \infty$, define
  \begin{equation}
    \label{eq:HNp}
    H_{N, p}(\theta) = \left(\sum_{j = 0}^{N - 1}\theta^{pj}\right)^{1/p},
    \qquad H_{N, \infty}(\theta) = \max_{0 \le j < N}\theta^j.
  \end{equation}
  For $1 \le p \le \infty$, with $N^{1/\infty} = 1$, the bounds are
  \begin{equation}
    \label{eq:equivalence_error_norm}
    \frac{N^{1/p}}{1 + \theta}\psi(U) \le \|U - U^*\|_{W, p}
    \le C_N\psi(U),
    \qquad C_N := N H_{N, p}(\theta).
  \end{equation}
  For comparison, the maximum weighted defect $\psi_\infty$ is defined and bounded by
  \begin{equation}
    \label{eq:max_error}
    \begin{gathered}
      \psi_\infty(U) := \max_{1 \le n \le N}w^{-n}\|d_n(U)\|, \\
      \frac{\psi_\infty(U)}{1 + \theta}
      \le \|U - U^*\|_{W, \infty}
      \le \left(\sum_{j = 0}^{N - 1}\theta^j\right)\psi_\infty(U).
    \end{gathered}
  \end{equation}
\end{theorem}
\begin{proof}
  Write $b_n = w^{-n}\|U_n - U_n^*\|$, $z_n = w^{-n}\|d_n(U)\|$, and $b_0 = 0$. The Lipschitz bound gives $z_n \le b_n + \theta b_{n - 1}$ and $b_n \le z_n + \theta b_{n - 1}$. Summation and H\"older's inequality yield $N\psi \le (1 + \theta)N^{1 - 1/p}\|b\|_p$. Iterating the second inequality gives
  \begin{equation}
    \label{eq:sharp_link}
    w^{-n}\|U_n - U_n^*\| \le \sum_{m = 1}^n\theta^{n - m}w^{-m}\|d_m(U)\|.
  \end{equation}
  Each unit-defect response has $\ell^p$ norm at most $H_{N, p}(\theta)$, so the triangle inequality gives $\|b\|_p \le H_{N, p}\sum_m z_m = C_N\psi$. Taking maxima in the first inequality and in \Cref{eq:sharp_link} proves \Cref{eq:max_error}.
\end{proof}
Averaging contributes the factor $N$. Acceptance at $\epsilon \le r/C_N$ guarantees error at most $r$; $r = \epsilon$ does not make the tests equivalent.

\begin{remark}[Weights are paid for by tolerance]
  \label[remark]{rem:w_not_free}
  If $w \ge 1$ and $\psi(U) \le \epsilon$, each nonnegative summand in $N\psi$ is at most $N\epsilon$, so
  \begin{equation}
    \label{eq:defect_bound}
    \|d_n(U)\| \le N\epsilon w^n \le N\epsilon w^N.
  \end{equation}
  For $\psi_\infty(U) \le \epsilon$, the factor $N$ is absent. A larger weight therefore requires a smaller tolerance to retain the same unweighted-defect guarantee. At $N = T = 128$, $w = e^{0.9056\Delta T}$, and $\epsilon = 10^{-9}$, the averaged bound is about $2.8 \times 10^{43}$ and provides no useful control of the tail. These bounds concern defects; statistical fidelity needs a separate assessment.
\end{remark}

Under the hypotheses of \Cref{thm:equivalence}, with $\Lambda_{\mathcal{F}} \le w$, \Cref{eq:sharp_link} gives $\|e_n\| \le N\epsilon w^n$ on acceptance. Thus $N\epsilon w^n \le \eta$ suffices for interface error at most $\eta$. Nominal Lyapunov exponents and empirical secants do not verify this hypothesis: the measured prefixes below are observations, not certified Lyapunov horizons.

\subsection{Scalability regimes}
\label{sec:scalability_regimes}
\label{sec:bracketing}

The conversion constant $C_N = NH_{N, p}(\theta)$ grows exponentially with $N$ for $\theta > 1$, weakening fixed-threshold guarantees. At $\theta = 1$, $C_N = N^{1 + 1/p}$ ($N^{3/2}$ for $p = 2$). For $0 \le \theta < 1$, $C_N \le N(1 - \theta^p)^{-1/p}$ when $p < \infty$, and $C_N = N$ when $p = \infty$. Averaging still permits a single large defect.
For the maximum variant, the error-conversion constant in \Cref{eq:max_error} is at most $(1 - \theta)^{-1}$ for fixed $\theta < 1$, independently of $N$, and equals $N$ at $\theta = 1$. Thus $w \ge \Lambda_{\mathcal{F}}$ removes exponential growth from the conversion constants, without establishing a bounded sweep count or speedup.

These regimes concern transport bounds: nonchaotic dynamics need not be contractive, and chaotic dynamics need not amplify perturbations on every chunk. Worst-case bounds can overstate accumulated transient growth.

A strictly contractive heat propagator gives $\theta < 1$ with $w = 1$; a unitary advection propagator gives $\theta = 1$. These illustrate error conversion, not sweep counts (\Cref{SM-tab:bracketing}). Zero defects with $U_0 = u_0$ force $U = U^*$, but small weighted defects can hide tail phase error. Temporal diagnostics assess oscillations; autocorrelation alone cannot detect a constant phase shift.

\subsection{Estimating the inputs and predicting a run}
\label{sec:certify}

All coarse-pass defects $d_n^0 = U_n^0 - \mathcal{F}(U_{n - 1}^0)$ are available after the first fine batch. \Cref{thm:equivalence} gives the budget formula's initial-error bound
\begin{equation}
  \label{eq:onelook}
  E_0 = C_N\psi(U^0),
\end{equation}
with $C_N$ as above. Transporting defects separately sharpens this bound. Sampled derivatives give predictions $\hat{e}_0,\hat{\beta}_W$; certification requires uniform enclosures on a neighbourhood containing the initial trajectory.

From $e_n^0 = d_n^0 + \mathcal{F}(U_{n - 1}^0) - \mathcal{F}(U_{n - 1}^*)$, a fine-map Lipschitz bound gives
\begin{equation}
  \label{eq:transported_initial}
  a_0 = 0,
  \qquad a_n = \frac{\Lambda_{\mathcal{F}}}{w}\, a_{n - 1} + w^{-n}\|d_n^0\|,
  \qquad \|e^0\|_{W, p} \le E_{\mathrm{tr}} := \|(a_n)_{n = 1}^N\|_{\ell^p}.
\end{equation}
Thus $\|e^0\|_{W, p} \le E_{\mathrm{tr}} \le C_N\psi(U^0)$.

Coarse-pass sampling gives $\hat{S}_{\mathrm{map}}$, $\hat{\Lambda}_{\mathcal{F}}$, and $\hat{\Lambda}_{\mathcal{G}}$. Fixing $w = \max\{1, \hat{\Lambda}_{\mathcal{F}}, \hat{\Lambda}_{\mathcal{G}}\}$, we use \emph{all} defects in \Cref{eq:transported_initial}; only derivative inputs are sampled and inflated. Frozen-model estimates use coarse nodes:
\begin{equation}
  \label{eq:mu_nu_estimates}
  \hat{\mu} = \max_l\|Df(u_l^0)\|,
  \qquad \hat{\nu} = \max_l\frac{\|Df(u_{l + 1}^0) - Df(u_l^0)\|}{H}
\end{equation}
giving $\hat{S}_{\mathrm{fr}} = L\xi(h\hat{\mu})^2(C_1 + C_2)$. Spacing $H$ samples drift that whole-chunk quotients can conceal. Replacing only $\hat{S}_{\mathrm{map}}$ in \Cref{eq:beta_weighted} gives $\hat{\beta}_{\mathrm{fr}}$ and $s_{\mathrm{fr}}$, retaining sampled transport, $w$, and $\hat{e}_0$. A coarse-node-only predictor also needs method-specific transport estimates (RK4 tests in \Cref{SM-sec:frozen_prediction}). Sampling and an unknown frozen remainder supply no uniform enclosure.

\begin{algorithm}[tbp]
  \caption{Predict a correction-sweep budget in a fixed weighted norm.}
  \label{alg:certify}
  \begin{algorithmic}[1]
    \REQUIRE $r$, desired budget $s_b$, $N$, $\Delta T$, $h$, $H = \xi h$, derivative sample set $\mathcal{I}$
    \STATE Set $L = \Delta T/H$; compute the coarse pass $U^0$ and first fine batch.
    \STATE Form every $d_n^0 = U_n^0 - \mathcal{F}(U_{n - 1}^0)$; in the stored-start implementation reconstruct $U_N^0$ from the final fine output, so $d_N^0 = 0$.
    \STATE Sample actual chunk derivatives on $\mathcal{I}$ to obtain $\hat{S}_{\mathrm{map}}, \hat{\Lambda}_{\mathcal{F}}, \hat{\Lambda}_{\mathcal{G}}$.
    \STATE Set $w = \max\{1, \hat{\Lambda}_{\mathcal{F}}, \hat{\Lambda}_{\mathcal{G}}\}$; evaluate $\hat{\beta}_W$ by \Cref{eq:beta_weighted} and $\hat{e}_0$ by \Cref{eq:transported_initial} using all defects.
    \STATE \rev{Optionally estimate $\hat{\mu}, \hat{\nu}$ by \Cref{eq:mu_nu_estimates} and compare $\hat{S}_{\mathrm{fr}}$, $\hat{\beta}_{\mathrm{fr}}$, and $s_{\mathrm{fr}}$.}
    \IF{$\hat{e}_0 \le r$}
      \STATE Return $s_{\mathrm{pred}} = 0$.
    \ELSIF{$\hat{\beta}_W < 1$}
      \STATE Compute $s_{\mathrm{pred}}(\hat{e}_0)$ by \Cref{eq:budget_box}, including the $\beta = 0$ case.
      \IF{$s_{\mathrm{pred}} \le s_b$}
        \STATE Return $s_{\mathrm{pred}}$ and the sampled inputs.
      \ENDIF
    \ENDIF
    \STATE To shorten the window, reduce the integer $L \ge 1$ at fixed $h, H, N$, or halve $h, H$ at fixed $L, \xi, N$; update $\Delta T = LH$.
    \STATE Alternatively, to keep $\Delta T$ fixed, halve $h, H$ and double $L$.
    \STATE For the chosen change, recompute the coarse pass, fine batch, and derivative inputs; repeat, or report that the requested budget was not predicted.
  \end{algorithmic}
\end{algorithm}

Reducing $L$ preserves alignment but may reach $L = 1$ without contraction. \Cref{cor:short_windows} instead reduces $h, H$ at fixed $L, \xi, N$; fixed-chunk refinement uses \Cref{eq:actual_source} with more steps. Sampled estimates need not decrease monotonically. Exact completion gives the fallback if propagation remains defined.

The work model excludes derivative sampling, diagnostics, and unsuccessful design trials, including their recomputation. Our experiments test prescribed windows, not adaptive moving-window sequences.

\section{Experimental results}
\label{sec:experiments}

Scaling studies, controls, and independent functional tests assess stopping on the logistic, Lorenz-63, and Lorenz-96 systems; \Cref{sec:prediction} tests sweep-budget estimates.

\subsection{Setup and metrics}

The principal grid uses powers of two through $N = 128$: strong scaling fixes $T = 128$, weak scaling fixes $\Delta T = 1$, and both share $N = T = 128$.

Propagation uses NSDERungeKutta steps and NSDETimeParallel, with identical parallel and serial chunk schedules: full steps plus any shortened final step. The scripts, source snapshots, numerical records, tables, and figures are available at \url{https://github.com/giancarloantonucci/finite-budget-parareal}. The tests use Float64 arithmetic, with tolerance $10^{-9}$ for chaotic and periodic problems and $10^{-12}$ for the logistic equation. The logistic problem is $u' = u(1 - u)$, $u_0 = 0.1$, with native BackwardEuler and $h = 10^{-3}$. Its stage solves use residual tolerances $10^{-15}$ (absolute) and $10^{-13}$ (relative), with at most 30 Newton updates. The Lorenz-63 equations are $x' = \sigma(y - x)$, $y' = x(\rho - z) - y$, and $z' = xy - bz$, with $(\sigma, \rho, b) = (10, 28, 8/3)$ and $u_0 = (2, 3, -14)$. The Lorenz-96 system satisfies $x_i' = (x_{i + 1} - x_{i - 2})x_{i - 1} - x_i + 8$, $i = 1, \ldots, 40$, with cyclic indexing $x_{i + 40} = x_i$ and initial components 8 except component 20, which is 8.01.

Both Lorenz systems use RK4 fine and coarse propagators with $h = 10^{-4}$; Lorenz-63 uses $\xi = 10, 100$ and Lorenz-96 $\xi = 100$. Prescribed Lyapunov exponents 0.9056 and 1.68 set the weights; they are neither computed here nor verified amplification bounds. \Cref{SM-sec:reproducibility} records the software.

We compare \emph{standard} stopping (\Cref{eq:relative_stop}) with three $\psi$ rules: \emph{unweighted} ($w = 1$), \emph{Lyapunov} ($w = \max\{1, e^{\lambda\Delta T}\}$), and \emph{secant} (\Cref{alg:secant}). Plots and tables use weight labels.

We report statuses C (criterion), E (exact completion), B (budget), and F (nonfinite failure). For equal fine and coarse step costs, work and speedup with exact-prefix reuse are
\begin{equation}
  \label{eq:cost}
  \mathcal{W} = K_F + K_G/\xi,
  \qquad K_G = (s + 1)(N - 1) - \frac{s(s + 1)}{2},
  \qquad \mathcal{S} = N/\mathcal{W}.
\end{equation}
Here $K_G$ includes $N - 1$ initialization calls; correction $k$ uses $N - k - 1$ calls. A bounded correction count $s$ does not imply bounded work at fixed $\xi$. Unequal step costs require a relative-cost factor; speedups also exclude communication and hardware overheads.

\paragraph{Statistical comparisons}
\label{sec:statistics}
Weighted acceptance gives no statistical guarantee (\Cref{rem:w_not_free}); we assess tracking, finite-horizon marginals, and long-reference functional estimates separately.

The statistical tests use $N = 16, 32, 64, 128$, $\Delta T = 1$, and $\xi = 100$. For each system we prepare eight independent pairs of starts, each a uniform perturbation of at most 0.25 per coordinate of the prescribed start, after a 40-time-unit fine spin-up. Coarse-only and accepted trajectories share the first fine start; the second supplies the sampling baseline. Sampling intervals are 0.02 for Lorenz-63 and 0.05 for Lorenz-96. For empirical measures $\widehat{\pi}, \widehat{\pi}'$ with coordinatewise cumulative distribution functions $\widehat{F}_i, \widehat{G}_i$, the reported distance is
\begin{equation}
  \label{eq:marginal_wasserstein}
  \overline{W}_1(\widehat{\pi}, \widehat{\pi}') = \frac{1}{d}\sum_{i = 1}^d
  \int_{\mathbb{R}}|\widehat{F}_i(x) - \widehat{G}_i(x)|\, dx.
\end{equation}
This marginal distance does not require histogram binning and has the scale of the state variables; fine and coarse-only baselines calibrate it. Sliced, moment, and temporal diagnostics probe further discrepancies (\Cref{SM-sec:statistics_full}). Paired accepted-minus-coarse intervals use 2000 bootstrap resamples of whole replicate differences, preserving within-record temporal dependence.

Chaotic scaling tables use these sampling intervals on single prescribed-start runs without spin-up, sampling within fine chunks without interpolation across jumps. These grid-dependent distances establish neither trajectory nor invariant-measure accuracy; interface errors assess tracking and replicates assess uncertainty.

\subsection{The logistic equation}
\label{sec:logistic}

At $\Delta T = 1$, $\xi = 100$, and $N = 16, 128$, unweighted $\psi$ takes six sweeps against seven for standard stopping, with errors below $2 \times 10^{-12}$ (\Cref{tab:scaling_logistic_main}). Late chunks near equilibrium are already accurate; bounded counts do not establish chaotic scalability.

\begin{table}[!htbp]
  \centering
  \small
  \caption{Logistic equation, standard and unweighted checks, selected $N = 2, 16, 128$; the remaining rows are in \Cref{SM-tab:scaling_logistic_extra}. The grids share the $N = T = 128$ run. $\mathcal{S}$ follows \Cref{eq:cost}, $e = U^s - U^*$ is unweighted, and $s = K_F - 1$; statuses follow the setup.}
  \label{tab:scaling_logistic_main}
\setlength{\tabcolsep}{1.5pt}\renewcommand{\arraystretch}{1.15}
\begin{tabular*}{\textwidth}{@{\extracolsep{\fill}}ccccccc@{\hspace{8pt}}cccc@{}}\toprule
& & & \multicolumn{4}{c}{Strong scaling ($T=128$)} & \multicolumn{4}{c}{Weak scaling ($\Delta T=1$)}\\
\cmidrule(lr){4-7}\cmidrule(lr){8-11}
$\xi$ & Rule & $N$ & $K_F$ & Stop & $\mathcal S$ & $\|e\|_2$ & $K_F$ & Stop & $\mathcal S$ & $\|e\|_2$\\
\midrule
\multirow{6}{*}{$10$} & \multirow{3}{*}{standard} & $2$ & $2$ & E & $0.952$ & $0$ & $2$ & E & $0.952$ & $0$\\
 &  & $16$ & $4$ & C & $1.70$ & $1.50\times 10^{-13}$ & $6$ & C & $1.19$ & $2.27\times 10^{-15}$\\
 &  & $128$ & $6$ & C & $1.59$ & $4.66\times 10^{-13}$ & $6$ & C & $1.59$ & $4.66\times 10^{-13}$\\
\cmidrule(lr){2-11}
 & \multirow{3}{*}{$\psi$ ($w=1$)} & $2$ & $1$ & C & $1.82$ & $7.07\times 10^{-14}$ & $2$ & E & $0.952$ & $0$\\
 &  & $16$ & $3$ & C & $2.22$ & $1.58\times 10^{-13}$ & $4$ & C & $1.70$ & $6.18\times 10^{-12}$\\
 &  & $128$ & $4$ & C & $2.36$ & $6.20\times 10^{-12}$ & $4$ & C & $2.36$ & $6.20\times 10^{-12}$\\
\midrule
\multirow{6}{*}{$100$} & \multirow{3}{*}{standard} & $2$ & $2$ & E & $0.995$ & $0$ & $2$ & E & $0.995$ & $0$\\
 &  & $16$ & $5$ & C & $2.83$ & $1.75\times 10^{-13}$ & $8$ & C & $1.79$ & $3.19\times 10^{-14}$\\
 &  & $128$ & $8$ & C & $7.16$ & $5.12\times 10^{-13}$ & $8$ & C & $7.16$ & $5.12\times 10^{-13}$\\
\cmidrule(lr){2-11}
 & \multirow{3}{*}{$\psi$ ($w=1$)} & $2$ & $1$ & C & $1.98$ & $7.71\times 10^{-14}$ & $2$ & E & $0.995$ & $0$\\
 &  & $16$ & $3$ & C & $4.68$ & $1.19\times 10^{-11}$ & $7$ & C & $2.04$ & $1.51\times 10^{-12}$\\
 &  & $128$ & $7$ & C & $8.16$ & $1.78\times 10^{-12}$ & $7$ & C & $8.16$ & $1.78\times 10^{-12}$\\
\bottomrule\end{tabular*}
\end{table}

\subsection{Lorenz-63: early acceptance and its accuracy}
\label{sec:lorenz}

At $\xi = 100$, $N = T = 128$, standard stopping takes 113 sweeps against two for Lyapunov and one for secant weighting, yet coarse-only has the smallest marginal distance (\Cref{tab:scaling_lorenz_main,fig:KvsTLorenz}). Both rules improve on coarse-only at $\Delta T = 1$, $N = 16$. The Lyapunov rule can accept without correction because long-chunk weights suppress defects, while division by $N$ weakens the averaged criterion on shorter chunks once early defects dominate (\Cref{SM-tab:scaling_lorenz,rem:w_not_free}). Such fine trajectories from coarse-pass starts differ from coarse-only.

\begin{table}[!htbp]
  \centering
  \small
  \caption{Lorenz-63 at $\xi = 100$, selected $N = 2, 16, 128$; the remaining rows are in \Cref{SM-tab:scaling_lorenz_extra}. Work columns follow \Cref{tab:scaling_logistic_main}. $\overline{W}_1$ measures dense mean marginal Wasserstein distance to the fine reference; coarse-only rows use the coarse trajectory.}
  \label{tab:scaling_lorenz_main}
\setlength{\tabcolsep}{1.5pt}\renewcommand{\arraystretch}{1.15}
\begin{tabular*}{\textwidth}{@{\extracolsep{\fill}}cccccc@{\hspace{8pt}}cccc@{}}\toprule
& & \multicolumn{4}{c}{Strong scaling ($T=128$)} & \multicolumn{4}{c}{Weak scaling ($\Delta T=1$)}\\
\cmidrule(lr){3-6}\cmidrule(lr){7-10}
$N$ & Rule & $K_F$ & Stop & $\mathcal S$ & $\overline{W}_1$ & $K_F$ & Stop & $\mathcal S$ & $\overline{W}_1$\\
\midrule
\multirow{5}{*}{$2$} & standard & $2$ & E & $0.995$ & $0$ & $2$ & E & $0.995$ & $0$\\*
 & $\psi$ ($w=1$) & $2$ & E & $0.995$ & $0$ & $2$ & E & $0.995$ & $0$\\*
 & $\psi$ ($w=e^{\lambda\Delta T}$) & $1$ & C & $1.98$ & $3.71\times 10^{-1}$ & $2$ & E & $0.995$ & $0$\\*
 & $\psi$ (secant $w$) & $2$ & E & $0.995$ & $0$ & $2$ & E & $0.995$ & $0$\\*
 & coarse-only &  &  &  & $1.62\times 10^{-1}$ &  &  &  & $4.97\times 10^{-4}$\\
\addlinespace[3pt]
\multirow{5}{*}{$16$} & standard & $16$ & E & $0.930$ & $0$ & $5$ & C & $2.83$ & $1.31\times 10^{-10}$\\*
 & $\psi$ ($w=1$) & $16$ & E & $0.930$ & $0$ & $4$ & C & $3.52$ & $1.30\times 10^{-9}$\\*
 & $\psi$ ($w=e^{\lambda\Delta T}$) & $2$ & C & $6.99$ & $3.49\times 10^{-1}$ & $3$ & C & $4.68$ & $1.37\times 10^{-5}$\\*
 & $\psi$ (secant $w$) & $3$ & C & $4.68$ & $1.14$ & $2$ & C & $6.99$ & $3.12\times 10^{-2}$\\*
 & coarse-only &  &  &  & $1.62\times 10^{-1}$ &  &  &  & $1.36$\\
\addlinespace[3pt]
\multirow{5}{*}{$128$} & standard & $114$ & C & $0.659$ & $9.87\times 10^{-11}$ & $114$ & C & $0.659$ & $9.87\times 10^{-11}$\\*
 & $\psi$ ($w=1$) & $102$ & C & $0.711$ & $1.73\times 10^{-5}$ & $102$ & C & $0.711$ & $1.73\times 10^{-5}$\\*
 & $\psi$ ($w=e^{\lambda\Delta T}$) & $3$ & C & $18.9$ & $5.74\times 10^{-1}$ & $3$ & C & $18.9$ & $5.74\times 10^{-1}$\\*
 & $\psi$ (secant $w$) & $2$ & C & $28.3$ & $3.37\times 10^{-1}$ & $2$ & C & $28.3$ & $3.37\times 10^{-1}$\\*
 & coarse-only &  &  &  & $1.62\times 10^{-1}$ &  &  &  & $1.62\times 10^{-1}$\\
\bottomrule\end{tabular*}
\end{table}

\begin{figure}[!htbp]
  \centering
  \includegraphics[width=\textwidth,trim=0 14 0 9,clip]{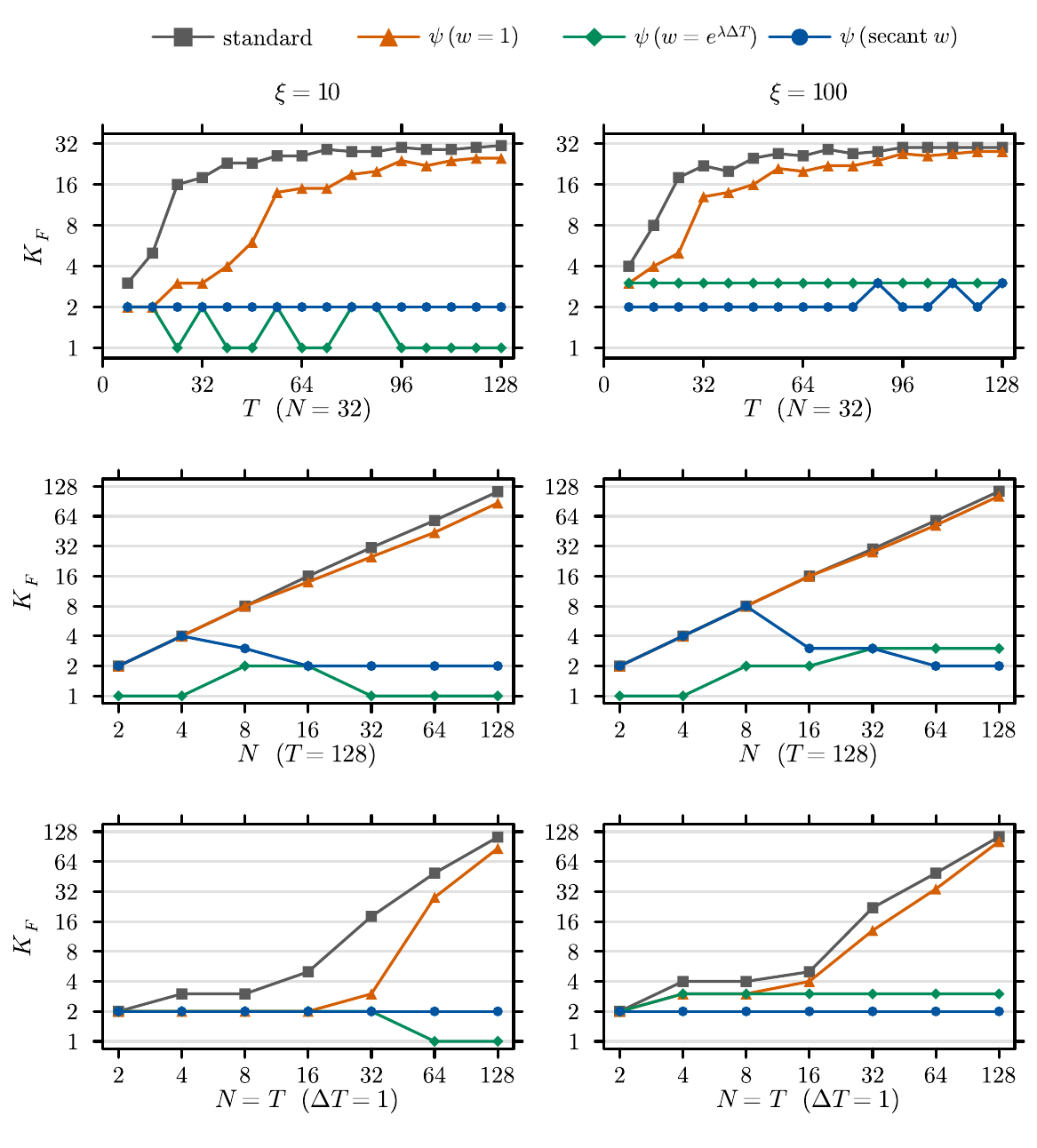}
  \caption{Lorenz-63 fine-batch counts $K_F$: growing horizon, strong scaling, and weak scaling. $K_F$ and, for strong and weak scaling, $N$ are on base-2 logarithmic axes.}
  \label{fig:KvsTLorenz}
\end{figure}

\begin{table}[!htbp]
  \centering
  \small
  \caption{Replicated $\overline{W}_1$ comparison with an accurate coarse solver (RK4, $H = 0.01$). Means are over eight spun-up replicate pairs; negative paired intervals indicate improvement over coarse-only. Coincident entries reflect identical accepted interfaces in every replicate. Full records and the $N = 64$ rows are in \Cref{SM-sec:statistics_full}.}
  \label{tab:w1_baseline}
\setlength{\tabcolsep}{1.5pt}\renewcommand{\arraystretch}{1.15}
\begin{tabular*}{\textwidth}{@{\extracolsep{\fill}}cccccc@{}}\toprule
& & \multicolumn{3}{c}{$\overline{W}_1$ to the fine reference} & 95\% paired CI\\
\cmidrule(lr){3-5}\cmidrule(lr){6-6}
$N$ & Rule & Fine baseline & Coarse-only & Accepted & vs coarse-only\\
\midrule
\multicolumn{6}{c}{\textit{Lorenz-63}}\\*
\multirow{2}{*}{$16$} & $\psi$ ($w=e^{\lambda\Delta T}$) & \multirow{2}{*}{$2.00$} & \multirow{2}{*}{$0.656$} & $2.72\times 10^{-4}$ & [$-1.01$, $-0.350$]\\*
 & $\psi$ (secant $w$) &  &  & $1.37\times 10^{-3}$ & [$-1.01$, $-0.349$]\\
\addlinespace[3pt]
\multirow{2}{*}{$32$} & $\psi$ ($w=e^{\lambda\Delta T}$) & \multirow{2}{*}{$1.11$} & \multirow{2}{*}{$0.802$} & $6.38\times 10^{-1}$ & [$-0.748$, $0.438$]\\*
 & $\psi$ (secant $w$) &  &  & $6.70\times 10^{-1}$ & [$-0.675$, $0.444$]\\
\addlinespace[3pt]
\multirow{2}{*}{$128$} & $\psi$ ($w=e^{\lambda\Delta T}$) & \multirow{2}{*}{$0.675$} & \multirow{2}{*}{$0.551$} & $6.15\times 10^{-1}$ & [$-0.352$, $0.554$]\\*
 & $\psi$ (secant $w$) &  &  & $6.15\times 10^{-1}$ & [$-0.352$, $0.554$]\\
\midrule
\multicolumn{6}{c}{\textit{Lorenz-96}}\\*
\multirow{2}{*}{$16$} & $\psi$ ($w=e^{\lambda\Delta T}$) & \multirow{2}{*}{$0.856$} & \multirow{2}{*}{$0.538$} & $1.79\times 10^{-1}$ & [$-0.425$, $-0.303$]\\*
 & $\psi$ (secant $w$) &  &  & $1.79\times 10^{-1}$ & [$-0.425$, $-0.303$]\\
\addlinespace[3pt]
\multirow{2}{*}{$32$} & $\psi$ ($w=e^{\lambda\Delta T}$) & \multirow{2}{*}{$0.602$} & \multirow{2}{*}{$0.469$} & $4.20\times 10^{-1}$ & [$-0.0985$, $-0.00490$]\\*
 & $\psi$ (secant $w$) &  &  & $4.20\times 10^{-1}$ & [$-0.0985$, $-0.00490$]\\
\addlinespace[3pt]
\multirow{2}{*}{$128$} & $\psi$ ($w=e^{\lambda\Delta T}$) & \multirow{2}{*}{$0.287$} & \multirow{2}{*}{$0.265$} & $2.61\times 10^{-1}$ & [$-0.0312$, $0.0268$]\\*
 & $\psi$ (secant $w$) &  &  & $2.61\times 10^{-1}$ & [$-0.0312$, $0.0268$]\\
\bottomrule\end{tabular*}
\end{table}

Replicated Lorenz-63 comparisons show significantly smaller marginal distances at $N = 16$, alongside accurate tracking (\Cref{tab:w1_baseline}), without establishing better invariant statistics. At $N \ge 32$, differences from coarse-only and independent-fine baselines are unresolved (\Cref{SM-tab:statistics_full}).

\begin{figure}[!htbp]
  \centering
  \includegraphics[width=\textwidth,trim=0 11 0 8,clip]{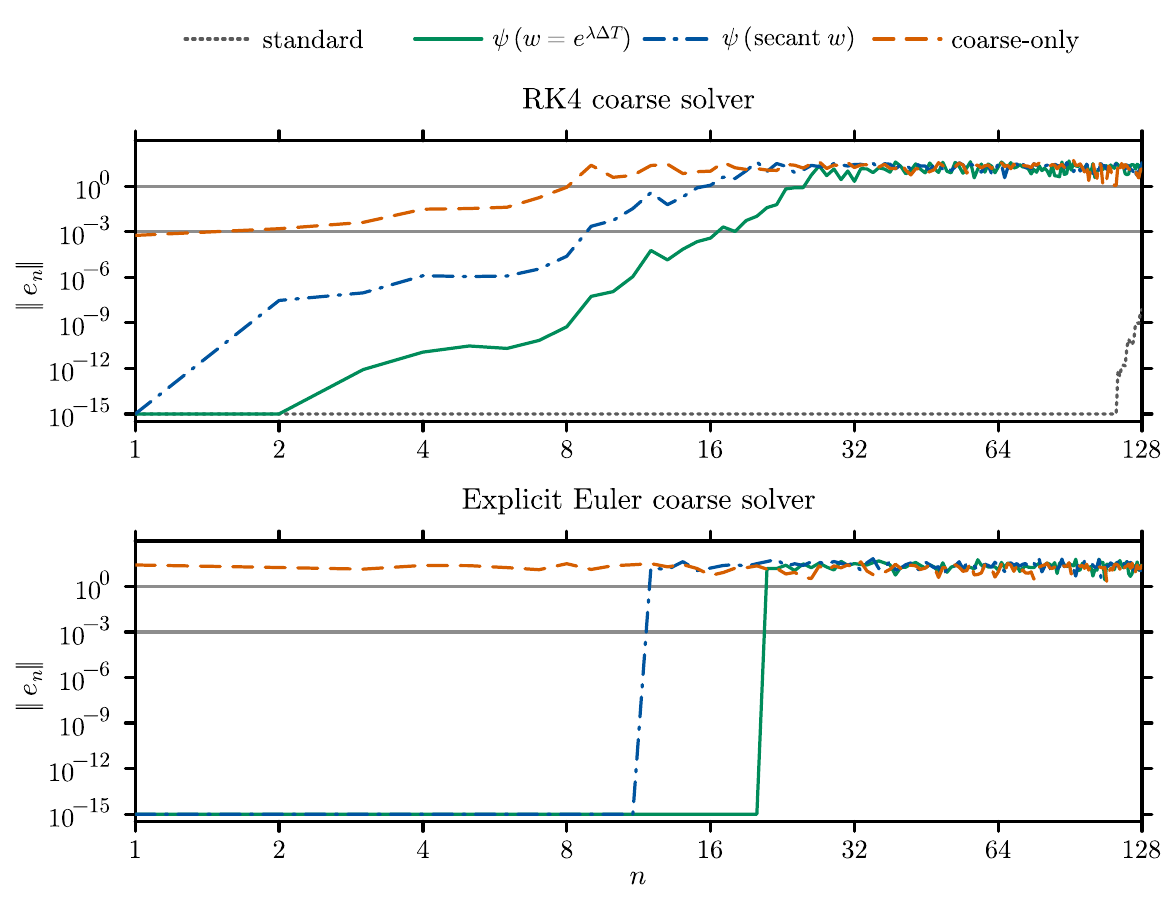}
  \caption{Per-interface error at $N = 128$ for Lorenz-63 start $(2, 3, -14)$, with RK4 (top) and explicit Euler (bottom) coarse propagation, $H = 0.01$. RK4 extends accuracy beyond the exact prefix; Euler gives little extension, and its standard run fails at correction 65. The interface axis is logarithmic; zero errors lie at the floor, and horizontal lines mark thresholds $10^{-3}$ and $1$ (prefix lengths in \Cref{SM-tab:prefix_summary}).}
  \label{fig:prefix}
\end{figure}

At $\Delta T = 1$, $N = 16$, $\xi = 100$, interface errors are $4.48 \times 10^{-4}$ (Lyapunov) and 1.48 (secant): the extra correction improves tracking (\Cref{SM-tab:scaling_short_horizon}). At $N = 128$, Lyapunov stays below $10^{-3}$ through chunk 16 and below 1 through chunk 25, beyond its two exact chunks; coarse-only stays below these thresholds through chunks 1 and 8 (\Cref{fig:prefix}). Corrected runs of the submitted long-chunk grids no longer give the former return to 100 weighted fine batches at $\xi = 100$; expensive secant runs and zero-correction Lyapunov acceptance remain (\Cref{SM-sec:submitted_grid}).

\subsection{Lorenz-96 and the choice of weight}
\label{sec:lorenz96}

At $\Delta T = 1$, $N = 16$, Lorenz-96 secant weighting takes one sweep versus Lyapunov's two, but gives larger marginal and interface errors (\Cref{tab:scaling_lorenz96_main,SM-tab:scaling_short_horizon}). Strong-scaling long-chunk weights permit zero-correction Lyapunov acceptance; at $N = 16$, secant and unweighted rules complete exactly after 15 corrections. At $N = T = 128$, coarse-only has the smallest distance: fewer sweeps do not ensure distributional gains.

Replicated Lorenz-96 comparisons show marginal gains at $N = 16, 32$, but the differences are unresolved at $N = 64, 128$ (\Cref{tab:w1_baseline}). At $N = 32$, neither rule reduces neighbour-product error (0.407, compared with 0.344 for coarse-only; \Cref{SM-tab:observables_full}).

\begin{table}[!htbp]
  \centering
  \small
  \caption{Lorenz-96 at $\xi = 100$, with prescribed $\lambda = 1.68$, selected $N = 2, 16, 128$; the remaining rows are in \Cref{SM-tab:scaling_lorenz96_extra}. Columns and rows follow \Cref{tab:scaling_lorenz_main}.}
  \label{tab:scaling_lorenz96_main}
\setlength{\tabcolsep}{1.5pt}\renewcommand{\arraystretch}{1.15}
\begin{tabular*}{\textwidth}{@{\extracolsep{\fill}}cccccc@{\hspace{8pt}}cccc@{}}\toprule
& & \multicolumn{4}{c}{Strong scaling ($T=128$)} & \multicolumn{4}{c}{Weak scaling ($\Delta T=1$)}\\
\cmidrule(lr){3-6}\cmidrule(lr){7-10}
$N$ & Rule & $K_F$ & Stop & $\mathcal S$ & $\overline{W}_1$ & $K_F$ & Stop & $\mathcal S$ & $\overline{W}_1$\\
\midrule
\multirow{5}{*}{$2$} & standard & $2$ & E & $0.995$ & $0$ & $2$ & E & $0.995$ & $0$\\*
 & $\psi$ ($w=1$) & $2$ & E & $0.995$ & $0$ & $2$ & E & $0.995$ & $0$\\*
 & $\psi$ ($w=e^{\lambda\Delta T}$) & $1$ & C & $1.98$ & $2.12\times 10^{-1}$ & $2$ & E & $0.995$ & $0$\\*
 & $\psi$ (secant $w$) & $2$ & E & $0.995$ & $0$ & $2$ & E & $0.995$ & $0$\\*
 & coarse-only &  &  &  & $2.20\times 10^{-1}$ &  &  &  & $4.51\times 10^{-4}$\\
\addlinespace[3pt]
\multirow{5}{*}{$16$} & standard & $16$ & E & $0.930$ & $0$ & $9$ & C & $1.60$ & $7.63\times 10^{-11}$\\*
 & $\psi$ ($w=1$) & $16$ & E & $0.930$ & $0$ & $5$ & C & $2.83$ & $2.09\times 10^{-8}$\\*
 & $\psi$ ($w=e^{\lambda\Delta T}$) & $2$ & C & $6.99$ & $2.64\times 10^{-1}$ & $3$ & C & $4.68$ & $1.75\times 10^{-2}$\\*
 & $\psi$ (secant $w$) & $16$ & E & $0.930$ & $0$ & $2$ & C & $6.99$ & $4.25\times 10^{-1}$\\*
 & coarse-only &  &  &  & $2.20\times 10^{-1}$ &  &  &  & $6.54\times 10^{-1}$\\
\addlinespace[3pt]
\multirow{5}{*}{$128$} & standard & $119$ & C & $0.640$ & $2.82\times 10^{-12}$ & $119$ & C & $0.640$ & $2.82\times 10^{-12}$\\*
 & $\psi$ ($w=1$) & $112$ & C & $0.666$ & $5.05\times 10^{-6}$ & $112$ & C & $0.666$ & $5.05\times 10^{-6}$\\*
 & $\psi$ ($w=e^{\lambda\Delta T}$) & $3$ & C & $18.9$ & $2.85\times 10^{-1}$ & $3$ & C & $18.9$ & $2.85\times 10^{-1}$\\*
 & $\psi$ (secant $w$) & $2$ & C & $28.3$ & $2.64\times 10^{-1}$ & $2$ & C & $28.3$ & $2.64\times 10^{-1}$\\*
 & coarse-only &  &  &  & $2.20\times 10^{-1}$ &  &  &  & $2.20\times 10^{-1}$\\
\bottomrule\end{tabular*}
\end{table}

\paragraph{Periodic control}
\label{sec:periodic}
The periodic Lorenz-63 control at $\rho = 100.5$~\cite{sparrow_lorenz_1982} oscillates without asymptotic exponential growth. Sampled one-chunk amplification reaches 158; multi-chunk amplification reaches only 14.8 (\Cref{SM-sec:periodic_table}).
The Lyapunov rule uses $w = 1$ since $\lambda = 0$, whereas the secant rule responds to transients, reaching error of order $10^{-9}$ in two batches at $\xi = 10$ but accepting error about 0.48 at $\xi = 100$ on an orbit of amplitude about 100. Multiplying chunk bounds loses alignment information (\Cref{rem:transport}).

\subsection{A deliberately poor coarse solver}
\label{sec:bad_coarse}

With explicit Euler coarse propagation at $H = 0.01$, $N = 16$, the Lyapunov rule reaches exact completion on all 16 chunks; the secant rule takes 12--16 batches and significantly improves on coarse-only (\Cref{tab:bad_coarse}).

\begin{table}[!htbp]
  \centering
  \small
  \caption{Lorenz-63 control with an explicit Euler coarse solver, $H = 0.01$, eight replicate starts per horizon. Fine and coarse-only baseline means use all eight starts. Accepted means use each rule's finite runs; paired intervals use their matched coarse-only counterparts and are conditional on finite completion. Full outcomes, including $H = 0.02$, are in \Cref{SM-sec:poorcoarse_details}.}
  \label{tab:bad_coarse}
\setlength{\tabcolsep}{1.5pt}\renewcommand{\arraystretch}{1.15}
\begin{tabular*}{\textwidth}{@{\extracolsep{\fill}}cccccccc@{}}\toprule
& & & & \multicolumn{3}{c}{$\overline{W}_1$ to the fine reference} & 95\% paired CI\\
\cmidrule(lr){5-7}\cmidrule(lr){8-8}
$N$ & Rule & Outcomes & $K_F$ & Fine baseline & Coarse-only & Accepted & vs coarse-only\\
\midrule
\multirow{2}{*}{$16$} & $\psi$ ($w=e^{\lambda\Delta T}$) & 8E & $16$ & \multirow{2}{*}{$1.53$} & \multirow{2}{*}{$2.78$} & $0$ & [$-3.65$, $-1.85$]\\*
 & $\psi$ (secant $w$) & 2C, 6E & $12$--$16$ &  &  & $0.0806$ & [$-3.54$, $-1.81$]\\
\addlinespace[3pt]
\multirow{2}{*}{$128$} & $\psi$ ($w=e^{\lambda\Delta T}$) & 5C, 3F & $21$--$22$ & \multirow{2}{*}{$0.363$} & \multirow{2}{*}{$1.58$} & $1.26$ & [$-0.620$, $0.0117$]\\*
 & $\psi$ (secant $w$) & 6C, 2F & $10$--$13$ &  &  & $1.39$ & [$-0.410$, $0.0314$]\\
\bottomrule\end{tabular*}
\end{table}

At $N = 128$, three Lyapunov and two secant runs fail. Finite runs take \mbox{21--22} and \mbox{10--13} batches, respectively, without significant marginal gains over matched coarse-only runs (\Cref{tab:bad_coarse}). After about 20 sweeps, roughly a hundred chunks remain without a guarantee of exactness or statistical accuracy (\Cref{rem:w_not_free,eq:tail_identity}). At $H = 0.02$ the corrections become nonfinite.

\subsection{Statistical functionals}
\label{sec:functional_results}

Each independent test uses 32 fresh matched starts after a 40-unit fine spin-up, with fine RK4 step $h = 10^{-4}$. Sampling every 0.01 units gives energy estimates $\widehat{a}_r^M$ of $\langle\|u\|_2^2/d\rangle$, compared with an independent refined long reference $\widehat{a}^\star$:
\begin{equation}
  \label{eq:functional_error}
  E(M) = \frac{1}{32}\sum_{r = 1}^{32}|\widehat{a}_r^M - \widehat{a}^\star|.
\end{equation}
This mean absolute error includes sampling variation and numerical bias. Four reference chains at $h = 5 \times 10^{-5}$ each span 8192 units for Lorenz-63 or 2048 for Lorenz-96; step-halving differences are unresolved at sampling precision. Intervals use 10,000 resamples of whole matched starts and independent 256-unit reference blocks, sharing each reference draw between methods. Secondary intervals are pointwise (full protocols in \Cref{SM-sec:functional_accuracy,SM-sec:stopping_adequacy}).

For Lorenz-63, the accuracy limit is 1.995, set at 1.25 times the mean error of 64 separate fine starts; intervals condition on this fixed estimate. Every correction run failed in the first 32-start test at $H = 0.1$. An amended eight-start screen selected RK4 $H = 0.05$ by requiring exact completion and bitwise fine-interface agreement, without using acceptance or statistical correction outcomes. Configuration and analysis were fixed before testing on 32 fresh starts; failures were not replaced. Success requires one-sided 95\% bounds: the coarse-error lower bound above the limit, the Lyapunov-error upper bound below it, and the upper bound on the ratio of mean modelled elapsed costs to standard stopping below one half. This tests mean error across starts, not individual-trajectory accuracy.

\begin{table}[!htbp]
  \centering
  \small
  \caption{Independent energy tests, 32 matched starts per system, all finite. $E$ is defined in \Cref{eq:functional_error}; $U_{0.95}$ is its one-sided 95\% upper bound. Paired intervals concern method-minus-coarse errors. Modelled elapsed-cost ratio $\mathcal{C} = (K_F + \zeta K_G)/N$ uses the coarse-to-fine chunk cost ratio $\zeta$, assumes enough workers, and excludes communication and diagnostics; serial fine propagation has cost one. No absolute adequacy limit was set for the Lorenz-96 confirmation.}
  \label{tab:functional_main}
\setlength{\tabcolsep}{10pt}\renewcommand{\arraystretch}{1.15}
\begin{tabular}{@{}ccccc@{}}\toprule
Method & $E$ & $U_{0.95}$ & Paired 95\% interval & $\mathcal{C}$\\
\midrule
\multicolumn{5}{c}{\textit{Lorenz-63: $T = N = 32$, $H = 0.05$, $\zeta = 0.002$, limit 1.995}}\\*
coarse-only & $2.815$ & $3.284$ &  & $2.00\times 10^{-3}$\\*
first fine batch & $1.652$ & $2.002$ &  & $3.32\times 10^{-2}$\\*
serial fine & $1.354$ & $1.602$ &  & $1.00$\\*
serial RK4 ($h=0.005$) & $1.450$ & $1.776$ &  & $2.00\times 10^{-2}$\\*
$\psi$ ($w=e^{\lambda\Delta T}$) & $1.516$ & $1.808$ & [$-1.979$, $-0.654$] & $3.49\times 10^{-1}$\\*
$\psi$ (secant $w$) & $2.548$ & $3.325$ & [$-1.190$, $0.722$] & $1.46\times 10^{-1}$\\*
standard & $1.354$ & $1.602$ &  & $7.30\times 10^{-1}$\\
\midrule
\multicolumn{5}{c}{\textit{Lorenz-96: $T = N = 128$, $H = 0.01$, $\zeta = 0.01$}}\\*
coarse-only & $0.153$ &  &  & $1.00\times 10^{-2}$\\*
serial RK4 ($h=0.005$) & $0.181$ &  &  & $2.00\times 10^{-2}$\\*
$\psi$ (either $w$) & $0.200$ &  & [$-0.0133$, $0.105$] & $3.54\times 10^{-2}$\\
\bottomrule\end{tabular}
\end{table}

The independent Lorenz-96 test does not confirm the exploratory energy-error reduction (\Cref{tab:functional_main}). For Lorenz-63, Lyapunov reduces energy error from 2.815 to 1.516 (paired interval $[-1.979, -0.654]$). The prescribed requirement is met: its upper bound 1.808 is below 1.995, the coarse lower bound 2.360 exceeds that limit, and the cost-ratio upper bound 0.496 is 0.004 below one half. Serial fine and standard stopping also meet the limit (upper bound 1.602), as does intermediate-step RK4 (1.776) at cost 0.02 versus Lyapunov's 0.349. Secant's paired interval includes zero.

The first fine batch has error 1.652 and upper bound 2.002, just above the limit, at $\mathcal{C} \approx 0.0332$, about one tenth of the Lyapunov cost. The post hoc paired interval $[-0.693, 0.407]$ establishes neither an additional gain nor equivalence. These statistical tests are not used for stopping.

\subsection{Predicting the sweep budget}
\label{sec:prediction}

With $U_0 = u_0$ fixed and $U_N = \mathcal{F}(U_{N - 1})$ reconstructed, we test the free-start error:
\begin{equation}
  \label{eq:start_norm}
  E_{\mathrm{start}}(U) = \left(\sum_{n = 1}^{N - 1}w^{-2n}\|U_n - U_n^*\|_2^2\right)^{1/2}.
\end{equation}

Restricting the derivative to stored starts deletes its terminal row and column, preserving the bound for $\beta_W$. If a fine-map Lipschitz bound covers $U_{N - 1}$ and $U_{N - 1}^*$, terminal reconstruction gives
\begin{equation}
  \label{eq:start_bound}
  \|U - U^*\|_{W, 2} \le \sqrt{1 + (\Lambda_{\mathcal{F}}/w)^2}\, E_{\mathrm{start}}(U).
\end{equation}
This follows by adding $w^{-2N}\|e_N\|^2 \le (\Lambda_{\mathcal{F}}/w)^2w^{-2(N - 1)}\|e_{N - 1}\|^2$ to $E_{\mathrm{start}}^2$.

\Cref{alg:certify} uses all interior defects and samples derivatives on 16 of 32 coarse-pass chunks, multiplying the sampled norm maxima by 1.25. \Cref{tab:prediction} compares predicted sweeps with first crossings $s_{\mathrm{hit}}$ of $r = 10^{-12}$ (logistic) or $10^{-9}$ (Lorenz-63 regimes) in the fixed error norm (\Cref{eq:start_norm}).

\begin{figure}[!htbp]
  \centering
  \includegraphics[width=\textwidth,trim=0 18 0 11,clip]{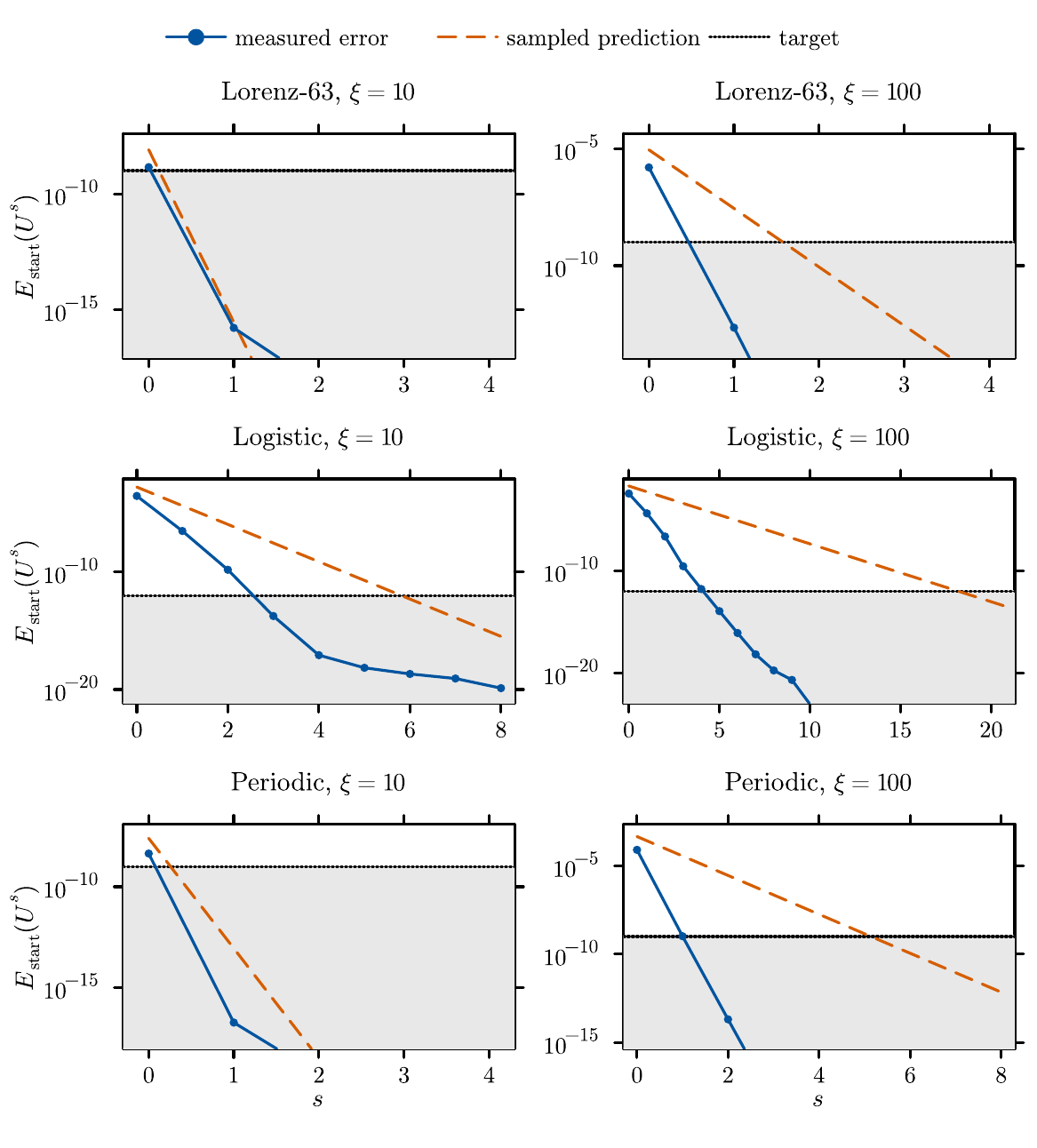}
  \caption{Measured weighted error on stored starts (solid) and sampled prediction $\hat{e}_0\hat{\beta}_W^s$ (dashed) at $N = 32$, for $\xi = 10$ and 100. The dotted line marks the target; each case uses the fixed weight in \Cref{tab:prediction}. Shading marks errors below target; curves may exit the panels, and exact zeros are omitted.}
  \label{fig:measured_vs_bound}
\end{figure}

\begin{table}[!htb]
  \centering
  \small
  \caption{\rev{Sweep-budget prediction at $N = 32$, $\Delta T = 1$, in the fixed norm with displayed weight $w$. The frozen-source comparison retains sampled transports, $w$, and $\hat{e}_0$. Uncapped $s_{\mathrm{fr}} = 309$ is superseded by exact completion at $N - 1$; a dash means no contraction prediction. Weights are rounded. Full inputs and the one-look comparison are in \Cref{SM-tab:prediction_details,SM-tab:prediction_error_link}.}}
  \label{tab:prediction}
\setlength{\tabcolsep}{1.5pt}\renewcommand{\arraystretch}{1.15}
\begin{tabular*}{\textwidth}{@{\extracolsep{\fill}}ccccccccc@{}}\toprule
& & \multicolumn{2}{c}{Source} & & & \multicolumn{3}{c}{Sweeps}\\\cmidrule(lr){3-4}\cmidrule(lr){7-9}
$\xi$ & $w$ & $\hat S_{\mathrm{map}}$ & $\hat S_{\mathrm{fr}}$ & $\hat\beta_W$ & $\hat e_0$ & $s_{\mathrm{pred}}$ & $s_{\mathrm{fr}}$ & $s_{\mathrm{hit}}$\\
\midrule
\multicolumn{9}{c}{\textit{Lorenz-63}}\\*
$10$ & $61.8$ & $7.25\times 10^{-8}$ & $3.45\times 10^{-1}$ & $3.64\times 10^{-8}$ & $8.02\times 10^{-9}$ & $1$ & $2$ & $1$\\*
$100$ & $365$ & $3.67\times 10^{-2}$ & $3.78$ & $3.12\times 10^{-3}$ & $8.96\times 10^{-6}$ & $2$ & $9$ & $1$\\
\midrule
\multicolumn{9}{c}{\textit{Logistic}}\\*
$10$ & $2.48$ & $2.06\times 10^{-3}$ & $6.75\times 10^{-3}$ & $2.58\times 10^{-2}$ & $1.68\times 10^{-3}$ & $6$ & $9$ & $3$\\*
$100$ & $2.48$ & $2.18\times 10^{-2}$ & $7.42\times 10^{-2}$ & $2.73\times 10^{-1}$ & $1.78\times 10^{-2}$ & $19$ & $309$ & $5$\\
\midrule
\multicolumn{9}{c}{\textit{Periodic}}\\*
$10$ & $194$ & $2.18\times 10^{-5}$ & $6.81\times 10^{-1}$ & $3.47\times 10^{-6}$ & $2.51\times 10^{-8}$ & $1$ & $2$ & $1$\\*
$100$ & $194$ & $4.97\times 10^{-1}$ & $7.45$ & $7.93\times 10^{-2}$ & $4.53\times 10^{-4}$ & $6$ & -- & $2$\\
\bottomrule\end{tabular*}
\end{table}

All six sampled budgets cover the crossings (\Cref{fig:measured_vs_bound}), with initial-error estimates 5.4--5.7 times the measured errors. In five cases, the error reduction per sweep up to the crossing is about 50 to $2\times10^4$ times stronger than predicted by $\hat{\beta}_W$. For Lorenz-63 at $\xi = 10$, the observed one-sweep error ratio is about $3.2\hat{\beta}_W$, but the initial error is only about 1.4 times the target and one sweep suffices. The sampled envelopes fall below measured errors only after the target crossing, at roundoff scale; they do not provide uniform bounds.

For Lorenz-63 at $\xi = 100$, $w \approx 365$ gives $w^{-4} \approx 5.6\times10^{-11}$: unit error at interface four contributes less than $10^{-9}$ to \Cref{eq:start_norm}. Small budgets therefore concern a strongly discounted norm.

For Lorenz-63 at $\xi = 10$, coarse-pass estimates are $\hat{\mu} = 43.4$, $\hat{\nu} = 768$, and fine/coarse amplification about 61.8. The frozen and actual source estimates, 0.345 and $7.25 \times 10^{-8}$, predict two sweeps and one, respectively, compared with one observed sweep. The RK4 frozen estimate is therefore less sharp. At $\xi = 100$, the logistic frozen budget exceeds $N$ and the periodic case predicts no contraction.

\label{sec:checks}
Supplementary \Cref{SM-sec:implementation_checks,SM-sec:frozen_prediction} test transport and constant-coefficient predictions; matched Euler predictions remain inconclusive, with nonfinite coarse propagation at the larger step (\Cref{SM-sec:coarse_prediction}). \Cref{SM-sec:numerical_checks} tests roundoff, refinement, tolerance, and secant-guard sensitivity.

\Needspace{5\baselineskip}
\section{Conclusions}
\label{sec:conclusions}

In this work, we have studied time-parallel methods for chaotic dynamical systems through a finite-budget convergence framework and a chaos-aware stopping criterion.

First, by modelling the time-parallel iteration as a local contraction mapping, we introduced an outer--inner-ball framework giving a sufficient budget for a prescribed tolerance and initial-error bound. Parareal's transport--source decomposition and first-batch defects supply the inputs. The frozen expansion separates tableau mismatch from Jacobian drift; actual-map estimates distinguish fixed-chunk refinement from shortening the window. Uniform neighbourhood hypotheses give short-window contraction, and sampled predictions cover all six tested weighted-error crossings.

Second, we introduced the \emph{proximity function}, which weights interface defects, and proved two-sided bounds relating it to weighted error. With accurate coarse propagation, the weighted rules usually accept after two or three fine batches in our Lorenz-63 and Lorenz-96 tests; accurate prefixes can extend beyond exact chunks.

The separate statistical assessment finds no clear energy-error improvement over coarse-only for Lorenz-96. In the independent Lorenz-63 test, Lyapunov stopping improves on coarse-only, meets the adequacy limit, and has lower modelled cost than standard stopping. Its gain over the zero-correction first fine batch is not significant, however, and intermediate-step serial RK4 meets the same limit at far lower cost.

\Needspace{5\baselineskip}
\subsection{Limitations and future work}
\label{sec:future_work}

Neither weight is uniformly preferable, sampled derivatives give no uniform enclosure, and modelled costs exclude communication.

Future work will examine adaptive moving windows to keep the discounted horizon short. Restarting the discount bounds local defects by $N_w\epsilon w^{N_w}$ at fixed $N_w$, $w \ge 1$, and $\epsilon$, independently of window count (\Cref{rem:w_not_free}). Further analysis is needed for propagated-start errors and long-time statistics. Tests should assess overlap, restarts, and total cost against coarse-only and serial baselines; \Cref{SM-sec:window_cost} gives the local cost model.

\bibliographystyle{siamplain}
\begingroup
\interlinepenalty=10000
\bibliography{references}

@article{aubanel_scheduling_2011,
  title = {Scheduling of tasks in the parareal algorithm},
  author = {Aubanel, Eric},
  year = 2011,
  month = mar,
  journal = {Parallel Comput.},
  volume = {37},
  number = {3},
  pages = {172--182},
  issn = {0167-8191},
  doi = {10.1016/j.parco.2010.10.004}
}

@inproceedings{bal_convergence_2005,
  title = {On the Convergence and the Stability of the Parareal Algorithm to Solve Partial Differential Equations},
  booktitle = {Domain Decomposition Methods in Science and Engineering},
  author = {Bal, Guillaume},
  editor = {Barth, Timothy J. and Griebel, Michael and Keyes, David E. and Nieminen, Risto M. and Roose, Dirk and Schlick, Tamar and Kornhuber, Ralf and Hoppe, Ronald and Périaux, Jacques and Pironneau, Olivier and Widlund, Olof and Xu, Jinchao},
  year = 2005,
  series = {Lecture Notes in Computational Science and Engineering},
  volume = {40},
  pages = {425--432},
  publisher = {Springer},
  address = {Berlin, Heidelberg},
  doi = {10.1007/3-540-26825-1_43},
  isbn = {978-3-540-26825-3},
  langid = {english}
}

@misc{chen_adjoint_2015,
  title = {An Adjoint Approach for Stabilizing the {Parareal} Method},
  author = {Chen, Feng and Hesthaven, Jan S. and Maday, Yvon and Nielsen, Allan Svejstrup},
  year = {2015},
  note = {Preprint},
  url = {https://infoscience.epfl.ch/bitstreams/52cf50be-a7a0-41c5-8cf2-f8ee51f757f4/download}
}

@article{dembo_inexact_1982,
  title = {Inexact Newton Methods},
  author = {Dembo, Ron S. and Eisenstat, Stanley C. and Steihaug, Trond},
  year = 1982,
  journal = {SIAM J. Numer. Anal.},
  volume = {19},
  number = {2},
  pages = {400--408},
  doi = {10.1137/0719025}
}

@article{emmett_efficient_2012,
  title = {Toward an efficient parallel in time method for partial differential equations},
  author = {Emmett, Matthew and Minion, Michael},
  year = 2012,
  month = mar,
  journal = {Commun. Appl. Math. Comput. Sci.},
  volume = {7},
  number = {1},
  pages = {105--132},
  publisher = {Mathematical Sciences Publishers},
  issn = {2157-5452},
  doi = {10.2140/camcos.2012.7.105}
}

@article{falgout_parallel_2014,
  title = {Parallel Time Integration with Multigrid},
  author = {Falgout, R. D. and Friedhoff, S. and Kolev, {\relax Tz}. V. and MacLachlan, S. P. and Schroder, J. B.},
  year = 2014,
  month = nov,
  journal = {SIAM J. Sci. Comput.},
  volume = {36},
  number = {6},
  pages = {C635-C661},
  publisher = {{Society for Industrial and Applied Mathematics}},
  issn = {1064-8275},
  doi = {10.1137/130944230}
}

@inproceedings{gander_50_2015,
  title = {50 Years of Time Parallel Time Integration},
  booktitle = {Multiple Shooting and Time Domain Decomposition Methods},
  author = {Gander, Martin J.},
  editor = {Carraro, Thomas and Geiger, Michael and Körkel, Stefan and Rannacher, Rolf},
  year = 2015,
  series = {Contributions in Mathematical and Computational Sciences},
  volume = {9},
  pages = {69--113},
  publisher = {Springer International Publishing},
  address = {Cham},
  doi = {10.1007/978-3-319-23321-5_3},
  isbn = {978-3-319-23321-5},
  langid = {english}
}

@article{gander_analysis_2007,
  title = {Analysis of the Parareal Time‐Parallel Time‐Integration Method},
  author = {Gander, Martin J. and Vandewalle, Stefan},
  year = 2007,
  month = jan,
  journal = {SIAM J. Sci. Comput.},
  volume = {29},
  number = {2},
  pages = {556--578},
  publisher = {{Society for Industrial and Applied Mathematics}},
  issn = {1064-8275},
  doi = {10.1137/05064607X}
}

@inproceedings{gander_nonlinear_2008,
  title = {Nonlinear Convergence Analysis for the Parareal Algorithm},
  booktitle = {Domain Decomposition Methods in Science and Engineering XVII},
  author = {Gander, Martin J. and Hairer, Ernst},
  editor = {Langer, Ulrich and Discacciati, Marco and Keyes, David E. and Widlund, Olof B. and Zulehner, Walter},
  year = 2008,
  series = {Lecture Notes in Computational Science and Engineering},
  volume = {60},
  pages = {45--56},
  publisher = {Springer},
  address = {Berlin, Heidelberg},
  doi = {10.1007/978-3-540-75199-1_4},
  isbn = {978-3-540-75199-1},
  langid = {english}
}

@book{gander_time_2024,
  title = {Time Parallel Time Integration},
  author = {Gander, Martin J. and Lunet, Thibaut},
  year = 2024,
  month = jan,
  series = {CBMS-NSF Regional Conference Series in Applied Mathematics},
  publisher = {{Society for Industrial and Applied Mathematics}},
  doi = {10.1137/1.9781611978025},
  isbn = {978-1-61197-801-8}
}

@book{hairer_solving_1996,
  title = {Solving Ordinary Differential Equations II},
  author = {Hairer, Ernst and Wanner, Gerhard},
  year = 1996,
  series = {Springer Series in Computational Mathematics},
  volume = {14},
  publisher = {Springer},
  address = {Berlin, Heidelberg},
  doi = {10.1007/978-3-642-05221-7},
  isbn = {978-3-642-05220-0 978-3-642-05221-7},
  langid = {english}
}

@article{lions_resolution_2001,
  title = {Résolution d'EDP par un schéma en temps «pararéel»},
  author = {Lions, Jacques-Louis and Maday, Yvon and Turinici, Gabriel},
  year = 2001,
  month = apr,
  journal = {C. R. Acad. Sci. Paris S{\'e}r. I Math.},
  series = {Série I - Mathematics},
  volume = {332},
  number = {7},
  pages = {661--668},
  issn = {0764-4442},
  doi = {10.1016/S0764-4442(00)01793-6}
}

@article{lorenz_reply_2008,
  title = {Reply to comment by L.-S. Yao and D. Hughes},
  author = {Lorenz, E. N.},
  year = 2008,
  month = may,
  journal = {Tellus A},
  volume = {60},
  number = {4},
  pages = {806--807},
  doi = {10.1111/j.1600-0870.2007.00302.x}
}

@article{maday_parareal_2002,
  title = {A parareal in time procedure for the control of partial differential equations},
  author = {Maday, Yvon and Turinici, Gabriel},
  year = 2002,
  journal = {C. R. Math. Acad. Sci. Paris},
  volume = {335},
  number = {4},
  pages = {387--392},
  issn = {1778-3569},
  doi = {10.1016/S1631-073X(02)02467-6},
  langid = {english}
}

@book{ortega_iterative_1970,
  title = {Iterative Solution of Nonlinear Equations in Several Variables},
  author = {Ortega, James M. and Rheinboldt, Werner C.},
  year = 1970,
  publisher = {Academic Press},
  isbn = {978-0-12-528550-6},
  langid = {english}
}

@article{rao_adjointbased_2014,
  title = {An adjoint-based scalable algorithm for time-parallel integration},
  author = {Rao, Vishwas and Sandu, Adrian},
  year = 2014,
  month = mar,
  journal = {J. Comput. Sci.},
  series = {Empowering Science through Computing + BioInspired Computing},
  volume = {5},
  number = {2},
  pages = {76--84},
  issn = {1877-7503},
  doi = {10.1016/j.jocs.2013.03.004}
}

@misc{references,
  urldate = {2025-12-05},
  howpublished = {http://parallel-in-time.org/references/}
}

@article{ruprecht_wave_2018,
  title = {Wave propagation characteristics of Parareal},
  author = {Ruprecht, Daniel},
  year = 2018,
  month = jun,
  journal = {Comput. Vis. Sci.},
  volume = {19},
  number = {1},
  pages = {1--17},
  issn = {1433-0369},
  doi = {10.1007/s00791-018-0296-z},
  langid = {english}
}

@article{samaddar_parallelization_2010,
  title = {Parallelization in time of numerical simulations of fully-developed plasma turbulence using the parareal algorithm},
  author = {Samaddar, D. and Newman, D. E. and Sánchez, R.},
  year = 2010,
  month = sep,
  journal = {J. Comput. Phys.},
  volume = {229},
  number = {18},
  pages = {6558--6573},
  issn = {0021-9991},
  doi = {10.1016/j.jcp.2010.05.012}
}

@inproceedings{staff_stability_2005,
  title = {Stability of the Parareal Algorithm},
  booktitle = {Domain Decomposition Methods in Science and Engineering},
  author = {Staff, Gunnar Andreas and Rønquist, Einar M.},
  editor = {Barth, Timothy J. and Griebel, Michael and Keyes, David E. and Nieminen, Risto M. and Roose, Dirk and Schlick, Tamar and Kornhuber, Ralf and Hoppe, Ronald and Périaux, Jacques and Pironneau, Olivier and Widlund, Olof and Xu, Jinchao},
  year = 2005,
  pages = {449--456},
  publisher = {Springer},
  address = {Berlin, Heidelberg},
  doi = {10.1007/3-540-26825-1_46},
  isbn = {978-3-540-26825-3},
  langid = {english}
}

@book{stuart_dynamical_1996,
  publisher = {Cambridge University Press},
  title = {Dynamical Systems and Numerical Analysis},
  author = {Stuart, Andrew M. and Humphries, A. R.},
  year = 1996,
  month = aug,
  series = {Cambridge Monographs on Applied and Computational Mathematics},
  isbn = {978-0-521-49672-8},
  langid = {english}
}

@book{pilyugin_shadowing_1999,
 author={Pilyugin, Sergei Yu.}, title={Shadowing in Dynamical Systems},
 series={Lecture Notes in Mathematics}, volume={1706},
 publisher={Springer}, address={Berlin}, year={1999}, doi={10.1007/BFb0093184}}

@article{vanvleck_shadowing_1995,
 author={Van Vleck, Erik S.}, title={Numerical Shadowing Near Hyperbolic Trajectories},
 journal={SIAM J. Sci. Comput.}, volume={16}, number={5},
 pages={1177--1189}, year={1995}, doi={10.1137/0916068}}

@book{hairer_nonstiff_1993,
 author={Hairer, Ernst and N{\o}rsett, Syvert P. and Wanner, Gerhard},
 title={Solving Ordinary Differential Equations I: Nonstiff Problems},
 edition={2}, series={Springer Series in Computational Mathematics}, volume={8},
 publisher={Springer}, address={Berlin}, year={1993}, doi={10.1007/978-3-540-78862-1}}

@book{sparrow_lorenz_1982,
 author={Sparrow, Colin}, title={The {Lorenz} Equations: Bifurcations, Chaos, and Strange Attractors},
 series={Applied Mathematical Sciences}, volume={41}, publisher={Springer},
 address={New York}, year={1982}, doi={10.1007/978-1-4612-5767-7}}

@article{young_srb_2002,
 author={Young, Lai-Sang}, title={What Are {SRB} Measures, and Which Dynamical Systems Have Them?},
 journal={J. Stat. Phys.}, volume={108}, pages={733--754}, year={2002}}

@article{coomes_rigorous_1995,
 author={Coomes, Brian A. and Ko{\c c}ak, H{\"u}seyin and Palmer, Kenneth J.},
 title={Rigorous computational shadowing of orbits of ordinary differential equations},
 journal={Numer. Math.}, volume={69}, number={4}, pages={401--421},
 year={1995}, doi={10.1007/s002110050100}}

@article{vargas_multigrid_2023,
 author={Vargas, David A. and Falgout, Robert D. and G{\"u}nther, Stefanie and Schroder, Jacob B.},
 title={Multigrid Reduction in Time for Chaotic Dynamical Systems},
 journal={SIAM J. Sci. Comput.}, volume={45}, number={4},
 pages={A2019--A2042}, year={2023},
 doi={10.1137/22M1518335}}
\endgroup
\end{document}


\raggedbottom

\maketitle

This supplement expands the frozen-Jacobian algebra and higher-order budget construction, then supplies the full scaling and statistical records behind the main comparisons. The functional tests document their independent references, calibration, and confirmation; the prediction and implementation checks assess derivative estimates, roundoff, refinement, and stopping sensitivity. The final sections report the corrected long-chunk runs and give the moving-window cost model, notation, and computational records.

\section[Proof of the frozen-Jacobian theorem]{Proof of \Cref{thm:parareal_beta_second_term}}
\label{sec:second_term}

\subsection{Setting}

We work in the frozen-Jacobian model of \Cref{eq:jacobian_products} and hold $L, \xi$ fixed as $h \to 0$. Along the reference trajectory through $U_n$, the propagator Jacobians over the $n$th chunk are
\begin{equation}
  \label{eq:sm_products}
  F_{\mathrm{fr}}' = \prod_{i = 1}^{\xi L} \mathcal{R}_{\mathcal{F}}(X_i),
  \qquad G_{\mathrm{fr}}' = \prod_{j = 1}^{L} \mathcal{R}_{\mathcal{G}}(Y_j).
\end{equation}
Later factors appear on the left. The matrices
\begin{equation}
  \label{eq:sm_XY}
  X_i = h \, J\bigl(T_n + (i - 1) h \bigr),
  \qquad Y_j = \xi h \, J\bigl(T_n + (j - 1) \xi h \bigr)
\end{equation}
are the scaled Jacobians at the fine and coarse nodes. The function
\begin{equation}
  \label{eq:sm_R}
  \mathcal{R}(Z) = I_d + (b^\top \otimes I_d) \, (I_{n_{\mathrm{st}}d} - A \otimes Z)^{-1} (\mathbf{1}_{n_{\mathrm{st}}} \otimes Z)
\end{equation}
is the matrix stability function for a Runge--Kutta method with tableau $(A, b, c)$ and $n_{\mathrm{st}}$ stages. For the linear problem $u' = Mu$ with $hM = Z$, it gives the step Jacobian~\cite{hairer_solving_1996}. Here $I_{n_{\mathrm{st}}d} = I_{n_{\mathrm{st}}} \otimes I_d$, $\mathbf{1}_{n_{\mathrm{st}}}$ is the vector of ones, and $\otimes$ denotes the Kronecker product. For the two methods we write $I_{\mathcal{F}} = I_{n_{\mathcal{F}} d}$, $I_{\mathcal{G}} = I_{n_{\mathcal{G}} d}$, $\mathbf{1}_{\mathcal{F}} = \mathbf{1}_{n_{\mathcal{F}}}$, and $\mathbf{1}_{\mathcal{G}} = \mathbf{1}_{n_{\mathcal{G}}}$.

The fine node $i$ lies in coarse block $j = \lceil i/\xi\rceil$, so $\xi(j - 1) + 1 \le i \le \xi j$. The time offset is $h[(i - 1) - \xi(j - 1)] \in [0, (\xi - 1)h]$. Integrating $J'$ between the coarse and fine nodes gives
\begin{equation}
  \label{eq:sm_Xi_expansion}
  X_i = \frac{Y_j}{\xi} + h \, E_{ij},
  \qquad E_{ij} = \int_{T_n + (j - 1) \xi h}^{T_n + (i - 1) h} J'(\tau) \, d\tau.
\end{equation}
The fine node is no earlier than the coarse node, so the integral has the positive sign shown. The remainder satisfies $\|E_{ij}\| \le h[(i - 1) - \xi(j - 1)]\nu$. This follows from the integral bound; the equality form of the mean-value theorem is not available for matrix-valued functions. Thus $X_i = Y_j/\xi + O(h^2)$, $\|X_i\| \le h\mu$, and $\|Y_j\| \le \xi h\mu$.

\subsection{Expansion of the stability functions}

For stage vectors, use the maximum of the state norms. For tableau coefficients, $\|\cdot\|_\infty$ denotes the vector maximum or matrix maximum-row-sum norm, so $\|A \otimes X\| \le \|A\|_\infty\|X\|$ in this induced block norm. The step-size condition \Cref{eq:step_size_condition} gives $\|A_{\mathcal{F}} \otimes X_i\| \le \|A_{\mathcal{F}}\|_\infty h\mu < 1$ and $\|A_{\mathcal{G}} \otimes Y_j\| \le \|A_{\mathcal{G}}\|_\infty\xi h\mu < 1$. The inverses in \Cref{eq:sm_R} therefore have the Neumann expansions
\begin{equation}
  \label{eq:sm_neumann}
  \begin{aligned}
    (I_{\mathcal{F}} - A_{\mathcal{F}} \otimes X_i)^{-1} &= \sum_{k = 0}^\infty A_{\mathcal{F}}^k \otimes X_i^k, \\
    (I_{\mathcal{G}} - A_{\mathcal{G}} \otimes Y_j)^{-1} &= \sum_{k = 0}^\infty A_{\mathcal{G}}^k \otimes Y_j^k.
  \end{aligned}
\end{equation}
Using $(b^\top \otimes I_d)(A^k \otimes X^k)(\mathbf{1} \otimes X) = (b^\top A^k\mathbf{1})X^{k + 1}$ to substitute these series into \Cref{eq:sm_R} gives
\begin{equation}
  \label{eq:sm_PQ}
  \begin{aligned}
    \mathcal{R}_{\mathcal{F}}(X_i) &= I_d + P_i, & P_i &= \sum_{k = 0}^\infty \bigl(b_{\mathcal{F}}^\top A_{\mathcal{F}}^k \mathbf{1}_{\mathcal{F}} \bigr) X_i^{k + 1}, \\
    \mathcal{R}_{\mathcal{G}}(Y_j) &= I_d + Q_j, & Q_j &= \sum_{k = 0}^\infty \bigl(b_{\mathcal{G}}^\top A_{\mathcal{G}}^k \mathbf{1}_{\mathcal{G}} \bigr) Y_j^{k + 1}.
  \end{aligned}
\end{equation}
Consistency gives $b^\top\mathbf{1} = 1$. The row-sum condition $A\mathbf{1} = c$ gives $b^\top A^k\mathbf{1} = b^\top A^{k - 1}c$ for $k \ge 1$. Hence
\begin{equation}
  \label{eq:sm_PQ_expanded}
  P_i = X_i + \bigl(b_{\mathcal{F}}^\top c_{\mathcal{F}} \bigr) X_i^2 + O(h^3),
  \qquad Q_j = Y_j + \bigl(b_{\mathcal{G}}^\top c_{\mathcal{G}} \bigr) Y_j^2 + O(h^3).
\end{equation}
The tails with $k \ge 2$ are $O(h^3)$. For the fine solver they are bounded by the convergent series $\sum_{k \ge 2}n_{\mathcal{F}}\|b_{\mathcal{F}}\|_\infty\|c_{\mathcal{F}}\|_\infty\|A_{\mathcal{F}}\|_\infty^{k - 1}(h\mu)^{k + 1}$. The coarse bound uses $n_{\mathcal{G}}, b_{\mathcal{G}}, c_{\mathcal{G}}, A_{\mathcal{G}}$ and replaces $h\mu$ by $\xi h\mu$. The explicit factor $\|A\|_\infty$ occurs in this tail bound, at order $h^3$.

\subsection{Expansion of the products}

Let $D_n = F_{\mathrm{fr}}' - G_{\mathrm{fr}}'$. Expanding the products gives
\begin{equation}
  \label{eq:sm_Dn}
  D_n = \prod_{i = 1}^{\xi L} (I_d + P_i) - \prod_{j = 1}^{L} (I_d + Q_j).
\end{equation}
Here we used \Cref{eq:sm_products,eq:sm_PQ}. Each $P_i$ and $Q_j$ is $O(h)$. In expanding a product, each factor is selected at most once. Grouping terms by the number of selected factors gives
\begin{equation}
  \label{eq:sm_Delta}
  D_n = \underbrace{\sum_{i = 1}^{\xi L} P_i - \sum_{j = 1}^{L} Q_j}_{\Delta_1}
  {} + \underbrace{\sum_{i_1 > i_2} P_{i_1} P_{i_2} - \sum_{j_1 > j_2} Q_{j_1} Q_{j_2}}_{\Delta_2}
  {} + O(h^3).
\end{equation}
The pair sums run over $\xi L \ge i_1 > i_2 \ge 1$ and $L \ge j_1 > j_2 \ge 1$. The later factor remains on the left, as in \Cref{eq:sm_Dn}. Since $L, \xi$ are fixed, the omitted products with three or more factors are $O(h^3)$.

\subsection{The single-factor part \texorpdfstring{$\Delta_1$}{Delta 1}}

Substituting \Cref{eq:sm_PQ_expanded},
\begin{equation}
  \label{eq:sm_Delta1}
  \Delta_1 = \underbrace{\sum_{i = 1}^{\xi L} X_i - \sum_{j = 1}^{L} Y_j}_{\text{linear part}}
  {} + \underbrace{b_{\mathcal{F}}^\top c_{\mathcal{F}} \sum_{i = 1}^{\xi L} X_i^2 - b_{\mathcal{G}}^\top c_{\mathcal{G}} \sum_{j = 1}^{L} Y_j^2}_{\text{quadratic part}}
  {} + O(h^3).
\end{equation}

\emph{Linear part (Jacobian drift).} Group the fine nodes by coarse block. By \Cref{eq:sm_Xi_expansion}, summing the $\xi$ fine nodes in the $j$th block gives a drift $h\sum_i E_{ij}$. We use the bound $\|E_{ij}\| \le h[(i - 1) - \xi(j - 1)]\nu$. In this bound the offsets $(i - 1) - \xi(j - 1)$ run through $0, 1, \ldots, \xi - 1$. Summing them gives
\begin{equation}
  \label{eq:sm_drift}
  \begin{aligned}
    \Bigl\| \sum_{i = 1}^{\xi L} X_i - \sum_{j = 1}^{L} Y_j \Bigr\|
    &\le h^2 \nu \sum_{j = 1}^{L} \sum_{m = 0}^{\xi - 1} m \\
    &= L h^2 \nu \, \frac{\xi (\xi - 1)}{2} \\
    &= L \xi (h \mu)^2 \, \frac{\xi - 1}{2} \, \frac{\nu}{\mu^2}.
  \end{aligned}
\end{equation}
This gives $C_2$, which measures the Jacobian drift within each coarse step that the fine solver samples and the coarse solver does not.

\emph{Quadratic part.} By \Cref{eq:sm_Xi_expansion}, $X_i^2 = Y_j^2/\xi^2 + O(h^3)$ for each $i$ in block $j$. Summing these identities over the $\xi$ fine nodes in block $j$ gives $Y_j^2/\xi + O(h^3)$. Thus
\begin{equation}
  \label{eq:sm_quad}
  b_{\mathcal{F}}^\top c_{\mathcal{F}} \sum_{i = 1}^{\xi L} X_i^2 - b_{\mathcal{G}}^\top c_{\mathcal{G}} \sum_{j = 1}^{L} Y_j^2
  = \Bigl(\frac{b_{\mathcal{F}}^\top c_{\mathcal{F}}}{\xi} - b_{\mathcal{G}}^\top c_{\mathcal{G}} \Bigr) \sum_{j = 1}^{L} Y_j^2 + O(h^3).
\end{equation}

\subsection{The two-factor part \texorpdfstring{$\Delta_2$}{Delta 2}}

Since $P_i = X_i + O(h^2)$ and $Q_j = Y_j + O(h^2)$,
\begin{equation}
  \label{eq:sm_Delta2}
  \Delta_2 = \sum_{i_1 > i_2} X_{i_1} X_{i_2} - \sum_{j_1 > j_2} Y_{j_1} Y_{j_2} + O(h^3).
\end{equation}
Split the fine pairs $(i_1, i_2)$ by their coarse blocks. For distinct blocks $j_1 > j_2$, we have $X_{i_1}X_{i_2} = Y_{j_1}Y_{j_2}/\xi^2 + O(h^3)$. Each block pair contains $\xi^2$ fine pairs, so their sum cancels the coarse pair up to $O(h^3)$. For two indices in the same block $j$, we instead have $X_{i_1}X_{i_2} = Y_j^2/\xi^2 + O(h^3)$. There are $\binom{\xi}{2} = \xi(\xi - 1)/2$ such pairs and no corresponding coarse pair. Therefore
\begin{equation}
  \label{eq:sm_cross}
  \Delta_2 = \frac{\xi - 1}{2 \xi} \sum_{j = 1}^{L} Y_j^2 + O(h^3).
\end{equation}

\subsection{Combination and conclusion}

The terms in \Cref{eq:sm_quad,eq:sm_cross} are scalar multiples of the same matrix $\sum_j Y_j^2$. Combine them before taking norms to retain cancellation between the tableau coefficients and the pair count. Bounding them separately would give a weaker constant. \Cref{eq:sm_Delta,eq:sm_Delta1,eq:sm_quad,eq:sm_cross} yield
\begin{equation}
  \label{eq:sm_Dn_combined}
  D_n = \Bigl(\sum_{i = 1}^{\xi L} X_i - \sum_{j = 1}^{L} Y_j \Bigr)
  + \Bigl(\frac{b_{\mathcal{F}}^\top c_{\mathcal{F}}}{\xi} - b_{\mathcal{G}}^\top c_{\mathcal{G}} + \frac{\xi - 1}{2 \xi} \Bigr) \sum_{j = 1}^{L} Y_j^2
  + O(h^3).
\end{equation}
Apply \Cref{eq:sm_drift} to the first term and $\|\sum_j Y_j^2\| \le L(\xi h\mu)^2$ to the second. Factoring out $L\xi(h\mu)^2$ gives
\begin{equation}
  \label{eq:sm_Dn_bound}
  \begin{aligned}
    \| D_n \| &\le L \xi (h \mu)^2 \, \frac{\xi - 1}{2} \, \frac{\nu}{\mu^2} \\
    &\quad {} + L \xi (h \mu)^2 \, \Bigl| b_{\mathcal{F}}^\top c_{\mathcal{F}} - \xi \, b_{\mathcal{G}}^\top c_{\mathcal{G}} + \frac{\xi - 1}{2} \Bigr| + O(h^3).
  \end{aligned}
\end{equation}
This is \Cref{eq:parareal_explicit_bound} with the constants in \Cref{eq:C1C2}. The estimate is uniform in $n$ and over a family of trajectories whenever the bounds on $J$, $J'$, and the remainder hold uniformly. This remains a bound for the frozen model; it does not by itself enclose actual nonlinear Runge--Kutta derivatives. \hfill $\square$

\section{The higher-order budget formula}
\label{sec:higher_order}

Suppose the inequality $\|e^{j + 1}\| \le \beta\|e^j\|^q$, $q > 1$ and $\beta > 0$, holds at every iterate in an invariant neighbourhood. Induction gives
\begin{equation}
  \label{eq:sm_higher_order}
  \|e^s\| \le \beta^{1 + q + \cdots + q^{s - 1}}\|e^0\|^{q^s}
  = \beta^{(q^s - 1)/(q - 1)}\|e^0\|^{q^s}.
\end{equation}
Inside an invariant ball of radius $\varrho$, with target $0 < r \le \varrho$, the certified outer radius is
\begin{equation}
  \label{eq:rho_bound_general}
  R(s) = \min\!\left\{\varrho,
  \left(\frac{r}{\beta^{(q^s - 1)/(q - 1)}}\right)^{1/q^s}\right\}.
\end{equation}
The limit $q \to 1$ is $\min\{\varrho, r\beta^{-s}\}$. A statement about asymptotic Q-order alone supplies neither the uniform inequality nor invariance; those must be checked separately. This is why the main working prediction uses the derivative-based linear-contraction estimate.

\section{Background: attractors and shadowing}
\label{sec:background}

Assumption (A3) concerns a compact absorbing set and numerical invariance; continuous dissipativity alone does not guarantee stability of an arbitrarily large explicit coarse step. Upper semicontinuity of numerical attractors places their points near the continuous attractor under the hypotheses of the approximation theorem, but the reverse inclusion needs additional structure~\cite[Chapter 7]{stuart_dynamical_1996}. Convergence of sets is distinct from convergence of physical measures. The role of SRB measures and their stability for appropriate uniformly hyperbolic maps is discussed by Young~\cite{young_srb_2002}; we do not transfer such a conclusion automatically to a time discretization of the Lorenz-63 flow.

A $\delta$-pseudo-orbit of a map satisfies $\|V_n - \mathcal{F}(V_{n - 1})\| \le \delta$. Uniformly hyperbolic discrete systems admit local shadowing results~\cite{pilyugin_shadowing_1999}, whereas shadowing of flows typically allows time reparametrization~\cite{vanvleck_shadowing_1995}. Rigorous computational bounds are available for particular ODE trajectories~\cite{coomes_rigorous_1995}. The defect bound in \Cref{rem:w_not_free} alone supplies neither a shadowing theorem nor statistical typicality. Such conclusions require additional hypotheses, which we do not impose on the Lorenz-63 tests.

\Cref{fig:shadowing} in the article illustrates the distinction between a shadowing orbit and the prescribed-initial-condition orbit.

The finite-budget geometry can be difficult to determine: as chunks lengthen, \Cref{fig:ErrorContoursLorenz} shows filaments developing in the interface-defect landscape. The figure depicts interface defects rather than a basin of attraction; it motivates bounding a ball inside the acceptance set without inferring the set from the plot.

\begin{figure}[H]
  \centering
  \includegraphics[width=\textwidth]{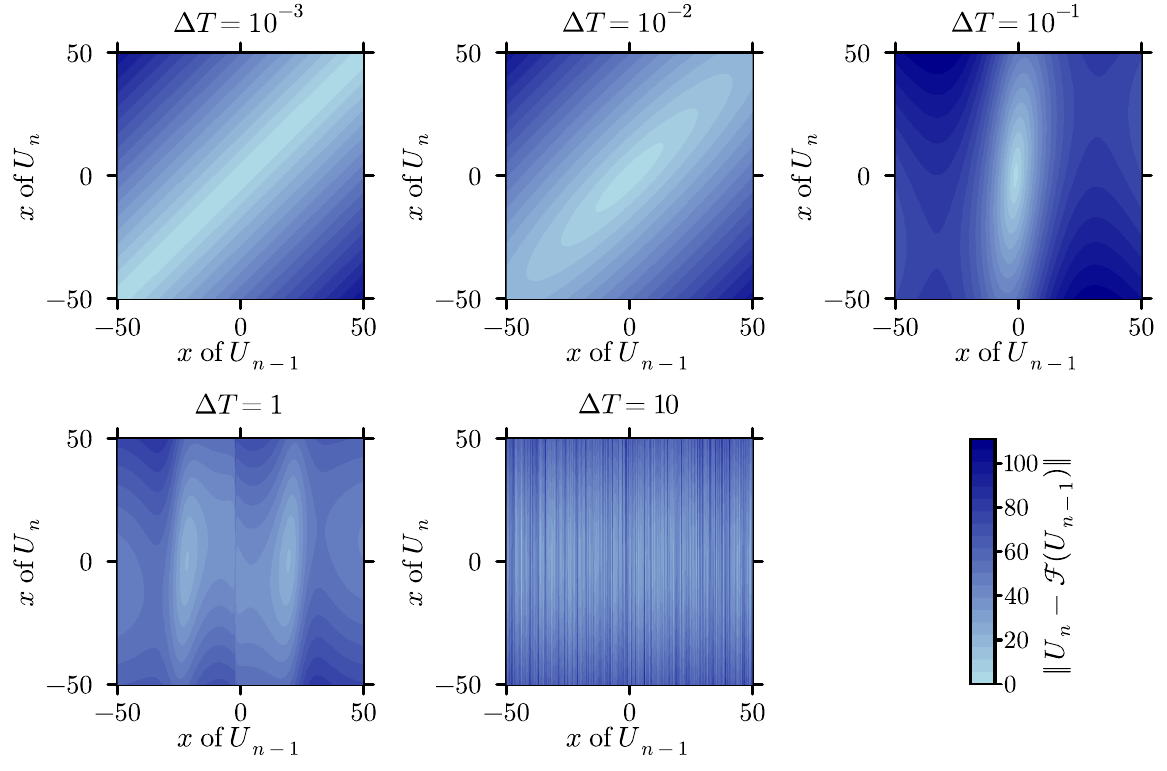}
  \caption{For Lorenz-63, colour represents $\|U_n - \mathcal{F}(U_{n - 1})\|$ as the first coordinates of two interface states vary over $[-50, 50]$, with $(y, z) = (3, -14)$. The five chunk lengths are $\Delta T = 10^{-3}, 10^{-2}, 10^{-1}, 1, 10$, with RK4 fine step $h = 10^{-4}$. Longer chunks produce filaments in the low-defect sets; the plot describes a residual landscape, not an asymptotic Parareal basin.}
  \label{fig:ErrorContoursLorenz}
\end{figure}

\section{Linear reference points for the weighted criterion}
\label{sec:linear_benchmarks}

The linear examples in \Cref{tab:bracketing} illustrate the error-conversion regimes of \Cref{sec:scalability_regimes} in familiar settings. The fine-map amplification alone does not determine Parareal convergence, whose derivative bound also involves the fine--coarse mismatch and coarse transport. In particular, a unitary advection map is critical for error conversion, but coarse phase errors are a separate obstacle to rapid convergence~\cite{ruprecht_wave_2018}.

\begin{table}[!htbp]
  \centering
  \small
  \caption{\rev{Linear reference points and the chaotic case. These illustrate defect-to-error conversion, not Parareal sweep counts. Norm statements require the indicated operator and time-discretization properties.}}
  \label{tab:bracketing}
  \begin{tabularx}{\textwidth}{@{}lXX@{}}
    \toprule
    Problem & Fine-map amplification & Consequence for weighting \\
    \midrule
    Heat & Strict contraction with Dirichlet conditions and a spectral gap, for a suitable stable discretization. Constant modes with periodic or Neumann conditions have norm one. & $w = 1$ is natural; discounting is unnecessary for a contractive map. \\
    \addlinespace[4pt]
    Advection & Exact propagation is unitary for a skew-adjoint spatial operator in the appropriate inner product. Imaginary spectrum alone is insufficient; RK4 is not unitary. & $w = 1$ gives the critical case for a unitary map. Weighting does not fix coarse phase errors. \\
    \addlinespace[4pt]
    Chaotic dynamics & Finite-time amplification and fine--coarse mismatch both matter; $\lambda = 0$ does not imply $\|D\mathcal{F}\| = 1$. & $w \ge \Lambda_{\mathcal{F}}$ gives polynomial conversion constants. Sampled or Lyapunov weights require empirical assessment. \\
    \bottomrule
  \end{tabularx}
\end{table}

\section{Implementation conventions}
\label{sec:stored_starts}

We return fine trajectories evolved from the stored starts $U_0, \ldots, U_{N - 1}$, with $U_0 = u_0$. Terminal reconstruction gives $U_N = \mathcal{F}(U_{N - 1})$ and $d_N = 0$. The article's \Cref{sec:prediction,eq:start_bound} connect the stored-start and full trajectory errors. The defect recurrence with $d_N = 0$ and the stated $NH_{N, 2}$ constant bound the stored-start norm; retaining the final transported component is conservative. Terminal reconstruction is not an additional correction of the iteration map. Full pointwise error tables and prefix plots retain all $N + 1$ interfaces.

For the logistic derivative diagnostics, each backward-Euler step is computed by the native Newton solver. Implicit differentiation then gives
\begin{equation}
  \label{eq:sm_logistic_derivative}
  \frac{\partial u_{j + 1}}{\partial u_j} =
  \frac{1}{1 - h(1 - 2u_{j + 1})}.
\end{equation}
Multiplying these factors evaluates the chunk sensitivity at the computed step solutions, without differentiating the iteration-dependent stopping branch of Newton's method. The closed-form logistic root is used only as an independent verification of the native solver.

\section{Additional scaling diagnostics}
\label{sec:scaling}

The article's three scaling tables contain the standard and unweighted logistic comparisons and the four-rule chaotic comparisons at $\xi = 100$. \Cref{tab:scaling_lorenz} adds Lorenz-63 at $\xi = 10$, and \Cref{tab:scaling_short_horizon} pairs interface tracking with marginal distances at $T = 4, 8, 16$. These horizons are diagnostic choices, not a claimed predictability threshold. The periodic control retains its full scaling table in \Cref{tab:scaling_periodic}. Logistic and periodic controls report unweighted interface error $\|e\|_2$; chaotic scaling comparisons report dense mean marginal Wasserstein distances against the serial fine reference, with coarse-only as a baseline. The Lyapunov-weight runs prescribe $\lambda = 0.9056$ and 1.68 for Lorenz-63 and Lorenz-96; the scripts do not estimate these exponents. Status C means criterion met, E exact completion, B budget exhausted, and F nonfinite failure. Model speedup charges fine batches and sequential coarse calls, including exact-prefix reuse (\Cref{eq:cost}).

For every chaotic scaling run, we sample the native fine trajectories and the coarse-only trajectory on the common grid $t_j = jT/M$, $j = 1, \ldots, M$, with $M = T/0.02$ for Lorenz-63 and $M = T/0.05$ for Lorenz-96. These ratios are integers on all tested grids. Native linear interpolation evaluates times between stored Runge--Kutta steps; at chunk boundaries the completed left chunk supplies the value. The accepted trajectory is sampled chunk by chunk, without interpolating across interface jumps. The statistic averages the empirical one-dimensional $W_1$ over all coordinates, giving $\overline{W}_1$ as in \Cref{eq:marginal_wasserstein}; the coarse-only rows apply it to the coarse-only trajectory against the same fine samples. This grid retains the prescribed starts and their transients, without spin-up or replication. Short-horizon distances are well defined but may mostly describe the initial transient. Increasing sample density does not lengthen the trajectories or supply independent samples. A small finite-horizon marginal distance neither implies small pointwise error nor establishes agreement of invariant measures. Replicated, spun-up comparisons and their uncertainty intervals remain in \Cref{tab:w1_baseline,sec:statistics_full}.

A sampling refinement at $N = 16$, $\Delta T = 1$, $\xi = 100$ checks the distinction between a stated empirical measure and a continuous-time occupation measure. For Lorenz-63, reducing the interval from 0.02 to 0.0025 changes the secant distance from 0.0312 to 0.0163, while Lyapunov changes from $1.37 \times 10^{-5}$ to $1.36 \times 10^{-5}$ and coarse-only stays at 1.36. For Lorenz-96, reducing 0.05 to 0.00625 changes the corresponding three distances from 0.425, 0.0175, and 0.654 to 0.423, 0.0164, and 0.651. The ordering is unchanged in these checks, but the tabulated numbers describe the specified sampling grids; we do not claim sampling convergence from them.

The main scaling tables select $N = 2, 16, 128$ to leave room for the proofs and independent functional tests. \Cref{tab:scaling_logistic_extra,tab:scaling_lorenz_extra,tab:scaling_lorenz96_extra} retain every omitted row at $N = 4, 8, 32, 64$, including logistic pointwise errors and the chaotic coarse-only baselines.

\Needspace{20\baselineskip}\begingroup\small\setlength{\tabcolsep}{1.5pt}\renewcommand{\arraystretch}{1.15}\setlength{\LTleft}{0pt}\setlength{\LTright}{0pt}
\begin{longtable}{@{\extracolsep{\fill}}ccccccc@{\hspace{8pt}}cccc@{}}
\caption{\rev{Logistic: remaining strong- and weak-scaling rows at $N = 4, 8, 32, 64$. The main article gives $N = 2, 16, 128$; together these tables retain the complete grid. Both $\xi = 10, 100$ and both stopping rules are retained. Errors are unweighted interface norms.}}\label{tab:scaling_logistic_extra}\\
\toprule
& & & \multicolumn{4}{c}{Strong scaling ($T=128$)} & \multicolumn{4}{c}{Weak scaling ($\Delta T=1$)}\\
\cmidrule(lr){4-7}\cmidrule(lr){8-11}
$\xi$ & Rule & $N$ & $K_F$ & Stop & $\mathcal S$ & $\|e\|_2$ & $K_F$ & Stop & $\mathcal S$ & $\|e\|_2$\\
\midrule\endfirsthead
\multicolumn{11}{l}{\tablename~\thetable\ (continued)}\\
\toprule
& & & \multicolumn{4}{c}{Strong scaling ($T=128$)} & \multicolumn{4}{c}{Weak scaling ($\Delta T=1$)}\\
\cmidrule(lr){4-7}\cmidrule(lr){8-11}
$\xi$ & Rule & $N$ & $K_F$ & Stop & $\mathcal S$ & $\|e\|_2$ & $K_F$ & Stop & $\mathcal S$ & $\|e\|_2$\\
\midrule\endhead
\multicolumn{11}{r}{Continued on next page}\\\endfoot
\bottomrule\endlastfoot
\multirow{4}{*}{$10$} & \multirow{4}{*}{standard} & $4$ & $2$ & C & $1.60$ & $7.07\times 10^{-14}$ & $4$ & E & $0.870$ & $0$\\*
 &  & $8$ & $3$ & C & $1.67$ & $9.99\times 10^{-14}$ & $6$ & C & $0.920$ & $3.33\times 10^{-16}$\\*
 &  & $32$ & $5$ & C & $1.64$ & $2.23\times 10^{-13}$ & $6$ & C & $1.39$ & $2.31\times 10^{-15}$\\*
 &  & $64$ & $6$ & C & $1.51$ & $3.25\times 10^{-13}$ & $6$ & C & $1.51$ & $2.40\times 10^{-13}$\\*
\midrule
\multirow{4}{*}{$10$} & \multirow{4}{*}{$\psi$ ($w=1$)} & $4$ & $1$ & C & $3.08$ & $8.68\times 10^{-14}$ & $4$ & E & $0.870$ & $0$\\*
 &  & $8$ & $2$ & C & $2.42$ & $1.12\times 10^{-13}$ & $4$ & C & $1.29$ & $3.47\times 10^{-12}$\\*
 &  & $32$ & $3$ & C & $2.67$ & $2.09\times 10^{-11}$ & $4$ & C & $2.03$ & $6.18\times 10^{-12}$\\*
 &  & $64$ & $4$ & C & $2.24$ & $1.14\times 10^{-12}$ & $4$ & C & $2.24$ & $6.19\times 10^{-12}$\\*
\midrule
\multirow{4}{*}{$100$} & \multirow{4}{*}{standard} & $4$ & $2$ & C & $1.95$ & $7.71\times 10^{-14}$ & $4$ & E & $0.985$ & $0$\\*
 &  & $8$ & $3$ & C & $2.52$ & $1.12\times 10^{-13}$ & $8$ & E & $0.966$ & $0$\\*
 &  & $32$ & $6$ & C & $4.15$ & $2.59\times 10^{-13}$ & $8$ & C & $3.14$ & $4.66\times 10^{-14}$\\*
 &  & $64$ & $8$ & C & $5.02$ & $3.57\times 10^{-13}$ & $8$ & C & $5.02$ & $2.68\times 10^{-13}$\\*
\midrule
\multirow{4}{*}{$100$} & \multirow{4}{*}{$\psi$ ($w=1$)} & $4$ & $1$ & C & $3.88$ & $3.03\times 10^{-13}$ & $4$ & E & $0.985$ & $0$\\*
 &  & $8$ & $2$ & C & $3.76$ & $1.54\times 10^{-13}$ & $6$ & C & $1.28$ & $5.48\times 10^{-12}$\\*
 &  & $32$ & $5$ & C & $4.96$ & $4.65\times 10^{-13}$ & $7$ & C & $3.57$ & $1.70\times 10^{-12}$\\*
 &  & $64$ & $6$ & C & $6.65$ & $2.73\times 10^{-12}$ & $7$ & C & $5.71$ & $1.72\times 10^{-12}$\\*
\end{longtable}\endgroup

\Needspace{20\baselineskip}\begingroup\small\setlength{\tabcolsep}{1.5pt}\renewcommand{\arraystretch}{1.15}\setlength{\LTleft}{0pt}\setlength{\LTright}{0pt}
\begin{longtable}{@{\extracolsep{\fill}}cccccc@{\hspace{8pt}}cccc@{}}
\caption{\rev{Lorenz-63: remaining strong- and weak-scaling rows at $N = 4, 8, 32, 64$. The main article gives $N = 2, 16, 128$; together these tables retain the complete grid. Here $\xi = 100$; distances use the dense mean marginal Wasserstein statistic.}}\label{tab:scaling_lorenz_extra}\\
\toprule
& & \multicolumn{4}{c}{Strong scaling ($T=128$)} & \multicolumn{4}{c}{Weak scaling ($\Delta T=1$)}\\
\cmidrule(lr){3-6}\cmidrule(lr){7-10}
$N$ & Rule & $K_F$ & Stop & $\mathcal S$ & $\overline{W}_1$ & $K_F$ & Stop & $\mathcal S$ & $\overline{W}_1$\\
\midrule\endfirsthead
\multicolumn{10}{l}{\tablename~\thetable\ (continued)}\\
\toprule
& & \multicolumn{4}{c}{Strong scaling ($T=128$)} & \multicolumn{4}{c}{Weak scaling ($\Delta T=1$)}\\
\cmidrule(lr){3-6}\cmidrule(lr){7-10}
$N$ & Rule & $K_F$ & Stop & $\mathcal S$ & $\overline{W}_1$ & $K_F$ & Stop & $\mathcal S$ & $\overline{W}_1$\\
\midrule\endhead
\multicolumn{10}{r}{Continued on next page}\\\endfoot
\bottomrule\endlastfoot
\multirow{5}{*}{$4$} & standard & $4$ & E & $0.985$ & $0$ & $4$ & E & $0.985$ & $0$\\*
 & $\psi$ ($w=1$) & $4$ & E & $0.985$ & $0$ & $3$ & C & $1.31$ & $2.29\times 10^{-13}$\\*
 & $\psi$ ($w=e^{\lambda\Delta T}$) & $1$ & C & $3.88$ & $3.05\times 10^{-1}$ & $3$ & C & $1.31$ & $2.29\times 10^{-13}$\\*
 & $\psi$ (secant $w$) & $4$ & E & $0.985$ & $0$ & $2$ & C & $1.95$ & $5.76\times 10^{-8}$\\*
 & coarse-only &  &  &  & $1.62\times 10^{-1}$ &  &  &  & $2.03\times 10^{-3}$\\
\addlinespace[3pt]
\multirow{5}{*}{$8$} & standard & $8$ & E & $0.966$ & $0$ & $4$ & C & $1.90$ & $4.41\times 10^{-13}$\\*
 & $\psi$ ($w=1$) & $8$ & E & $0.966$ & $0$ & $3$ & C & $2.52$ & $2.03\times 10^{-11}$\\*
 & $\psi$ ($w=e^{\lambda\Delta T}$) & $2$ & C & $3.76$ & $1.04$ & $3$ & C & $2.52$ & $2.03\times 10^{-11}$\\*
 & $\psi$ (secant $w$) & $8$ & E & $0.966$ & $0$ & $2$ & C & $3.76$ & $1.24\times 10^{-6}$\\*
 & coarse-only &  &  &  & $1.62\times 10^{-1}$ &  &  &  & $4.76\times 10^{-2}$\\
\addlinespace[3pt]
\multirow{5}{*}{$32$} & standard & $30$ & C & $0.916$ & $8.54\times 10^{-12}$ & $22$ & C & $1.21$ & $1.72\times 10^{-10}$\\*
 & $\psi$ ($w=1$) & $28$ & C & $0.973$ & $3.61\times 10^{-8}$ & $13$ & C & $1.97$ & $9.90\times 10^{-7}$\\*
 & $\psi$ ($w=e^{\lambda\Delta T}$) & $3$ & C & $8.21$ & $1.04$ & $3$ & C & $8.21$ & $4.81\times 10^{-1}$\\*
 & $\psi$ (secant $w$) & $3$ & C & $8.21$ & $1.04$ & $2$ & C & $12.3$ & $5.05\times 10^{-1}$\\*
 & coarse-only &  &  &  & $1.62\times 10^{-1}$ &  &  &  & $1.41$\\
\addlinespace[3pt]
\multirow{5}{*}{$64$} & standard & $58$ & C & $0.820$ & $8.42\times 10^{-11}$ & $49$ & C & $0.940$ & $1.71\times 10^{-9}$\\*
 & $\psi$ ($w=1$) & $52$ & C & $0.895$ & $6.89\times 10^{-6}$ & $34$ & C & $1.28$ & $1.40\times 10^{-3}$\\*
 & $\psi$ ($w=e^{\lambda\Delta T}$) & $3$ & C & $13.2$ & $4.42\times 10^{-1}$ & $3$ & C & $13.2$ & $1.22\times 10^{-1}$\\*
 & $\psi$ (secant $w$) & $2$ & C & $19.7$ & $5.45\times 10^{-1}$ & $2$ & C & $19.7$ & $1.65\times 10^{-1}$\\*
 & coarse-only &  &  &  & $1.62\times 10^{-1}$ &  &  &  & $1.22$\\
\end{longtable}\endgroup

\Needspace{20\baselineskip}\begingroup\small\setlength{\tabcolsep}{1.5pt}\renewcommand{\arraystretch}{1.15}\setlength{\LTleft}{0pt}\setlength{\LTright}{0pt}
\begin{longtable}{@{\extracolsep{\fill}}cccccc@{\hspace{8pt}}cccc@{}}
\caption{\rev{Lorenz-96: remaining strong- and weak-scaling rows at $N = 4, 8, 32, 64$. The main article gives $N = 2, 16, 128$; together these tables retain the complete grid. Here $\xi = 100$; distances use the dense mean marginal Wasserstein statistic.}}\label{tab:scaling_lorenz96_extra}\\
\toprule
& & \multicolumn{4}{c}{Strong scaling ($T=128$)} & \multicolumn{4}{c}{Weak scaling ($\Delta T=1$)}\\
\cmidrule(lr){3-6}\cmidrule(lr){7-10}
$N$ & Rule & $K_F$ & Stop & $\mathcal S$ & $\overline{W}_1$ & $K_F$ & Stop & $\mathcal S$ & $\overline{W}_1$\\
\midrule\endfirsthead
\multicolumn{10}{l}{\tablename~\thetable\ (continued)}\\
\toprule
& & \multicolumn{4}{c}{Strong scaling ($T=128$)} & \multicolumn{4}{c}{Weak scaling ($\Delta T=1$)}\\
\cmidrule(lr){3-6}\cmidrule(lr){7-10}
$N$ & Rule & $K_F$ & Stop & $\mathcal S$ & $\overline{W}_1$ & $K_F$ & Stop & $\mathcal S$ & $\overline{W}_1$\\
\midrule\endhead
\multicolumn{10}{r}{Continued on next page}\\\endfoot
\bottomrule\endlastfoot
\multirow{5}{*}{$4$} & standard & $4$ & E & $0.985$ & $0$ & $4$ & E & $0.985$ & $0$\\*
 & $\psi$ ($w=1$) & $4$ & E & $0.985$ & $0$ & $3$ & C & $1.31$ & $1.69\times 10^{-11}$\\*
 & $\psi$ ($w=e^{\lambda\Delta T}$) & $1$ & C & $3.88$ & $2.30\times 10^{-1}$ & $3$ & C & $1.31$ & $1.69\times 10^{-11}$\\*
 & $\psi$ (secant $w$) & $4$ & E & $0.985$ & $0$ & $2$ & C & $1.95$ & $2.82\times 10^{-6}$\\*
 & coarse-only &  &  &  & $2.20\times 10^{-1}$ &  &  &  & $1.06\times 10^{-2}$\\
\addlinespace[3pt]
\multirow{5}{*}{$8$} & standard & $8$ & E & $0.966$ & $0$ & $5$ & C & $1.52$ & $3.74\times 10^{-13}$\\*
 & $\psi$ ($w=1$) & $8$ & E & $0.966$ & $0$ & $4$ & C & $1.90$ & $1.93\times 10^{-12}$\\*
 & $\psi$ ($w=e^{\lambda\Delta T}$) & $1$ & C & $7.48$ & $2.44\times 10^{-1}$ & $3$ & C & $2.52$ & $1.76\times 10^{-8}$\\*
 & $\psi$ (secant $w$) & $8$ & E & $0.966$ & $0$ & $2$ & C & $3.76$ & $2.08\times 10^{-3}$\\*
 & coarse-only &  &  &  & $2.20\times 10^{-1}$ &  &  &  & $4.99\times 10^{-1}$\\
\addlinespace[3pt]
\multirow{5}{*}{$32$} & standard & $32$ & E & $0.866$ & $0$ & $26$ & C & $1.04$ & $1.24\times 10^{-11}$\\*
 & $\psi$ ($w=1$) & $32$ & E & $0.866$ & $0$ & $19$ & C & $1.38$ & $1.94\times 10^{-5}$\\*
 & $\psi$ ($w=e^{\lambda\Delta T}$) & $3$ & C & $8.21$ & $2.68\times 10^{-1}$ & $3$ & C & $8.21$ & $4.16\times 10^{-1}$\\*
 & $\psi$ (secant $w$) & $3$ & C & $8.21$ & $2.68\times 10^{-1}$ & $2$ & C & $12.3$ & $4.99\times 10^{-1}$\\*
 & coarse-only &  &  &  & $2.20\times 10^{-1}$ &  &  &  & $4.36\times 10^{-1}$\\
\addlinespace[3pt]
\multirow{5}{*}{$64$} & standard & $61$ & C & $0.789$ & $1.32\times 10^{-12}$ & $59$ & C & $0.810$ & $6.01\times 10^{-12}$\\*
 & $\psi$ ($w=1$) & $57$ & C & $0.832$ & $1.01\times 10^{-6}$ & $51$ & C & $0.909$ & $4.69\times 10^{-6}$\\*
 & $\psi$ ($w=e^{\lambda\Delta T}$) & $2$ & C & $19.7$ & $2.62\times 10^{-1}$ & $3$ & C & $13.2$ & $3.79\times 10^{-1}$\\*
 & $\psi$ (secant $w$) & $2$ & C & $19.7$ & $2.62\times 10^{-1}$ & $2$ & C & $19.7$ & $4.36\times 10^{-1}$\\*
 & coarse-only &  &  &  & $2.20\times 10^{-1}$ &  &  &  & $3.71\times 10^{-1}$\\
\end{longtable}\endgroup

\Needspace{20\baselineskip}\begingroup\small\setlength{\tabcolsep}{1.5pt}\renewcommand{\arraystretch}{1.15}\setlength{\LTleft}{0pt}\setlength{\LTright}{0pt}
\begin{longtable}{@{\extracolsep{\fill}}cccccc@{\hspace{8pt}}cccc@{}}
\caption{\rev{Lorenz-63 at $\xi=10$: additional strong- and weak-scaling records; $\xi=100$ is in \Cref{tab:scaling_lorenz_main}. Dense mean marginal Wasserstein distances $\overline{W}_1$ of the accepted and coarse-only trajectories use the same serial fine reference and the single prescribed start without spin-up. Sampling and work conventions are stated in \Cref{sec:scaling}. The scaling grids share $N=T=128$, so those strong- and weak-scaling rows are the same run.}}\label{tab:scaling_lorenz}\\
\toprule
& & \multicolumn{4}{c}{Strong scaling ($T=128$)} & \multicolumn{4}{c}{Weak scaling ($\Delta T=1$)}\\
\cmidrule(lr){3-6}\cmidrule(lr){7-10}
$N$ & Rule & $K_F$ & Stop & $\mathcal S$ & $\overline{W}_1$ & $K_F$ & Stop & $\mathcal S$ & $\overline{W}_1$\\
\midrule\endfirsthead
\multicolumn{10}{l}{\tablename~\thetable\ (continued)}\\
\toprule
& & \multicolumn{4}{c}{Strong scaling ($T=128$)} & \multicolumn{4}{c}{Weak scaling ($\Delta T=1$)}\\
\cmidrule(lr){3-6}\cmidrule(lr){7-10}
$N$ & Rule & $K_F$ & Stop & $\mathcal S$ & $\overline{W}_1$ & $K_F$ & Stop & $\mathcal S$ & $\overline{W}_1$\\
\midrule\endhead
\multicolumn{10}{r}{Continued on next page}\\\endfoot
\bottomrule\endlastfoot
\multirow{5}{*}{$2$} & standard & $2$ & E & $0.952$ & $0$ & $2$ & E & $0.952$ & $0$\\*
 & $\psi$ ($w=1$) & $2$ & E & $0.952$ & $0$ & $2$ & E & $0.952$ & $0$\\*
 & $\psi$ ($w=e^{\lambda\Delta T}$) & $1$ & C & $1.82$ & $6.72\times 10^{-1}$ & $2$ & E & $0.952$ & $0$\\*
 & $\psi$ (secant $w$) & $2$ & E & $0.952$ & $0$ & $2$ & E & $0.952$ & $0$\\*
 & coarse-only &  &  &  & $3.98\times 10^{-1}$ &  &  &  & $4.53\times 10^{-8}$\\
\addlinespace[3pt]
\multirow{5}{*}{$4$} & standard & $4$ & E & $0.870$ & $0$ & $3$ & C & $1.11$ & $7.76\times 10^{-14}$\\*
 & $\psi$ ($w=1$) & $4$ & E & $0.870$ & $0$ & $2$ & C & $1.60$ & $3.03\times 10^{-13}$\\*
 & $\psi$ ($w=e^{\lambda\Delta T}$) & $1$ & C & $3.08$ & $4.06\times 10^{-1}$ & $2$ & C & $1.60$ & $3.03\times 10^{-13}$\\*
 & $\psi$ (secant $w$) & $4$ & E & $0.870$ & $0$ & $2$ & C & $1.60$ & $3.03\times 10^{-13}$\\*
 & coarse-only &  &  &  & $3.98\times 10^{-1}$ &  &  &  & $9.07\times 10^{-8}$\\
\addlinespace[3pt]
\multirow{5}{*}{$8$} & standard & $8$ & E & $0.741$ & $0$ & $3$ & C & $1.67$ & $5.36\times 10^{-12}$\\*
 & $\psi$ ($w=1$) & $8$ & E & $0.741$ & $0$ & $2$ & C & $2.42$ & $5.69\times 10^{-12}$\\*
 & $\psi$ ($w=e^{\lambda\Delta T}$) & $2$ & C & $2.42$ & $7.77\times 10^{-1}$ & $2$ & C & $2.42$ & $5.69\times 10^{-12}$\\*
 & $\psi$ (secant $w$) & $3$ & C & $1.67$ & $6.78\times 10^{-1}$ & $2$ & C & $2.42$ & $5.69\times 10^{-12}$\\*
 & coarse-only &  &  &  & $3.98\times 10^{-1}$ &  &  &  & $2.11\times 10^{-6}$\\
\addlinespace[3pt]
\multirow{5}{*}{$16$} & standard & $16$ & E & $0.571$ & $0$ & $5$ & C & $1.39$ & $3.75\times 10^{-11}$\\*
 & $\psi$ ($w=1$) & $14$ & C & $0.618$ & $5.93\times 10^{-7}$ & $2$ & C & $3.27$ & $3.19\times 10^{-9}$\\*
 & $\psi$ ($w=e^{\lambda\Delta T}$) & $2$ & C & $3.27$ & $5.03\times 10^{-1}$ & $2$ & C & $3.27$ & $3.19\times 10^{-9}$\\*
 & $\psi$ (secant $w$) & $2$ & C & $3.27$ & $5.03\times 10^{-1}$ & $2$ & C & $3.27$ & $3.19\times 10^{-9}$\\*
 & coarse-only &  &  &  & $3.98\times 10^{-1}$ &  &  &  & $1.19\times 10^{-3}$\\
\addlinespace[3pt]
\multirow{5}{*}{$32$} & standard & $31$ & C & $0.397$ & $1.51\times 10^{-12}$ & $18$ & C & $0.547$ & $8.80\times 10^{-9}$\\*
 & $\psi$ ($w=1$) & $25$ & C & $0.441$ & $3.24\times 10^{-3}$ & $3$ & C & $2.67$ & $7.73\times 10^{-3}$\\*
 & $\psi$ ($w=e^{\lambda\Delta T}$) & $1$ & C & $7.80$ & $3.98\times 10^{-1}$ & $2$ & C & $3.95$ & $8.01\times 10^{-3}$\\*
 & $\psi$ (secant $w$) & $2$ & C & $3.95$ & $1.93\times 10^{-1}$ & $2$ & C & $3.95$ & $8.01\times 10^{-3}$\\*
 & coarse-only &  &  &  & $3.98\times 10^{-1}$ &  &  &  & $2.14\times 10^{-1}$\\
\addlinespace[3pt]
\multirow{5}{*}{$64$} & standard & $58$ & C & $0.248$ & $2.34\times 10^{-10}$ & $49$ & C & $0.267$ & $1.04\times 10^{-9}$\\*
 & $\psi$ ($w=1$) & $44$ & C & $0.282$ & $2.21\times 10^{-1}$ & $28$ & C & $0.384$ & $4.20\times 10^{-2}$\\*
 & $\psi$ ($w=e^{\lambda\Delta T}$) & $1$ & C & $8.77$ & $3.98\times 10^{-1}$ & $1$ & C & $8.77$ & $2.63\times 10^{-1}$\\*
 & $\psi$ (secant $w$) & $2$ & C & $4.41$ & $7.79\times 10^{-1}$ & $2$ & C & $4.41$ & $1.62\times 10^{-1}$\\*
 & coarse-only &  &  &  & $3.98\times 10^{-1}$ &  &  &  & $2.63\times 10^{-1}$\\
\addlinespace[3pt]
\multirow{5}{*}{$128$} & standard & $113$ & C & $0.140$ & $3.59\times 10^{-10}$ & $113$ & C & $0.140$ & $3.59\times 10^{-10}$\\*
 & $\psi$ ($w=1$) & $87$ & C & $0.157$ & $1.63\times 10^{-1}$ & $87$ & C & $0.157$ & $1.63\times 10^{-1}$\\*
 & $\psi$ ($w=e^{\lambda\Delta T}$) & $1$ & C & $9.34$ & $3.98\times 10^{-1}$ & $1$ & C & $9.34$ & $3.98\times 10^{-1}$\\*
 & $\psi$ (secant $w$) & $2$ & C & $4.69$ & $8.14\times 10^{-1}$ & $2$ & C & $4.69$ & $8.14\times 10^{-1}$\\*
 & coarse-only &  &  &  & $3.98\times 10^{-1}$ &  &  &  & $3.98\times 10^{-1}$\\
\end{longtable}\endgroup

\Needspace{20\baselineskip}\begingroup\small\setlength{\tabcolsep}{10pt}\renewcommand{\arraystretch}{1.15}
\begin{longtable}{@{}cccc@{}}
\caption{\rev{Short-horizon runs from the weak-scaling grid: $\Delta T=1$, $\xi=100$, and the prescribed starts without spin-up; $K_F$ and the coarse-only baselines are in the article's \Cref{tab:scaling_lorenz_main,tab:scaling_lorenz96_main} and in \Cref{tab:scaling_lorenz_extra,tab:scaling_lorenz96_extra}. The unweighted interface error is $\|e\|_2=(\sum_{n=0}^{N}\|U_n-U_n^*\|_2^2)^{1/2}$, including the terminal fine output; $U^*$ is the serial fine discrete reference. Dense mean marginal Wasserstein distance $\overline{W}_1$ uses the sampling protocol in \Cref{sec:scaling}. The two measures test interface tracking and finite-time marginal agreement, respectively.}}\label{tab:scaling_short_horizon}\\
\toprule
$T=N$ & Rule & $\|e\|_2$ & $\overline{W}_1$\\
\midrule\endfirsthead
\multicolumn{4}{l}{\tablename~\thetable\ (continued)}\\
\toprule
$T=N$ & Rule & $\|e\|_2$ & $\overline{W}_1$\\
\midrule\endhead
\multicolumn{4}{r}{Continued on next page}\\\endfoot
\bottomrule\endlastfoot
\multicolumn{4}{c}{\textit{Lorenz-63}}\\*
\multirow{4}{*}{$4$} & standard & $0$ & $0$\\*
 & $\psi$ ($w=1$) & $1.37\times 10^{-12}$ & $2.29\times 10^{-13}$\\*
 & $\psi$ ($w=e^{\lambda\Delta T}$) & $1.37\times 10^{-12}$ & $2.29\times 10^{-13}$\\*
 & $\psi$ (secant $w$) & $9.15\times 10^{-7}$ & $5.76\times 10^{-8}$\\
\addlinespace[3pt]
\multirow{4}{*}{$8$} & standard & $4.93\times 10^{-12}$ & $4.41\times 10^{-13}$\\*
 & $\psi$ ($w=1$) & $9.80\times 10^{-11}$ & $2.03\times 10^{-11}$\\*
 & $\psi$ ($w=e^{\lambda\Delta T}$) & $9.80\times 10^{-11}$ & $2.03\times 10^{-11}$\\*
 & $\psi$ (secant $w$) & $2.45\times 10^{-5}$ & $1.24\times 10^{-6}$\\
\addlinespace[3pt]
\multirow{4}{*}{$16$} & standard & $3.88\times 10^{-9}$ & $1.31\times 10^{-10}$\\*
 & $\psi$ ($w=1$) & $4.73\times 10^{-8}$ & $1.30\times 10^{-9}$\\*
 & $\psi$ ($w=e^{\lambda\Delta T}$) & $4.48\times 10^{-4}$ & $1.37\times 10^{-5}$\\*
 & $\psi$ (secant $w$) & $1.48$ & $3.12\times 10^{-2}$\\
\midrule
\multicolumn{4}{c}{\textit{Lorenz-96}}\\*
\multirow{4}{*}{$4$} & standard & $0$ & $0$\\*
 & $\psi$ ($w=1$) & $1.62\times 10^{-9}$ & $1.69\times 10^{-11}$\\*
 & $\psi$ ($w=e^{\lambda\Delta T}$) & $1.62\times 10^{-9}$ & $1.69\times 10^{-11}$\\*
 & $\psi$ (secant $w$) & $2.29\times 10^{-4}$ & $2.82\times 10^{-6}$\\
\addlinespace[3pt]
\multirow{4}{*}{$8$} & standard & $4.75\times 10^{-11}$ & $3.74\times 10^{-13}$\\*
 & $\psi$ ($w=1$) & $2.70\times 10^{-10}$ & $1.93\times 10^{-12}$\\*
 & $\psi$ ($w=e^{\lambda\Delta T}$) & $2.76\times 10^{-6}$ & $1.76\times 10^{-8}$\\*
 & $\psi$ (secant $w$) & $2.96\times 10^{-1}$ & $2.08\times 10^{-3}$\\
\addlinespace[3pt]
\multirow{4}{*}{$16$} & standard & $2.19\times 10^{-8}$ & $7.63\times 10^{-11}$\\*
 & $\psi$ ($w=1$) & $5.80\times 10^{-6}$ & $2.09\times 10^{-8}$\\*
 & $\psi$ ($w=e^{\lambda\Delta T}$) & $5.69$ & $1.75\times 10^{-2}$\\*
 & $\psi$ (secant $w$) & $73.5$ & $4.25\times 10^{-1}$\\
\end{longtable}\endgroup

\Cref{fig:SerialWork} displays the fine-work depth of the Lorenz-63 runs in \Cref{fig:KvsTLorenz}; its separation from the coarse cost explains why small batch counts alone do not prove strong scaling.
\begin{figure}[H]
  \centering
  \includegraphics[width=\textwidth]{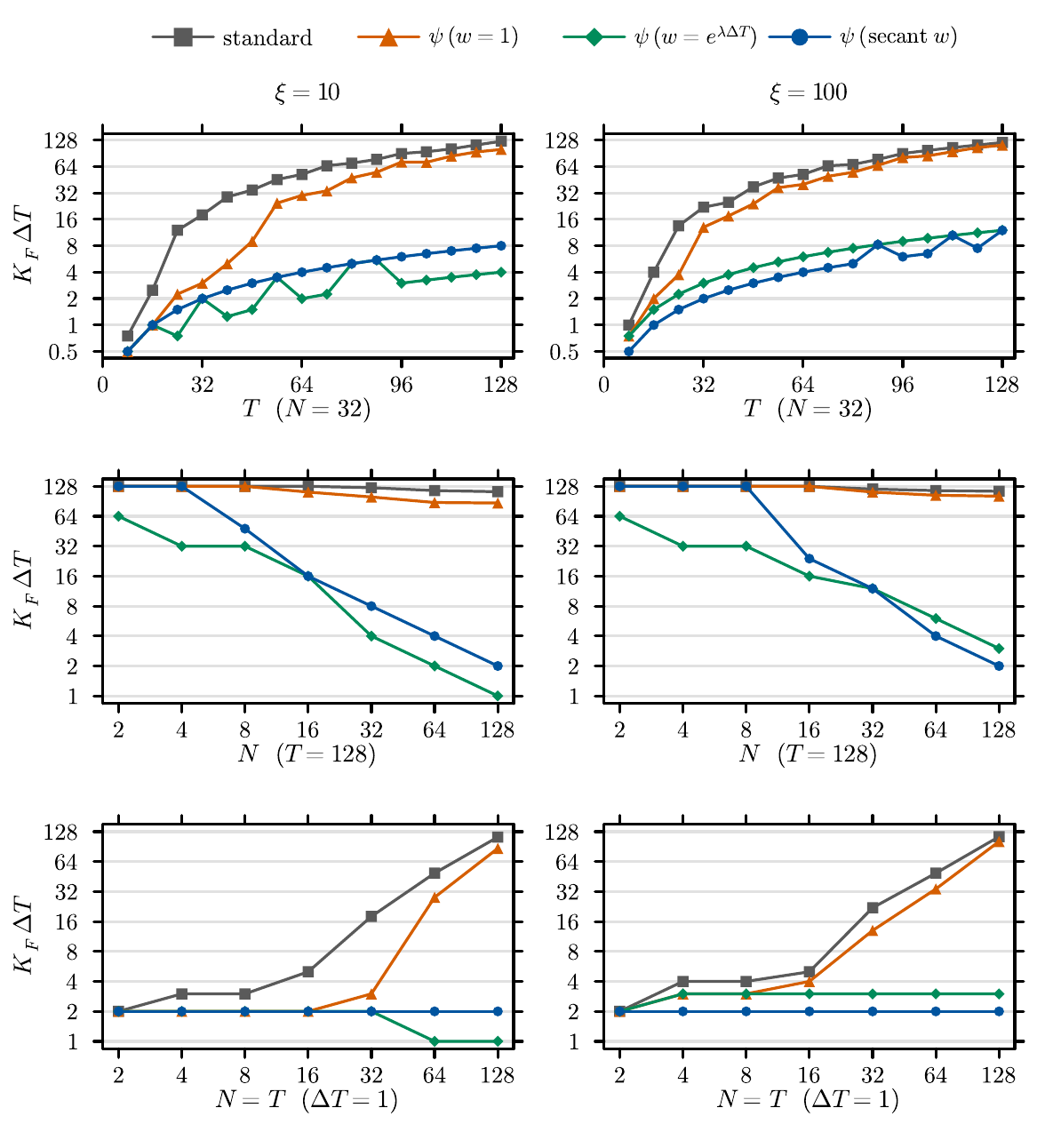}
  \caption{Fine-work depth $K_F\Delta T$ for the same Lorenz-63 records used in the main count figure. It is only one contribution to serial work; coarse propagation is charged separately by \Cref{eq:cost}. Axes and rule labels follow \Cref{fig:KvsTLorenz}.}
  \label{fig:SerialWork}
\end{figure}

\section{Repeated statistical comparisons}
\label{sec:statistics_full}

Using the replicate construction and paired bootstrap of the main text, we supplement marginal $\overline{W}_1$ with sliced distances over 16 fixed random directions, coordinate means, mean energy, coupled and neighbour observables, and temporal comparisons at lag 0.5. These probe discrepancies that marginal distances can miss. Autocorrelation tests temporal dependence but does not determine absolute phase: a constant phase shift leaves a periodic signal's period-averaged autocorrelation unchanged. \Cref{tab:statistics_full} presents both paired comparisons, distinguishing improvement over coarse-only from distance to an independent fine baseline. \Cref{tab:observables_full} tests whether the marginal conclusions extend to joint, moment, and temporal diagnostics; the independent-fine rows set their sampling scale.

The paired intervals address variation across prepared starts, which the single-run scaling tables cannot quantify; an interval containing zero does not establish equivalence. The additional observables also limit the interpretation of marginal agreement. At $N = 16$, both Lorenz-63 weighted rules improve on coarse-only in every listed diagnostic. For Lorenz-96 at $N = 32$, however, the neighbour-product error is 0.407 for either weighted rule versus 0.344 for coarse-only, despite the marginal Wasserstein gain. These are descriptive comparisons of replicate means, not simultaneous significance tests.

\Needspace{20\baselineskip}\begingroup\small\setlength{\tabcolsep}{1.5pt}\renewcommand{\arraystretch}{1.15}\setlength{\LTleft}{0pt}\setlength{\LTright}{0pt}
\begin{longtable}{@{\extracolsep{\fill}}ccccc@{}}
\caption{\rev{Marginal Wasserstein distances $\overline{W}_1$ to the reference fine trajectory, as means over eight spun-up replicate pairs. The intervals are bootstrap 95\% intervals for the mean paired difference from the coarse-only distance and from the fine-baseline distance, which share the reference fine trajectory within a replicate; negative intervals indicate a smaller distance than the named comparator.}}\label{tab:statistics_full}\\
\toprule
& & & \multicolumn{2}{c}{95\% paired CI}\\
\cmidrule(lr){4-5}
$N$ & Trajectory & $\overline{W}_1$ & vs coarse-only & vs fine baseline\\
\midrule\endfirsthead
\multicolumn{5}{l}{\tablename~\thetable\ (continued)}\\
\toprule
& & & \multicolumn{2}{c}{95\% paired CI}\\
\cmidrule(lr){4-5}
$N$ & Trajectory & $\overline{W}_1$ & vs coarse-only & vs fine baseline\\
\midrule\endhead
\multicolumn{5}{r}{Continued on next page}\\\endfoot
\bottomrule\endlastfoot
\multicolumn{5}{c}{\textit{Lorenz-63}}\\*
\multirow{4}{*}{$16$} & fine baseline & $2.00$ &  & \\*
 & coarse-only & $6.56\times 10^{-1}$ &  & [$-2.27$, $-0.443$]\\*
 & $\psi$ ($w=e^{\lambda\Delta T}$) & $2.72\times 10^{-4}$ & [$-1.01$, $-0.350$] & [$-2.85$, $-1.22$]\\*
 & $\psi$ (secant $w$) & $1.37\times 10^{-3}$ & [$-1.01$, $-0.349$] & [$-2.85$, $-1.22$]\\
\addlinespace[3pt]
\multirow{4}{*}{$32$} & fine baseline & $1.11$ &  & \\*
 & coarse-only & $8.02\times 10^{-1}$ &  & [$-1.03$, $0.312$]\\*
 & $\psi$ ($w=e^{\lambda\Delta T}$) & $6.38\times 10^{-1}$ & [$-0.748$, $0.438$] & [$-1.11$, $0.284$]\\*
 & $\psi$ (secant $w$) & $6.70\times 10^{-1}$ & [$-0.675$, $0.444$] & [$-1.09$, $0.299$]\\
\addlinespace[3pt]
\multirow{4}{*}{$64$} & fine baseline & $7.64\times 10^{-1}$ &  & \\*
 & coarse-only & $6.51\times 10^{-1}$ &  & [$-0.761$, $0.531$]\\*
 & $\psi$ ($w=e^{\lambda\Delta T}$) & $8.87\times 10^{-1}$ & [$-0.0434$, $0.551$] & [$-0.421$, $0.698$]\\*
 & $\psi$ (secant $w$) & $8.87\times 10^{-1}$ & [$-0.0434$, $0.551$] & [$-0.421$, $0.698$]\\
\addlinespace[3pt]
\multirow{4}{*}{$128$} & fine baseline & $6.75\times 10^{-1}$ &  & \\*
 & coarse-only & $5.51\times 10^{-1}$ &  & [$-0.424$, $0.158$]\\*
 & $\psi$ ($w=e^{\lambda\Delta T}$) & $6.15\times 10^{-1}$ & [$-0.352$, $0.554$] & [$-0.419$, $0.315$]\\*
 & $\psi$ (secant $w$) & $6.15\times 10^{-1}$ & [$-0.352$, $0.554$] & [$-0.419$, $0.315$]\\
\midrule
\multicolumn{5}{c}{\textit{Lorenz-96}}\\*
\multirow{4}{*}{$16$} & fine baseline & $8.56\times 10^{-1}$ &  & \\*
 & coarse-only & $5.38\times 10^{-1}$ &  & [$-0.394$, $-0.245$]\\*
 & $\psi$ ($w=e^{\lambda\Delta T}$) & $1.79\times 10^{-1}$ & [$-0.425$, $-0.303$] & [$-0.728$, $-0.623$]\\*
 & $\psi$ (secant $w$) & $1.79\times 10^{-1}$ & [$-0.425$, $-0.303$] & [$-0.728$, $-0.623$]\\
\addlinespace[3pt]
\multirow{4}{*}{$32$} & fine baseline & $6.02\times 10^{-1}$ &  & \\*
 & coarse-only & $4.69\times 10^{-1}$ &  & [$-0.169$, $-0.0977$]\\*
 & $\psi$ ($w=e^{\lambda\Delta T}$) & $4.20\times 10^{-1}$ & [$-0.0985$, $-0.00490$] & [$-0.237$, $-0.128$]\\*
 & $\psi$ (secant $w$) & $4.20\times 10^{-1}$ & [$-0.0985$, $-0.00490$] & [$-0.237$, $-0.128$]\\
\addlinespace[3pt]
\multirow{4}{*}{$64$} & fine baseline & $4.27\times 10^{-1}$ &  & \\*
 & coarse-only & $3.67\times 10^{-1}$ &  & [$-0.0840$, $-0.0334$]\\*
 & $\psi$ ($w=e^{\lambda\Delta T}$) & $3.65\times 10^{-1}$ & [$-0.0337$, $0.0318$] & [$-0.0949$, $-0.0336$]\\*
 & $\psi$ (secant $w$) & $3.65\times 10^{-1}$ & [$-0.0337$, $0.0318$] & [$-0.0949$, $-0.0336$]\\
\addlinespace[3pt]
\multirow{4}{*}{$128$} & fine baseline & $2.87\times 10^{-1}$ &  & \\*
 & coarse-only & $2.65\times 10^{-1}$ &  & [$-0.0598$, $0.0102$]\\*
 & $\psi$ ($w=e^{\lambda\Delta T}$) & $2.61\times 10^{-1}$ & [$-0.0312$, $0.0268$] & [$-0.0778$, $0.0149$]\\*
 & $\psi$ (secant $w$) & $2.61\times 10^{-1}$ & [$-0.0312$, $0.0268$] & [$-0.0778$, $0.0149$]\\
\end{longtable}\endgroup

\Needspace{20\baselineskip}\begingroup\small\setlength{\tabcolsep}{1.5pt}\renewcommand{\arraystretch}{1.15}\setlength{\LTleft}{0pt}\setlength{\LTright}{0pt}
\begin{longtable}{@{\extracolsep{\fill}}ccccc@{}}
\caption{\rev{Distributional, moment, and temporal diagnostics against the reference fine trajectory. SW is the sliced Wasserstein distance over 16 fixed directions. The remaining columns are absolute errors in mean energy $\|u\|_2^2/d$, mean $x_1$, the mean neighbour product, mean $x_1x_3$, and first-coordinate autocorrelation (AC) at lag 0.5. The neighbour product is $x_1x_2$ for Lorenz-63 and $d^{-1}\sum_i x_ix_{i-1}$ for Lorenz-96. The fine-baseline rows use the independent fine trajectory. Entries are means over eight spun-up replicate pairs; these diagnostic means carry no uncertainty intervals.}}\label{tab:observables_full}\\
\multicolumn{5}{@{}l}{\textbf{(a) Distribution and first moments}}\\
\toprule
& & Distribution & \multicolumn{2}{c}{Moments}\\
\cmidrule(lr){3-3}\cmidrule(lr){4-5}
$N$ & Trajectory & SW & Energy & Mean $x_1$\\
\midrule\endfirsthead
\multicolumn{5}{l}{\tablename~\thetable\ (continued)}\\
\toprule
& & Distribution & \multicolumn{2}{c}{Moments}\\
\cmidrule(lr){3-3}\cmidrule(lr){4-5}
$N$ & Trajectory & SW & Energy & Mean $x_1$\\
\midrule\endhead
\multicolumn{5}{r}{Continued on next page}\\\endfoot
\bottomrule\endlastfoot
\multicolumn{5}{c}{\textit{Lorenz-63}}\\*
\multirow{4}{*}{$16$} & fine baseline & $1.95$ & $2.94$ & $2.54$\\*
 & coarse-only & $6.25\times 10^{-1}$ & $1.75$ & $7.10\times 10^{-1}$\\*
 & $\psi$ ($w=e^{\lambda\Delta T}$) & $2.61\times 10^{-4}$ & $1.08\times 10^{-3}$ & $4.49\times 10^{-5}$\\*
 & $\psi$ (secant $w$) & $1.35\times 10^{-3}$ & $2.25\times 10^{-3}$ & $2.19\times 10^{-4}$\\
\addlinespace[3pt]
\multirow{4}{*}{$32$} & fine baseline & $1.01$ & $1.91$ & $1.30$\\*
 & coarse-only & $8.02\times 10^{-1}$ & $1.24$ & $9.03\times 10^{-1}$\\*
 & $\psi$ ($w=e^{\lambda\Delta T}$) & $6.40\times 10^{-1}$ & $1.10$ & $7.89\times 10^{-1}$\\*
 & $\psi$ (secant $w$) & $6.74\times 10^{-1}$ & $1.07$ & $8.31\times 10^{-1}$\\
\addlinespace[3pt]
\multirow{4}{*}{$64$} & fine baseline & $7.51\times 10^{-1}$ & $1.20$ & $9.18\times 10^{-1}$\\*
 & coarse-only & $6.40\times 10^{-1}$ & $1.05$ & $7.55\times 10^{-1}$\\*
 & $\psi$ ($w=e^{\lambda\Delta T}$) & $8.84\times 10^{-1}$ & $1.13$ & $1.07$\\*
 & $\psi$ (secant $w$) & $8.84\times 10^{-1}$ & $1.13$ & $1.07$\\
\addlinespace[3pt]
\multirow{4}{*}{$128$} & fine baseline & $6.30\times 10^{-1}$ & $9.79\times 10^{-1}$ & $7.91\times 10^{-1}$\\*
 & coarse-only & $5.40\times 10^{-1}$ & $6.12\times 10^{-1}$ & $6.97\times 10^{-1}$\\*
 & $\psi$ ($w=e^{\lambda\Delta T}$) & $6.06\times 10^{-1}$ & $4.81\times 10^{-1}$ & $7.91\times 10^{-1}$\\*
 & $\psi$ (secant $w$) & $6.06\times 10^{-1}$ & $4.81\times 10^{-1}$ & $7.91\times 10^{-1}$\\
\midrule
\multicolumn{5}{c}{\textit{Lorenz-96}}\\*
\multirow{4}{*}{$16$} & fine baseline & $9.00\times 10^{-1}$ & $7.67\times 10^{-1}$ & $3.88\times 10^{-1}$\\*
 & coarse-only & $5.96\times 10^{-1}$ & $4.03\times 10^{-1}$ & $4.72\times 10^{-1}$\\*
 & $\psi$ ($w=e^{\lambda\Delta T}$) & $2.06\times 10^{-1}$ & $1.22\times 10^{-1}$ & $1.68\times 10^{-1}$\\*
 & $\psi$ (secant $w$) & $2.06\times 10^{-1}$ & $1.22\times 10^{-1}$ & $1.68\times 10^{-1}$\\
\addlinespace[3pt]
\multirow{4}{*}{$32$} & fine baseline & $6.08\times 10^{-1}$ & $5.23\times 10^{-1}$ & $5.09\times 10^{-1}$\\*
 & coarse-only & $4.68\times 10^{-1}$ & $6.34\times 10^{-1}$ & $3.02\times 10^{-1}$\\*
 & $\psi$ ($w=e^{\lambda\Delta T}$) & $4.26\times 10^{-1}$ & $4.77\times 10^{-1}$ & $2.93\times 10^{-1}$\\*
 & $\psi$ (secant $w$) & $4.26\times 10^{-1}$ & $4.77\times 10^{-1}$ & $2.93\times 10^{-1}$\\
\addlinespace[3pt]
\multirow{4}{*}{$64$} & fine baseline & $4.52\times 10^{-1}$ & $2.30\times 10^{-1}$ & $3.05\times 10^{-1}$\\*
 & coarse-only & $3.80\times 10^{-1}$ & $2.97\times 10^{-1}$ & $2.43\times 10^{-1}$\\*
 & $\psi$ ($w=e^{\lambda\Delta T}$) & $3.66\times 10^{-1}$ & $2.23\times 10^{-1}$ & $2.43\times 10^{-1}$\\*
 & $\psi$ (secant $w$) & $3.66\times 10^{-1}$ & $2.23\times 10^{-1}$ & $2.43\times 10^{-1}$\\
\addlinespace[3pt]
\multirow{4}{*}{$128$} & fine baseline & $2.92\times 10^{-1}$ & $1.87\times 10^{-1}$ & $1.91\times 10^{-1}$\\*
 & coarse-only & $2.66\times 10^{-1}$ & $2.02\times 10^{-1}$ & $2.51\times 10^{-1}$\\*
 & $\psi$ ($w=e^{\lambda\Delta T}$) & $2.74\times 10^{-1}$ & $1.91\times 10^{-1}$ & $1.78\times 10^{-1}$\\*
 & $\psi$ (secant $w$) & $2.74\times 10^{-1}$ & $1.91\times 10^{-1}$ & $1.78\times 10^{-1}$\\
\end{longtable}\endgroup
\Needspace{20\baselineskip}\begingroup\small\setlength{\tabcolsep}{1.5pt}\renewcommand{\arraystretch}{1.15}\setlength{\LTleft}{0pt}\setlength{\LTright}{0pt}
\addtocounter{table}{-1}\def\theHtable{\arabic{table}.\number\value{LT@tables}}\begin{longtable}{@{\extracolsep{\fill}}ccccc@{}}
\multicolumn{5}{@{}l}{\textbf{(b) Joint moments and temporal dependence}}\\
\toprule
& & \multicolumn{2}{c}{Moments} & Temporal\\
\cmidrule(lr){3-4}\cmidrule(lr){5-5}
$N$ & Trajectory & Neighbour & $x_1x_3$ & AC\\
\midrule\endfirsthead
\multicolumn{5}{l}{\tablename~\thetable\ (continued)}\\
\toprule
& & \multicolumn{2}{c}{Moments} & Temporal\\
\cmidrule(lr){3-4}\cmidrule(lr){5-5}
$N$ & Trajectory & Neighbour & $x_1x_3$ & AC\\
\midrule\endhead
\multicolumn{5}{r}{Continued on next page}\\\endfoot
\bottomrule\endlastfoot
\multicolumn{5}{c}{\textit{Lorenz-63}}\\*
\multirow{4}{*}{$16$} & fine baseline & $9.67\times 10^{-1}$ & $68.6$ & $9.49\times 10^{-2}$\\*
 & coarse-only & $7.47\times 10^{-1}$ & $18.8$ & $6.28\times 10^{-2}$\\*
 & $\psi$ ($w=e^{\lambda\Delta T}$) & $7.56\times 10^{-4}$ & $1.13\times 10^{-3}$ & $9.00\times 10^{-6}$\\*
 & $\psi$ (secant $w$) & $1.32\times 10^{-3}$ & $4.81\times 10^{-3}$ & $5.02\times 10^{-5}$\\
\addlinespace[3pt]
\multirow{4}{*}{$32$} & fine baseline & $5.84\times 10^{-1}$ & $35.1$ & $1.00\times 10^{-1}$\\*
 & coarse-only & $4.36\times 10^{-1}$ & $24.5$ & $1.08\times 10^{-1}$\\*
 & $\psi$ ($w=e^{\lambda\Delta T}$) & $4.87\times 10^{-1}$ & $21.5$ & $7.19\times 10^{-2}$\\*
 & $\psi$ (secant $w$) & $4.73\times 10^{-1}$ & $22.6$ & $6.71\times 10^{-2}$\\
\addlinespace[3pt]
\multirow{4}{*}{$64$} & fine baseline & $3.64\times 10^{-1}$ & $24.7$ & $6.15\times 10^{-2}$\\*
 & coarse-only & $3.37\times 10^{-1}$ & $20.4$ & $5.99\times 10^{-2}$\\*
 & $\psi$ ($w=e^{\lambda\Delta T}$) & $4.02\times 10^{-1}$ & $29.0$ & $7.59\times 10^{-2}$\\*
 & $\psi$ (secant $w$) & $4.02\times 10^{-1}$ & $29.0$ & $7.59\times 10^{-2}$\\
\addlinespace[3pt]
\multirow{4}{*}{$128$} & fine baseline & $3.17\times 10^{-1}$ & $21.3$ & $7.25\times 10^{-2}$\\*
 & coarse-only & $1.80\times 10^{-1}$ & $18.9$ & $3.76\times 10^{-2}$\\*
 & $\psi$ ($w=e^{\lambda\Delta T}$) & $1.50\times 10^{-1}$ & $21.3$ & $3.73\times 10^{-2}$\\*
 & $\psi$ (secant $w$) & $1.50\times 10^{-1}$ & $21.3$ & $3.73\times 10^{-2}$\\
\midrule
\multicolumn{5}{c}{\textit{Lorenz-96}}\\*
\multirow{4}{*}{$16$} & fine baseline & $3.54\times 10^{-1}$ & $1.29$ & $1.14\times 10^{-1}$\\*
 & coarse-only & $2.37\times 10^{-1}$ & $2.06$ & $7.44\times 10^{-2}$\\*
 & $\psi$ ($w=e^{\lambda\Delta T}$) & $1.01\times 10^{-1}$ & $8.19\times 10^{-1}$ & $2.91\times 10^{-2}$\\*
 & $\psi$ (secant $w$) & $1.01\times 10^{-1}$ & $8.19\times 10^{-1}$ & $2.91\times 10^{-2}$\\
\addlinespace[3pt]
\multirow{4}{*}{$32$} & fine baseline & $2.32\times 10^{-1}$ & $1.60$ & $7.24\times 10^{-2}$\\*
 & coarse-only & $3.44\times 10^{-1}$ & $1.16$ & $8.35\times 10^{-2}$\\*
 & $\psi$ ($w=e^{\lambda\Delta T}$) & $4.07\times 10^{-1}$ & $1.09$ & $8.37\times 10^{-2}$\\*
 & $\psi$ (secant $w$) & $4.07\times 10^{-1}$ & $1.09$ & $8.37\times 10^{-2}$\\
\addlinespace[3pt]
\multirow{4}{*}{$64$} & fine baseline & $1.58\times 10^{-1}$ & $2.98\times 10^{-1}$ & $4.88\times 10^{-2}$\\*
 & coarse-only & $2.11\times 10^{-1}$ & $5.46\times 10^{-1}$ & $6.69\times 10^{-2}$\\*
 & $\psi$ ($w=e^{\lambda\Delta T}$) & $1.29\times 10^{-1}$ & $6.22\times 10^{-1}$ & $4.40\times 10^{-2}$\\*
 & $\psi$ (secant $w$) & $1.29\times 10^{-1}$ & $6.22\times 10^{-1}$ & $4.40\times 10^{-2}$\\
\addlinespace[3pt]
\multirow{4}{*}{$128$} & fine baseline & $1.37\times 10^{-1}$ & $4.32\times 10^{-1}$ & $3.94\times 10^{-2}$\\*
 & coarse-only & $7.82\times 10^{-2}$ & $4.83\times 10^{-1}$ & $5.97\times 10^{-2}$\\*
 & $\psi$ ($w=e^{\lambda\Delta T}$) & $1.66\times 10^{-1}$ & $4.98\times 10^{-1}$ & $3.90\times 10^{-2}$\\*
 & $\psi$ (secant $w$) & $1.66\times 10^{-1}$ & $4.98\times 10^{-1}$ & $3.90\times 10^{-2}$\\
\end{longtable}\endgroup

\Cref{fig:statistics} shows the marginal distributions for a representative Lorenz-63 comparison; it illustrates the distributions whose repeated measurements are tabulated above.
\begin{figure}[H]
  \centering
  \includegraphics[width=\textwidth]{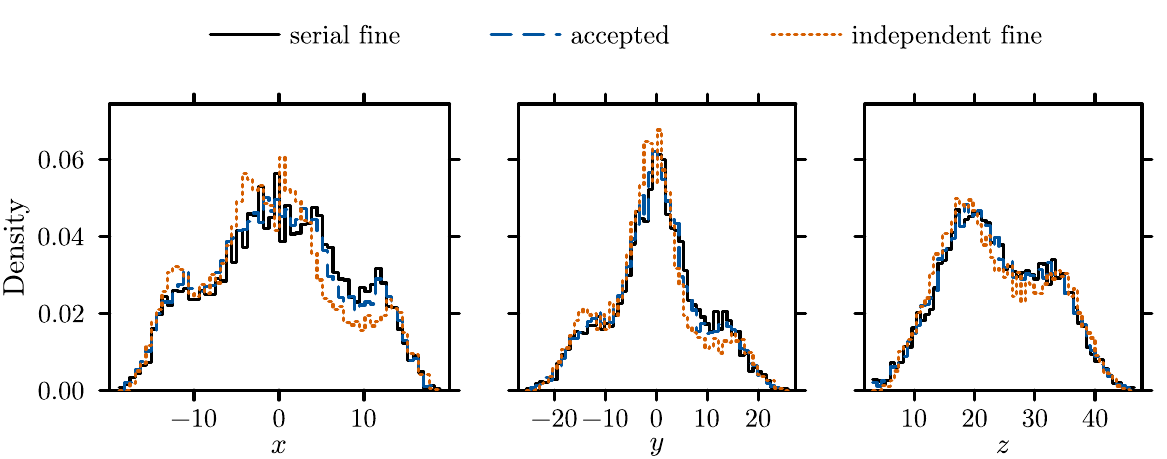}
  \caption{Marginal distributions of $x$, $y$, and $z$ for one Lorenz-63 replicate at $N = 128$: the serial fine trajectory, the trajectory accepted by the secant rule, and an independent fine baseline. The tables quantify agreement for both systems, including coarse-only comparisons and paired intervals.}
  \label{fig:statistics}
\end{figure}
\Cref{tab:prefix_summary} gives the contiguous prefix lengths in \Cref{fig:prefix}, separating the exact prefix from the longer accurate prefix.

\Needspace{20\baselineskip}\begingroup\small\setlength{\tabcolsep}{10pt}\renewcommand{\arraystretch}{1.15}
\begin{longtable}{@{}cccccc@{}}
\caption{\rev{Prefix lengths for Lorenz-63 at $N = 128$ and $H = 0.01$: $n_\eta$ is the largest $n$ such that $\|U_m-U_m^*\|\le\eta$ for every $0\le m\le n$, so $n_0$ counts the exactly reproduced interfaces. The prescribed start is $(2,3,-14)$; the replicate start is the first prepared $N=128$ start of the control with a poor coarse solver. A failed run (F) accepts no interfaces.}}\label{tab:prefix_summary}\\
\toprule
& \multicolumn{2}{c}{Stopping} & \multicolumn{3}{c}{Prefix length}\\
\cmidrule(lr){2-3}\cmidrule(lr){4-6}
Rule & $K_F$ & Stop & $n_0$ & $n_{10^{-3}}$ & $n_1$\\
\midrule\endfirsthead
\multicolumn{6}{l}{\tablename~\thetable\ (continued)}\\
\toprule
& \multicolumn{2}{c}{Stopping} & \multicolumn{3}{c}{Prefix length}\\
\cmidrule(lr){2-3}\cmidrule(lr){4-6}
Rule & $K_F$ & Stop & $n_0$ & $n_{10^{-3}}$ & $n_1$\\
\midrule\endhead
\multicolumn{6}{r}{Continued on next page}\\\endfoot
\bottomrule\endlastfoot
\multicolumn{6}{c}{\textit{RK4 coarse solver, prescribed start}}\\*
standard & $114$ & C & $113$ & $128$ & $128$\\*
$\psi$ ($w=e^{\lambda\Delta T}$) & $3$ & C & $2$ & $16$ & $25$\\*
$\psi$ (secant $w$) & $2$ & C & $1$ & $8$ & $15$\\*
coarse-only &  &  & $0$ & $1$ & $8$\\
\midrule
\multicolumn{6}{c}{\textit{Explicit Euler coarse solver, prescribed start}}\\*
standard & $65$ & F & -- & -- & --\\*
$\psi$ ($w=e^{\lambda\Delta T}$) & $21$ & C & $20$ & $20$ & $20$\\*
$\psi$ (secant $w$) & $12$ & C & $11$ & $11$ & $11$\\*
coarse-only &  &  & $0$ & $0$ & $0$\\
\midrule
\multicolumn{6}{c}{\textit{Explicit Euler coarse solver, replicate start}}\\*
standard & $33$ & F & -- & -- & --\\*
$\psi$ ($w=e^{\lambda\Delta T}$) & $22$ & C & $21$ & $21$ & $21$\\*
$\psi$ (secant $w$) & $12$ & C & $11$ & $11$ & $11$\\*
coarse-only &  &  & $0$ & $0$ & $0$\\
\end{longtable}\endgroup

\section{Functional accuracy and computational cost}
\label{sec:functional_accuracy}

Finite-trajectory distances do not identify which long-time functionals early stopping preserves. We therefore compare finite-horizon estimates with independent long-reference estimates, in two independent tests: Lorenz-96 with accurate coarse propagation here, and Lorenz-63 after an inadequate coarse solution in \Cref{sec:stopping_adequacy}. An earlier eight-start exploratory grid motivated these tests; its role and the lack of independent confirmation are disclosed below.

All starts receive a 40-time-unit fine spin-up. The observables are the mean of $z$, energy $\langle\|u\|^2/3\rangle$, and lag-0.5 Pearson autocorrelation of $x$ for Lorenz-63; for Lorenz-96 they are the time and spatial mean, energy $\langle\|u\|^2/40\rangle$, and time and spatial mean of $x_i x_{i - 1}$ with cyclic indices. For invariant averages, the Lorenz-96 energy balance gives $\langle\|u\|^2/40\rangle = 8\langle\bar{x}\rangle$: mean and energy carry the same information. Finite records have endpoint and sampling corrections, and accepted piecewise trajectories also have interface-jump terms, so the reported columns are closely related rather than exact multiples. Sampling uses intervals of 0.01, excluding the initial sample and including the terminal one, with linear interpolation between stored Runge--Kutta steps.

\paragraph{Reference and uncertainty}
\Cref{tab:functional_reference} gives independent long-reference estimates at $h = 5 \times 10^{-5}$ and the step-halving differences from $h = 10^{-4}$. Each step size uses four separately prepared chains after spin-up, of length 8192 for Lorenz-63 and 2048 for Lorenz-96. Reference means and autocorrelations are computed from pooled sufficient statistics; lag pairs crossing stored 128-time-unit block boundaries are omitted. We resample consecutive pairs of blocks, giving 256-time-unit blocks, with 4000 bootstrap draws. Repeating with 128-time-unit blocks changes the refined-reference interval widths by less than 7\%. The initial Lorenz-63 reference, of length 2048 per chain, showed a resolved step-halving difference, so we extended that reference fourfold once. The final paired step-halving intervals all contain zero. This is a refinement check at the available sampling precision, not a bound on discretization bias or a proof that the reference equals an invariant expectation.

For observable $j$, let $\widehat{a}_{r, j}^{M}$ be the estimate from method $M$ and start $r$, and $\widehat{a}_j^\star$ the refined long-reference estimate. With $n_{\mathrm{rep}}$ matched starts, we report
\begin{equation}
  \label{eq:sm_functional_error}
  E_j(M) = \frac{1}{n_{\mathrm{rep}}}\sum_{r = 1}^{n_{\mathrm{rep}}}
  |\widehat{a}_{r, j}^{M} - \widehat{a}_j^\star|.
\end{equation}
Because each estimate uses a finite horizon, this error includes sampling variation as well as numerical bias. Intervals resample whole matched starts and, independently, the long-reference blocks, using a common reference draw in each method comparison. This preserves dependence between time samples within each record. The confirmation studies use 10,000 bootstrap draws. Secondary intervals are pointwise, without multiplicity adjustment. A method with any nonfinite outcome has no aggregate error estimate; failures are counted rather than discarded.

\Needspace{20\baselineskip}\begingroup\small\setlength{\tabcolsep}{10pt}\renewcommand{\arraystretch}{1.15}
\begin{longtable}{@{}ccccc@{}}
\caption{\rev{Long-reference functional estimates. Each step size uses four independently prepared trajectories of duration 8192 (Lorenz-63) or 2048 (Lorenz-96) after spin-up. The reference is the estimate at $h=5\times10^{-5}$; the halving difference is the estimate at $h=10^{-4}$ minus the reference. Intervals are 95\% block-bootstrap intervals with 256-time-unit blocks; the halving intervals resample the same blocks at both step sizes.}}\label{tab:functional_reference}\\
\toprule
 & \multicolumn{2}{c}{Reference} & \multicolumn{2}{c}{Step halving}\\\cmidrule(lr){2-3}\cmidrule(lr){4-5}
Observable & Estimate & 95\% interval & Difference & 95\% interval\\
\midrule\endfirsthead
\multicolumn{5}{l}{\tablename~\thetable\ (continued)}\\
\toprule
 & \multicolumn{2}{c}{Reference} & \multicolumn{2}{c}{Step halving}\\\cmidrule(lr){2-3}\cmidrule(lr){4-5}
Observable & Estimate & 95\% interval & Difference & 95\% interval\\
\midrule\endhead
\multicolumn{5}{r}{Continued on next page}\\\endfoot
\bottomrule\endlastfoot
\multicolumn{5}{c}{\textit{Lorenz-63}}\\*
$\langle z\rangle$ & $23.547$ & [$23.537$, $23.557$] & $0.004$ & [$-0.010$, $0.018$]\\*
$\langle \|u\|^2/3\rangle$ & $257.622$ & [$257.525$, $257.715$] & $0.037$ & [$-0.086$, $0.170$]\\*
$\rho_x(0.5)$ & $0.1744$ & [$0.1680$, $0.1805$] & $0.0017$ & [$-0.0067$, $0.0105$]\\
\midrule
\multicolumn{5}{c}{\textit{Lorenz-96}}\\*
$\langle \bar{x}\rangle$ & $2.3391$ & [$2.3326$, $2.3459$] & $0.0047$ & [$-0.0071$, $0.0165$]\\*
$\langle \|u\|^2/40\rangle$ & $18.713$ & [$18.662$, $18.767$] & $0.038$ & [$-0.057$, $0.133$]\\*
$\langle x_i x_{i-1}\rangle$ & $6.345$ & [$6.314$, $6.377$] & $0.005$ & [$-0.051$, $0.063$]\\
\end{longtable}\endgroup

\paragraph{Cost and exploratory selection}
For a finite run with $s$ corrections, prefix reuse gives $J_F = (s + 1)N - s(s + 1)/2$ fine chunk solves. With coarse-to-fine chunk cost ratio $\zeta$, define
\begin{equation}
  \label{eq:sm_cost_models}
  \mathcal{C} = \frac{K_F + \zeta K_G}{N},
  \qquad \mathcal{Q} = \frac{J_F + \zeta K_G}{N}.
\end{equation}
Here $\mathcal{C}$ models elapsed cost with enough workers and $\mathcal{Q}$ counts total right-hand-side work, both normalized by serial fine propagation. For RK4 fine $h = 10^{-4}$ and RK4 coarse $H = 0.01$, $\zeta = 0.01$; serial RK4 with step $h_{\mathrm{ser}}$ has $\mathcal{C} = \mathcal{Q} = 10^{-4}/h_{\mathrm{ser}}$. These models exclude communication, allocation, diagnostics, and reference computations. They do not optimize a serial budget over trajectory length and independent replicas.

The exploratory grid used eight matched starts at $T = N = 32, 128$, RK4 fine $h = 10^{-4}$, RK4 or Euler coarse $H = 0.01$, both weighted rules, forced corrections, and serial alternatives. Energy was selected for independent confirmation after an earlier run of this grid, drawn from a different random-number stream, showed a Lorenz-96 energy reduction at $T = 128$ with RK4 coarse propagation, from 0.242 to 0.110, whose unadjusted paired interval excluded zero. The rerun reported here does not reproduce this reduction: the energy error is 0.154 for coarse-only and 0.231 for either weighted rule, with interval $[-0.0559, 0.199]$. These eight-start samples are distinct from the poor-coarse samples in \Cref{sec:poorcoarse_details}. Euler-based weighted runs failed for all eight Lorenz-96 starts at both horizons and, at $T = 128$, for five Lorenz-63 starts with the Lyapunov weight and four with the secant weight. The archived grid retains these failures and all other outcomes; the independent results below are not pooled with it.

\paragraph{Independent confirmation}
\label{sec:functional_confirmation}
The earlier exploratory energy reduction motivated a separate Lorenz-96 confirmation at $T = N = 128$ with 32 fresh starts, RK4 fine $h = 10^{-4}$ and coarse $H = 0.01$, and the preparation, sampling, and reference specified above. Both weighted rules use $\epsilon = 10^{-9}$; the nominal Lyapunov exponent is 1.68. Before generating these starts, we fixed mean absolute energy-estimation error as the primary outcome and Lyapunov acceptance minus coarse-only as the primary contrast. Secant acceptance, serial RK4 at $h_{\mathrm{ser}} = 0.005$, and the other observables are secondary comparisons. The sample size was fixed, failed starts were not replaced, and these records are not pooled with the original eight starts.

\Cref{tab:functional_confirmation} gives the outcomes in panel (a) and paired intervals in panel (b). All 128 method runs are finite; both weighted rules take two fine batches for every start and give identical functional estimates. Energy error is 0.153 for coarse-only, 0.200 for either accepted rule, and 0.181 for the intermediate serial solver. The primary difference is $+0.0474$, with interval $[-0.0133, 0.105]$, so the independent test does not establish an energy-error reduction. All secondary intervals also contain zero. Against the intermediate serial solver, the energy difference interval is $[-0.0476, 0.0851]$, while acceptance costs $\mathcal{C} = 0.0354$, $\mathcal{Q} = 2.01$ versus $\mathcal{C} = \mathcal{Q} = 0.02$. The extra modelled cost therefore has no demonstrated statistical gain in this test.

\begingroup\small\setlength{\tabcolsep}{1.5pt}\renewcommand{\arraystretch}{1.15}\setlength{\LTleft}{0pt}\setlength{\LTright}{0pt}
\Needspace{20\baselineskip}\begin{longtable}{@{\extracolsep{\fill}}ccccccc@{}}
\caption{\rev{Independent Lorenz-96 confirmation at $T = N = 128$: 32 fresh matched starts. Panel (a): mean absolute errors against long-reference estimates, fine batches $K_F$, modelled elapsed cost $\mathcal{C}$ and total work $\mathcal{Q}$. All runs are finite. Panel (b): paired 95\% intervals for error differences (method minus baseline); negative intervals favour the method. The starred energy contrast against coarse-only, prespecified for $\psi$ ($w=e^{\lambda\Delta T}$), is primary; all others are secondary pointwise comparisons. Resampling includes matched starts and long-reference uncertainty. Both weighted rules were run separately; they take the same fine batches and give identical estimates at every start, so they share one row.}}\label{tab:functional_confirmation}\\
\multicolumn{7}{@{}l}{\textbf{(a) Outcomes and cost}}\\
\toprule
Method & $K_F$ & $\mathcal{C}$ & $\mathcal{Q}$ & Mean & Energy & Neighbour\\
\midrule\endfirsthead
\multicolumn{7}{l}{\tablename~\thetable\ (continued)}\\
\toprule
Method & $K_F$ & $\mathcal{C}$ & $\mathcal{Q}$ & Mean & Energy & Neighbour\\
\midrule\endhead
\multicolumn{7}{r}{Continued on next page}\\\endfoot
\bottomrule\endlastfoot
coarse-only &  & $0.0100$ & $0.0100$ & $0.0195$ & $0.153$ & $0.116$\\*
serial RK4 ($h=0.005$) &  & $0.0200$ & $0.0200$ & $0.0223$ & $0.181$ & $0.113$\\*
$\psi$ (either $w$) & $2$ & $0.0354$ & $2.01$ & $0.0255$ & $0.200$ & $0.113$\\
\end{longtable}
\Needspace{20\baselineskip}\addtocounter{table}{-1}\def\theHtable{\arabic{table}.\number\value{LT@tables}}\begin{longtable}{@{\extracolsep{\fill}}cccc@{}}
\multicolumn{4}{@{}l}{\textbf{(b) Paired error differences}}\\
\toprule
Method & Mean & Energy & Neighbour\\
\midrule\endfirsthead
\multicolumn{4}{l}{\tablename~\thetable\ (continued)}\\
\toprule
Method & Mean & Energy & Neighbour\\
\midrule\endhead
\multicolumn{4}{r}{Continued on next page}\\\endfoot
\bottomrule\endlastfoot
\multicolumn{4}{c}{\textit{Against coarse-only}}\\*
$\psi$ (either $w$) & [$-0.00150$, $0.0131$] & [$-0.0133$, $0.105$]$^{*}$ & [$-0.0369$, $0.0312$]\\*
serial RK4 ($h=0.005$) & [$-0.00492$, $0.0106$] & [$-0.0342$, $0.0910$] & [$-0.0421$, $0.0406$]\\
\addlinespace[3pt]
\multicolumn{4}{c}{\textit{Against serial RK4 ($h=0.005$)}}\\*
$\psi$ (either $w$) & [$-0.00504$, $0.0113$] & [$-0.0476$, $0.0851$] & [$-0.0442$, $0.0403$]\\
\end{longtable}
\endgroup

\section{Stopping after an inadequate coarse solution}
\label{sec:stopping_adequacy}

We test whether Parareal can repair an inadequate coarse solution and whether the criterion stops at an adequate estimate. We fix Lorenz-63 at $T = N = 32$, $\Delta T = 1$, fine RK4 $h = 10^{-4}$, and the same three observables, sampling, and long reference as in \Cref{sec:functional_accuracy}, with energy primary. Before any coarse-solver pilot, 64 independently prepared fine runs set each accuracy limit to 1.25 times the fine mean absolute estimation error. The resulting limits are 0.18261 for mean $z$, 1.9954 for energy, and 0.097514 for autocorrelation. Calibration and confirmation use the same reference estimate, $T = 32$, sampling interval 0.01, and 40-unit fine spin-up, with separate random starts. These benchmark-specific limits remain fixed; the reported intervals are conditional on their calibration. The limits are themselves 64-run estimates and change with the random starts.

\paragraph{Pilot and feasibility}
The eight-start pilot tests Euler at $H = 0.005$, 0.01, and 0.02, and RK4 at $H = 0.025$, 0.04, 0.05, and 0.1. It selects the cheapest right-hand-side cost ratio among settings with finite coarse-only outputs and energy error above the limit: RK4 $H = 0.1$. In its first independent 32-start test the coarse-only energy error is 15.129, but every Lyapunov, secant, and standard run becomes nonfinite. Stability of the coarse pass alone therefore does not ensure usable Parareal corrections. These outcomes are retained; no aggregate error is computed for a method with failures.

We then amend the protocol to screen feasibility, rather than treating these runs as evidence about successful stopping. On the same eight pilot starts and the same grid, a candidate must additionally reach native exact completion in every forced run and reproduce the serial fine interfaces bitwise. The screen uses neither weighted acceptance nor the statistical outcomes of corrections. RK4 $H = 0.05$ is the cheapest eligible choice; the cheaper RK4 $H = 0.1$ and Euler $H = 0.02$ candidates fail seven and all eight of their exact-completion runs. Before testing on another 32 fresh starts, kept separate from the first test, we fix the configuration and analysis, retaining the calibrated limits and both criteria without retuning. The original protocol and amendment are retained with the computation. In the complete rerun reported here, every stage was computed in one pass by the fixed protocol code; applied to its records, the pilot and screening rules again select RK4 $H = 0.1$ and then RK4 $H = 0.05$.

\paragraph{Decision and comparison}
Both weighted rules use $\epsilon = 10^{-9}$; the Lyapunov base is $e^{0.9056}$ and the other rule is the default same-chunk secant. The standard rule uses absolute and relative tolerances $10^{-12}$ and $10^{-9}$. Native forced runs retain every iterate from zero to 31 corrections, including continuation beyond acceptance; diagnostic reconstruction uses native fine propagation and is excluded from solver cost. Accepted functional estimates match those of the corresponding forced iterate. Serial fine RK4 and serial RK4 at $h_{\mathrm{ser}} = 0.005$ provide additional comparisons.

The primary success requirement is the conjunction of three tests: the coarse energy error's one-sided 95\% lower bound exceeds the fixed limit, Lyapunov acceptance's one-sided 95\% upper bound is below it, and the upper bound on the ratio of its mean modelled elapsed cost to that of standard stopping is below one half. We use $n_{\mathrm{rep}} = 32$ and the resampling procedure in \Cref{sec:functional_accuracy}. Secant and the other observables are secondary. The cost definitions $\mathcal{C}, \mathcal{Q}$ in \Cref{sec:functional_accuracy} apply with $\zeta = 0.002$; a reduced $\mathcal{C}$ need not mean less total work.

\Cref{tab:stopping_adequacy} shows an error reduction and a cost saving in the second confirmation, where every run is finite. The coarse energy error is 2.815, with lower bound 2.360 above the fixed limit, whereas Lyapunov acceptance reduces it to 1.516 with paired accepted-minus-coarse interval $[-1.979, -0.654]$. Its 8--14 fine batches compare with 19--26 for standard stopping. Mean costs are $\mathcal{C} = 0.349$ against 0.730, a ratio of means 0.479 with upper bound 0.496; total work is $\mathcal{Q} = 9.02$ against 14.9. The energy-error upper bound is 1.808, below 1.995. All three requirements hold, although the cost bound is only 0.004 below one half.

Secant acceptance uses 3--7 batches and $\mathcal{C} = 0.146$, but its energy error is 2.548 and its paired difference interval $[-1.190, 0.722]$ includes zero. Lyapunov acceptance's mean-$z$ and autocorrelation bounds meet their secondary limits, although their improvement intervals include zero. The serial fine and standard trajectories also meet the primary upper-bound test, both with energy error 1.354 and upper bound 1.602, and so does serial RK4 at $h_{\mathrm{ser}} = 0.005$, with energy error 1.450, upper bound 1.776, and $\mathcal{C} = \mathcal{Q} = 0.02$. Meeting the limit therefore does not single out weighted stopping: the intermediate serial solver is adequate here at about a seventeenth of Lyapunov acceptance's modelled elapsed cost.

\begingroup\small\setlength{\tabcolsep}{1.5pt}\renewcommand{\arraystretch}{1.15}\setlength{\LTleft}{0pt}\setlength{\LTright}{0pt}
\Needspace{20\baselineskip}\begin{longtable}{@{\extracolsep{\fill}}ccccccc@{}}
\caption{\rev{Independent Lorenz-63 confirmation after the feasibility screen: $T=N=32$, RK4 coarse $H=0.05$, 32 fresh starts, all runs finite. Errors are means of absolute functional-estimation errors against the long reference. $L_{0.95},U_{0.95}$ are one-sided 95\% bootstrap bounds. The first fine batch is the forced run's first iterate: fine trajectories from coarse-pass starts, without correction. Panel (b) reports paired energy-error differences against coarse-only and ratios of mean elapsed costs against standard stopping, with a two-sided 95\% interval and one-sided upper bound, respectively. Secondary limits appear in panel (c).}}\label{tab:stopping_adequacy}\\
\multicolumn{7}{@{}l}{\textbf{(a) Primary energy error (limit 1.995)}}\\
\toprule
 & \multicolumn{3}{c}{Energy error} & \multicolumn{3}{c}{Work and cost}\\\cmidrule(lr){2-4}\cmidrule(lr){5-7}
Method & Estimate & $L_{0.95}$ & $U_{0.95}$ & $K_F$ & $\mathcal{C}$ & $\mathcal{Q}$\\
\midrule\endfirsthead
\multicolumn{7}{l}{\tablename~\thetable\ (continued)}\\
\toprule
 & \multicolumn{3}{c}{Energy error} & \multicolumn{3}{c}{Work and cost}\\\cmidrule(lr){2-4}\cmidrule(lr){5-7}
Method & Estimate & $L_{0.95}$ & $U_{0.95}$ & $K_F$ & $\mathcal{C}$ & $\mathcal{Q}$\\
\midrule\endhead
\multicolumn{7}{r}{Continued on next page}\\\endfoot
\bottomrule\endlastfoot
coarse-only & $2.815$ & $2.360$ & $3.284$ &  & $2.00\times 10^{-3}$ & $2.00\times 10^{-3}$\\*
first fine batch & $1.652$ & $1.333$ & $2.002$ & $1$ & $3.32\times 10^{-2}$ & $1.00$\\*
serial fine & $1.354$ & $1.107$ & $1.602$ &  & $1.00$ & $1.00$\\*
serial RK4 ($h=0.005$) & $1.450$ & $1.141$ & $1.776$ &  & $2.00\times 10^{-2}$ & $2.00\times 10^{-2}$\\*
$\psi$ ($w=e^{\lambda\Delta T}$) & $1.516$ & $1.238$ & $1.808$ & $8$--$14$ & $3.49\times 10^{-1}$ & $9.02$\\*
$\psi$ (secant $w$) & $2.548$ & $1.868$ & $3.325$ & $3$--$7$ & $1.46\times 10^{-1}$ & $4.17$\\*
standard & $1.354$ & $1.107$ & $1.602$ & $19$--$26$ & $7.30\times 10^{-1}$ & $14.9$\\
\end{longtable}
\Needspace{20\baselineskip}\addtocounter{table}{-1}\def\theHtable{\arabic{table}.\number\value{LT@tables}}\begin{longtable}{@{\extracolsep{\fill}}ccccc@{}}
\multicolumn{5}{@{}l}{\textbf{(b) Paired energy and elapsed-cost comparisons}}\\
\toprule
 & \multicolumn{2}{c}{Energy: accepted minus coarse-only} & \multicolumn{2}{c}{Cost: accepted / standard}\\\cmidrule(lr){2-3}\cmidrule(lr){4-5}
Rule & Difference & 95\% interval & Ratio & $U_{0.95}$\\
\midrule\endfirsthead
\multicolumn{5}{l}{\tablename~\thetable\ (continued)}\\
\toprule
 & \multicolumn{2}{c}{Energy: accepted minus coarse-only} & \multicolumn{2}{c}{Cost: accepted / standard}\\\cmidrule(lr){2-3}\cmidrule(lr){4-5}
Rule & Difference & 95\% interval & Ratio & $U_{0.95}$\\
\midrule\endhead
\multicolumn{5}{r}{Continued on next page}\\\endfoot
\bottomrule\endlastfoot
$\psi$ ($w=e^{\lambda\Delta T}$) & $-1.299$ & [$-1.979$, $-0.654$] & $0.479$ & $0.496$\\*
$\psi$ (secant $w$) & $-0.267$ & [$-1.190$, $0.722$] & $0.200$ & $0.212$\\
\end{longtable}
\Needspace{20\baselineskip}\addtocounter{table}{-1}\def\theHtable{\arabic{table}.\number\value{LT@tables}}\begin{longtable}{@{\extracolsep{\fill}}ccccc@{}}
\multicolumn{5}{@{}l}{\textbf{(c) Secondary observables}}\\
\toprule
 & \multicolumn{2}{c}{$\langle z\rangle$ (limit $0.183$)} & \multicolumn{2}{c}{$\rho_x(0.5)$ (limit $0.0975$)}\\\cmidrule(lr){2-3}\cmidrule(lr){4-5}
Method & Estimate & $U_{0.95}$ & Estimate & $U_{0.95}$\\
\midrule\endfirsthead
\multicolumn{5}{l}{\tablename~\thetable\ (continued)}\\
\toprule
 & \multicolumn{2}{c}{$\langle z\rangle$ (limit $0.183$)} & \multicolumn{2}{c}{$\rho_x(0.5)$ (limit $0.0975$)}\\\cmidrule(lr){2-3}\cmidrule(lr){4-5}
Method & Estimate & $U_{0.95}$ & Estimate & $U_{0.95}$\\
\midrule\endhead
\multicolumn{5}{r}{Continued on next page}\\\endfoot
\bottomrule\endlastfoot
coarse-only & $0.150$ & $0.184$ & $0.0784$ & $0.0959$\\*
first fine batch & $0.151$ & $0.183$ & $0.0816$ & $0.0987$\\*
serial fine & $0.114$ & $0.138$ & $0.0589$ & $0.0745$\\*
serial RK4 ($h=0.005$) & $0.121$ & $0.151$ & $0.0670$ & $0.0844$\\*
$\psi$ ($w=e^{\lambda\Delta T}$) & $0.131$ & $0.161$ & $0.0770$ & $0.0896$\\*
$\psi$ (secant $w$) & $0.244$ & $0.320$ & $0.0966$ & $0.118$\\*
standard & $0.114$ & $0.138$ & $0.0589$ & $0.0745$\\
\end{longtable}
\endgroup

\paragraph{What the correction history shows}
\Cref{fig:stopping_adequacy} separates correction progress from statistical estimation. The first fine batch, run from coarse-pass starts with no correction, already has energy error 1.652; it is distinct from the coarse-only trajectory. With $K_F = 1$, $K_G = N - 1 = 31$, and $\zeta = 0.002$, its modelled elapsed cost is $\mathcal{C} = (1 + 0.002\times31)/32 \approx 0.0332$. This baseline is included in \Cref{tab:stopping_adequacy} and the article's \Cref{tab:functional_main}. Its upper bound, 2.002, lies just above the fixed limit; later pointwise upper bounds fall below the limit only after several corrections, and not monotonically. The estimated error at Lyapunov acceptance is 0.137 below that of the first batch. The post hoc secondary paired interval for Lyapunov acceptance minus the first fine batch is $[-0.693, 0.407]$, which includes zero; it establishes neither an improvement over the first batch nor equivalence. The reduction relative to coarse-only therefore does not establish an additional energy-estimation benefit from corrections or criterion-based stopping. The initial fine trajectories need not satisfy the weighted stopping criterion or have small interface defects. Extra corrections sometimes worsen the energy estimate; earlier crossings and later recrossings show why an earliest observed crossing is only a retrospective diagnostic.

\begin{figure}[H]
  \centering
  \includegraphics[width=\textwidth]{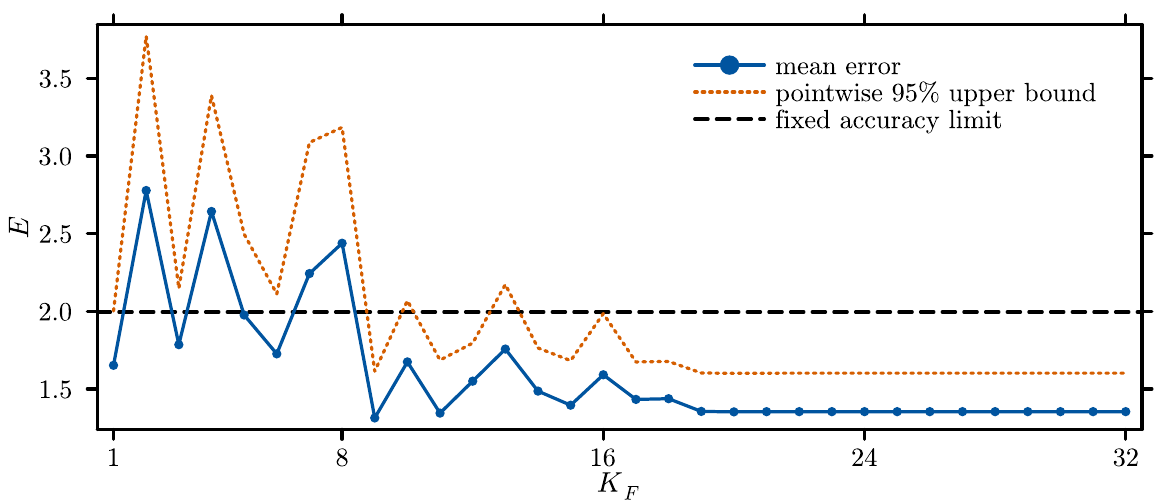}
  \caption{Lorenz-63 energy-estimation error along the 32 forced fine batches in the second confirmation. The dotted curve gives pointwise one-sided 95\% upper bounds and the dashed line the fixed accuracy limit. The bounds fall below the limit only after several corrections, and not monotonically. The first batch contains fine trajectories from coarse-pass starts; it is distinct from coarse-only. These are retrospective diagnostics, not statistical tests used by the stopping criterion.}
  \label{fig:stopping_adequacy}
\end{figure}

\section{Prediction inputs and implementation checks}
\label{sec:implementation_checks}

\Cref{tab:prediction_details} gives the derivative samples for \Cref{tab:prediction}; \Cref{tab:prediction_error_link} compares the two initial-error estimates used in its budgets.

\begin{remark}[Finite precision]
  \label[remark]{rem:roundoff}
  Under the hypotheses of \Cref{thm:ostrovski}, if each computed sweep adds a perturbation of norm at most $\eta_j$ and the computed iterates remain in $\Omega$, then
  \begin{equation}
    \label{eq:roundoff}
    \|e^s\| \le \beta^s\|e^0\| + \sum_{j = 0}^{s - 1}\beta^{s - 1 - j}\eta_j.
  \end{equation}
  If $\eta_j \le \eta$ for all $j$, the sum is at most $\eta(1 - \beta^s)/(1 - \beta)$. Thus the upper envelope tends to $\eta/(1 - \beta)$; the estimate allows a roundoff plateau but does not require one.
\end{remark}

The article's \Cref{fig:measured_vs_bound,sec:prediction} show the stored-start error histories and explain the conservative budget predictions. The sampled envelopes fall below measured errors only after the target crossings, at the roundoff scale, consistent with \Cref{rem:roundoff}; sampled curves do not provide uniform neighbourhood bounds.

\Needspace{20\baselineskip}\begingroup\small\setlength{\tabcolsep}{10pt}\renewcommand{\arraystretch}{1.15}
\begin{longtable}{@{}cccccc@{}}
\caption{\rev{Sampled derivative diagnostics for \Cref{tab:prediction}. The vector-field estimates use every coarse-step node of the coarse pass; the chunk-map estimates use 16 of the 32 coarse-pass chunk starts, and their maxima are inflated by 1.25. Maxima use the state-induced matrix norm; hats denote sampled estimates.}}\label{tab:prediction_details}\\
\toprule
& \multicolumn{2}{c}{Vector field} & \multicolumn{3}{c}{Chunk maps}\\\cmidrule(lr){2-3}\cmidrule(lr){4-6}
$\xi$ & $\hat\mu$ & $\hat\nu$ & $\hat\Lambda_{\mathcal F}$ & $\hat\Lambda_{\mathcal G}$ & $\hat S_{\mathrm{map}}$\\
\midrule\endfirsthead
\multicolumn{6}{l}{\tablename~\thetable\ (continued)}\\
\toprule
& \multicolumn{2}{c}{Vector field} & \multicolumn{3}{c}{Chunk maps}\\\cmidrule(lr){2-3}\cmidrule(lr){4-6}
$\xi$ & $\hat\mu$ & $\hat\nu$ & $\hat\Lambda_{\mathcal F}$ & $\hat\Lambda_{\mathcal G}$ & $\hat S_{\mathrm{map}}$\\
\midrule\endhead
\multicolumn{6}{r}{Continued on next page}\\\endfoot
\bottomrule\endlastfoot
\multicolumn{6}{c}{\textit{Lorenz-63}}\\*
$10$ & $43.4$ & $768$ & $61.8$ & $61.8$ & $7.25\times 10^{-8}$\\*
$100$ & $43.4$ & $764$ & $365$ & $365$ & $3.67\times 10^{-2}$\\
\midrule
\multicolumn{6}{c}{\textit{Logistic}}\\*
$10$ & $1.00$ & $0.500$ & $2.47$ & $2.48$ & $2.06\times 10^{-3}$\\*
$100$ & $1.00$ & $0.500$ & $2.47$ & $2.48$ & $2.18\times 10^{-2}$\\
\midrule
\multicolumn{6}{c}{\textit{Periodic}}\\*
$10$ & $63.5$ & $1510$ & $194$ & $194$ & $2.18\times 10^{-5}$\\*
$100$ & $63.5$ & $1500$ & $194$ & $194$ & $4.97\times 10^{-1}$\\
\end{longtable}\endgroup

\Needspace{20\baselineskip}\begingroup\small\setlength{\tabcolsep}{10pt}\renewcommand{\arraystretch}{1.15}
\begin{longtable}{@{}cccc@{}}
\caption{\rev{Transported-defect and one-look initial-error estimates. The relative excess is $g=100[C_N\psi(U^0)/\hat e_0-1]$. Using $C_N\psi(U^0)$ in place of $\hat e_0$ leaves every budget in \Cref{tab:prediction} unchanged.}}\label{tab:prediction_error_link}\\
\toprule
& \multicolumn{2}{c}{Initial-error estimate} & \\\cmidrule(lr){2-3}
$\xi$ & $\hat e_0$ & $C_N\psi(U^0)$ & $g$ (\%)\\
\midrule\endfirsthead
\multicolumn{4}{l}{\tablename~\thetable\ (continued)}\\
\toprule
& \multicolumn{2}{c}{Initial-error estimate} & \\\cmidrule(lr){2-3}
$\xi$ & $\hat e_0$ & $C_N\psi(U^0)$ & $g$ (\%)\\
\midrule\endhead
\multicolumn{4}{r}{Continued on next page}\\\endfoot
\bottomrule\endlastfoot
\multicolumn{4}{c}{\textit{Lorenz-63}}\\*
$10$ & $8.02\times 10^{-9}$ & $8.02\times 10^{-9}$ & $1.15\times 10^{-2}$\\*
$100$ & $8.96\times 10^{-6}$ & $8.96\times 10^{-6}$ & $2.48\times 10^{-3}$\\
\midrule
\multicolumn{4}{c}{\textit{Logistic}}\\*
$10$ & $1.68\times 10^{-3}$ & $1.70\times 10^{-3}$ & $1.08$\\*
$100$ & $1.78\times 10^{-2}$ & $1.80\times 10^{-2}$ & $1.03$\\
\midrule
\multicolumn{4}{c}{\textit{Periodic}}\\*
$10$ & $2.51\times 10^{-8}$ & $2.51\times 10^{-8}$ & $2.58\times 10^{-2}$\\*
$100$ & $4.53\times 10^{-4}$ & $4.53\times 10^{-4}$ & $2.67\times 10^{-2}$\\
\end{longtable}\endgroup

\paragraph{One-look comparison}
\label{sec:one_look}
The one-look estimate in \Cref{eq:onelook} is only about 0.0025\%--1.1\% above the transported value and gives the same integer budgets in all six cases. Here $\theta_{\mathcal{F}} = \hat{\Lambda}_{\mathcal{F}}/w \le 1$ is close to one in all six cases. Put $z_m = w^{-m}\|d_m^0\|$. The transported estimate is the $\ell^2$ norm of $a_n = \sum_{m \le n}\theta_{\mathcal{F}}^{n - m}z_m$, whereas the one-look estimate is $H_{N, 2}(\theta_{\mathcal{F}})\sum_m z_m$. When the first weighted defect dominates, both are close to $H_{N, 2}(\theta_{\mathcal{F}})z_1$. This dominance is stronger at the Lorenz-63 weights $w \approx 62, 365$ than at the logistic weight $w \approx 2.5$. At $\theta_{\mathcal{F}} = 1$, the two estimates are $\|(a_n)\|_2$ and $\sqrt{N}\, a_N$; they agree exactly when the cumulative sums $a_n$ are constant, or equivalently when all nonzero weighted defects occur at the first interface. Constant-coefficient RK4 checks $u' = \pm u$ at $N = 8, 32$ use closed scalar chunk factors without state sampling, and the native stored-start Parareal in the stored-start norm. They give $\beta_W$ between $4.34\times 10^{-7}$ and $2.38\times 10^{-6}$ and $s_{\mathrm{pred}} = s_{\mathrm{hit}} = 1$ in all four cases.

\subsection{Frozen-source and coarse-node predictions}
\label{sec:frozen_prediction}

We consider two distinct comparisons. The main prediction table replaces only the source estimate:
\begin{equation}
  \label{eq:sm_beta_fr}
  \hat{\beta}_{\mathrm{fr}} = \frac{\hat{S}_{\mathrm{fr}}}{w}
  \sum_{i = 0}^{N - 2}\left(\frac{\hat{\Lambda}_{\mathcal{G}}}{w}\right)^i,
  \qquad s_{\mathrm{fr}} = s_{\mathrm{pred}}(\hat{e}_0;\hat{\beta}_{\mathrm{fr}}),
\end{equation}
where the second notation means that \Cref{eq:budget_box} uses the frozen-source factor. Transport, weight, and initial-error estimate still use sampled actual-map derivatives. The six factors, in main-table order, are approximately 0.173, 0.321, 0.0845, 0.926, 0.109, and 1.19. Budgets are evaluated from unrounded inputs: the fourth is 309, rather than an estimate from the printed factor. This is an uncapped sufficient-budget calculation; native exact completion gives the smaller fallback of $N - 1 = 31$ corrections. A factor at least one makes this contraction calculation inconclusive, even if the observed iteration converges rapidly. The values of $\hat{\mu}\Delta T$ are about 43.4, 1, and 63.5 for Lorenz-63, logistic, and periodic cases respectively; a large value alone does not invalidate the fixed-step-count asymptotic theorem or quantify its remainder.

For the shrinking-window comparison, we use Lorenz-63 with the prescribed start $(2, 3, -14)$, RK4, $N = 32$, and $r = 10^{-9}$. At $\xi = 10$ we fix $L = 20$, and at $\xi = 100$ we fix $L = 2$. Halving $h = 10^{-4}$ to $5 \times 10^{-5}$ then halves $\Delta T = 0.02$ to 0.01 without changing either step count per chunk. \Cref{tab:prediction_short} compares the two sources at identical sampled transport and weight. At $\xi = 10$, the shorter-window coarse start already meets the weighted error target, giving zero sweeps. At $\xi = 100$, shortening lowers the frozen factor from about 1.25 to 0.383, but its ten-sweep prediction remains much larger than the observed one.

\Needspace{20\baselineskip}\begingroup\small\setlength{\tabcolsep}{5pt}\renewcommand{\arraystretch}{1.15}
\begin{longtable}{@{}ccccccccc@{}}
\caption{\rev{Lorenz-63 short-window comparison, $N=32$, RK4, start $(2,3,-14)$, and target $r=10^{-9}$. Each pair halves $h,H$ at fixed $L,\xi,N$ ($L=20$ for $\xi=10$, $L=2$ for $\xi=100$). Sixteen derivative samples and inflation 1.25 give $\hat\beta_W$. Replacing only the source by the frozen leading term gives $\hat\beta_{\mathrm{fr}}$ and the uncapped budget $s_{\mathrm{fr}}$, retaining the same weight and initial-error estimate. A dash means no contraction prediction. The observed crossing $s_{\mathrm{hit}}$ uses this same fixed stored-start norm.}}\label{tab:prediction_short}\\
\toprule
& & & & \multicolumn{2}{c}{Contraction factor} & \multicolumn{3}{c}{Sweeps}\\\cmidrule(lr){5-6}\cmidrule(lr){7-9}
$\xi$ & $h$ & $\Delta T$ & $\hat\mu\Delta T$ & $\hat\beta_W$ & $\hat\beta_{\mathrm{fr}}$ & $s_{\mathrm{pred}}$ & $s_{\mathrm{fr}}$ & $s_{\mathrm{hit}}$\\
\midrule\endfirsthead
\multicolumn{9}{l}{\tablename~\thetable\ (continued)}\\
\toprule
& & & & \multicolumn{2}{c}{Contraction factor} & \multicolumn{3}{c}{Sweeps}\\\cmidrule(lr){5-6}\cmidrule(lr){7-9}
$\xi$ & $h$ & $\Delta T$ & $\hat\mu\Delta T$ & $\hat\beta_W$ & $\hat\beta_{\mathrm{fr}}$ & $s_{\mathrm{pred}}$ & $s_{\mathrm{fr}}$ & $s_{\mathrm{hit}}$\\
\midrule\endhead
\multicolumn{9}{r}{Continued on next page}\\\endfoot
\bottomrule\endlastfoot
\multirow{2}{*}{$10$} & $10^{-4}$ & $0.02$ & $0.869$ & $2.19\times 10^{-7}$ & $0.114$ & $1$ & $2$ & $1$\\*
 & $5\times 10^{-5}$ & $0.01$ & $0.434$ & $9.80\times 10^{-9}$ & $0.0348$ & $0$ & $0$ & $0$\\
\addlinespace[3pt]
\multirow{2}{*}{$100$} & $10^{-4}$ & $0.02$ & $0.869$ & $2.15\times 10^{-3}$ & $1.25$ & $2$ & -- & $1$\\*
 & $5\times 10^{-5}$ & $0.01$ & $0.434$ & $9.86\times 10^{-5}$ & $0.383$ & $1$ & $10$ & $1$\\
\end{longtable}\endgroup

A separate diagnostic avoids differentiating chunk maps. For classical RK4, a uniform bound $\|Df\| \le \mu$ at all stages bounds the derivative of one step of length $\delta$ by $P_4(\delta\mu)$, where $P_4(z) = 1 + z + z^2/2 + z^3/6 + z^4/24$. Indeed, the norms of the four stage derivatives satisfy the same upper recursions as RK4 applied to the scalar variational equation $v' = \mu v$. Thus a uniform bound gives chunk factors $P_4(h\mu)^{L\xi}$ and $P_4(H\mu)^L$. With the coarse-node sample $\hat{\mu}$ in their place, we compute the diagnostic quantities
\begin{equation}
  \label{eq:rk4_transport_diagnostic}
  \hat{\Lambda}_{\mathcal{F}, \mathrm{RK}} = P_4(h\hat{\mu})^{L\xi},
  \quad \hat{\Lambda}_{\mathcal{G}, \mathrm{RK}} = P_4(H\hat{\mu})^L,
  \quad w_{\mathrm{RK}} = \max\{1, \hat{\Lambda}_{\mathcal{F}, \mathrm{RK}}, \hat{\Lambda}_{\mathcal{G}, \mathrm{RK}}\}.
\end{equation}
We use these factors with $\hat{S}_{\mathrm{fr}}$ for $\hat{\beta}_{\mathrm{RK}}$ and with all first-batch defects for $\hat{e}_{0, \mathrm{RK}}$. Both the predicted budget $s_{\mathrm{RK}}$ and the observed crossing $s_{\mathrm{hit}, \mathrm{RK}}$ use this new fixed weight; comparisons across weights therefore concern different accuracy targets. \Cref{tab:prediction_stage} gives one predicted and observed sweep at $\xi = 10$, $\Delta T = 0.02$. At $\xi = 100$, the uncapped prediction is 564 sweeps, superseded by native exact completion at $N - 1 = 31$ corrections, against one observed. Halving reduces the prediction to ten, still against one observed. The four rows consequently do not justify a universal switch to the frozen route when $\hat{\mu}\Delta T \le 1$.

\Needspace{20\baselineskip}\begingroup\small\setlength{\tabcolsep}{10pt}\renewcommand{\arraystretch}{1.15}
\begin{longtable}{@{}ccccccc@{}}
\caption{\rev{Coarse-node RK4-transport diagnostic for the same short windows. The sampled $\hat\mu$ gives stage majorants and weight $w_{\mathrm{RK}}$ by \Cref{eq:rk4_transport_diagnostic}; the frozen source gives $\hat\beta_{\mathrm{RK}}$, and all first-batch defects give $\hat e_{0,\mathrm{RK}}$. Predicted and observed counts use this new fixed weight. Budgets are uncapped; native exact completion supplies the $N-1=31$-correction fallback. These are sampled diagnostics, not uniform enclosures.}}\label{tab:prediction_stage}\\
\toprule
& & & & & \multicolumn{2}{c}{Sweeps}\\\cmidrule(lr){6-7}
$\xi$ & $h$ & $w_{\mathrm{RK}}$ & $\hat\beta_{\mathrm{RK}}$ & $\hat e_{0,\mathrm{RK}}$ & $s_{\mathrm{RK}}$ & $s_{\mathrm{hit},\mathrm{RK}}$\\
\midrule\endfirsthead
\multicolumn{7}{l}{\tablename~\thetable\ (continued)}\\
\toprule
& & & & & \multicolumn{2}{c}{Sweeps}\\\cmidrule(lr){6-7}
$\xi$ & $h$ & $w_{\mathrm{RK}}$ & $\hat\beta_{\mathrm{RK}}$ & $\hat e_{0,\mathrm{RK}}$ & $s_{\mathrm{RK}}$ & $s_{\mathrm{hit},\mathrm{RK}}$\\
\midrule\endhead
\multicolumn{7}{r}{Continued on next page}\\\endfoot
\bottomrule\endlastfoot
\multirow{2}{*}{$10$} & $10^{-4}$ & $2.38$ & $0.0898$ & $7.20\times 10^{-9}$ & $1$ & $1$\\*
 & $5\times 10^{-5}$ & $1.54$ & $0.0347$ & $5.92\times 10^{-10}$ & $0$ & $0$\\
\addlinespace[3pt]
\multirow{2}{*}{$100$} & $10^{-4}$ & $2.38$ & $0.980$ & $7.67\times 10^{-5}$ & $564$ & $1$\\*
 & $5\times 10^{-5}$ & $1.54$ & $0.381$ & $6.15\times 10^{-6}$ & $10$ & $1$\\
\end{longtable}\endgroup

The coarse-node diagnostic still needs $Df$ (evaluated here by automatic differentiation of the vector field) and all initial defects. It avoids differentiating entire chunk maps, but the sampled $\hat{\mu}, \hat{\nu}$ and omitted frozen remainder do not certify its factors. Exact-flow Gr\"onwall bounds cannot simply be assigned to arbitrary numerical methods: for scalar backward Euler with $0 < H\mu < 1$, the chunk derivative is $(1 - H\mu)^{-L} > e^{\mu LH}$. For RK4 the stage majorant above supplies the required method-specific argument. Similarly, the frozen Neumann condition checks its matrix expansion, not solvability of nonlinear stage equations from samples alone.

\subsection{Matched coarse-start predictions}
\label{sec:coarse_prediction}

To test the interpretation of the poor-coarse control, we use the prescribed Lorenz-63 start $(2, 3, -14)$ at $N = 16, 128$, $\Delta T = 1$, and $H = 0.01, 0.02$, with matched RK4 and Euler coarse maps and fine RK4 at $h = 10^{-4}$. These are eight additional deterministic diagnostics, separate from the replicated statistical experiment. We sample sixteen coarse-pass chunk derivatives, inflate their maxima by 1.25, and use all initial defects, exactly as in the main prediction procedure. Each candidate fixes its sampled weight throughout the error history. \Cref{tab:prediction_coarse} compares those predictions with the observed target crossings, including cases where prediction is inconclusive; its caption lists the coarse passes that fail.

\Needspace{20\baselineskip}\begingroup\small\setlength{\tabcolsep}{10pt}\renewcommand{\arraystretch}{1.15}
\begin{longtable}{@{}cccccccc@{}}
\caption{\rev{Matched coarse-start predictions for Lorenz-63, prescribed start $(2,3,-14)$, $\Delta T=1$, fine RK4 with $h=10^{-4}$, and $r=10^{-9}$. Sixteen coarse-pass derivative samples, inflated by 1.25, define the row's fixed weight $w$ and factor. All initial defects define $\hat e_0$. Predicted and observed crossings use that same weighted error, not the stopping rule. A dash means no contraction prediction. The coarse passes of Euler at $H=0.02$ are nonfinite, so those rows are left out. Weights differ between rows; the final discount factor $w^{-N}$ for RK4 at $N=128$, $H=0.01, 0.02$ underflows in Float64.}}\label{tab:prediction_coarse}\\
\toprule
& & & & & & \multicolumn{2}{c}{Sweeps}\\\cmidrule(lr){7-8}
$N$ & $H$ & Coarse & $w$ & $\hat\beta_W$ & $\hat e_0$ & $s_{\mathrm{pred}}$ & $s_{\mathrm{hit}}$\\
\midrule\endfirsthead
\multicolumn{8}{l}{\tablename~\thetable\ (continued)}\\
\toprule
& & & & & & \multicolumn{2}{c}{Sweeps}\\\cmidrule(lr){7-8}
$N$ & $H$ & Coarse & $w$ & $\hat\beta_W$ & $\hat e_0$ & $s_{\mathrm{pred}}$ & $s_{\mathrm{hit}}$\\
\midrule\endhead
\multicolumn{8}{r}{Continued on next page}\\\endfoot
\bottomrule\endlastfoot
\multirow{3}{*}{$16$} & \multirow{2}{*}{$0.01$} & RK4 & $365$ & $1.51\times 10^{-3}$ & $6.34\times 10^{-6}$ & $2$ & $1$\\*
 &  & Euler & $22.5$ & $1.95$ & $4.71$ & -- & $7$\\*
 & $0.02$ & RK4 & $17.0$ & $1.87\times 10^{-2}$ & $7.47\times 10^{-3}$ & $4$ & $2$\\
\addlinespace[3pt]
\multirow{3}{*}{$128$} & \multirow{2}{*}{$0.01$} & RK4 & $365$ & $1.28\times 10^{-2}$ & $1.79\times 10^{-5}$ & $3$ & $1$\\*
 &  & Euler & $46.8$ & $140$ & $6.92\times 10^{-1}$ & -- & $6$\\*
 & $0.02$ & RK4 & $366$ & $2.29\times 10^{-1}$ & $9.80\times 10^{-4}$ & $10$ & $1$\\
\end{longtable}\endgroup

At $H = 0.01$ both Euler factors exceed one, so the prediction is inconclusive even though the weighted target is attained. Neither sampled factors nor observed error ratios verify contraction on a ball or determine $\varrho$, so they establish neither membership nor delayed entry. An upper estimate of initial error exceeding a candidate radius would not itself prove that the actual error exceeds it.

For a comparison in the same norm across methods, we also retain errors with $w = e^{0.9056}$ and their ratios $q_s = E_{s + 1}/E_s$. \Cref{tab:prediction_coarse_ratios} gives the first three and the target crossing; the input records retain every sweep. The Euler errors decrease from the first sweep, albeit much more slowly than the RK4 errors. Later ratios can fluctuate, particularly below the target where roundoff matters. Continuing the $N = 128$, $H = 0.01$ Euler run beyond its weighted-error target crossing at sweep 26 leads to a nonfinite correction at sweep 65, matching the standard run in \Cref{fig:prefix}. This reference-based error target differs from the standard update criterion; neither run is an early-accepted weighted run. At $H = 0.02$ the prescribed-start Euler coarse trajectory is already nonfinite, so no initial derivative prediction is attempted.

\Needspace{20\baselineskip}\begingroup\small\setlength{\tabcolsep}{5pt}\renewcommand{\arraystretch}{1.15}
\begin{longtable}{@{}cccccccc@{}}
\caption{\rev{Common-norm error ratios for the coarse starts in \Cref{tab:prediction_coarse}. Here $E_s=E_{\mathrm{start}}(U^s)$ as defined in \Cref{sec:stored_starts} with fixed $w=e^{0.9056}$, $q_s=E_{s+1}/E_s$, and $s_{\mathrm{hit,L}}$ is the first crossing of $10^{-9}$ in this norm. These crossings must not be compared with predictions in a different norm. Runs are forced beyond acceptance to exact completion at $N-1$ corrections or to failure; the last column gives the outcome and, for a failure, the last finite sweep. Full histories are in the records, including below-target ratios that may be dominated by roundoff.}}\label{tab:prediction_coarse_ratios}\\
\toprule
& & & \multicolumn{3}{c}{Consecutive-error ratios} & & \\\cmidrule(lr){4-6}
$N$ & $H$ & Coarse & $q_0$ & $q_1$ & $q_2$ & $s_{\mathrm{hit,L}}$ & Forced run\\
\midrule\endfirsthead
\multicolumn{8}{l}{\tablename~\thetable\ (continued)}\\
\toprule
& & & \multicolumn{3}{c}{Consecutive-error ratios} & & \\\cmidrule(lr){4-6}
$N$ & $H$ & Coarse & $q_0$ & $q_1$ & $q_2$ & $s_{\mathrm{hit,L}}$ & Forced run\\
\midrule\endhead
\multicolumn{8}{r}{Continued on next page}\\\endfoot
\bottomrule\endlastfoot
\multirow{3}{*}{$16$} & \multirow{2}{*}{$0.01$} & RK4 & $1.05\times 10^{-3}$ & $1.57\times 10^{-4}$ & $3.10\times 10^{-5}$ & $3$ & exact\\*
 &  & Euler & $4.76\times 10^{-1}$ & $1.68\times 10^{-1}$ & $8.59\times 10^{-1}$ & $15$ & exact\\*
 & $0.02$ & RK4 & $1.03\times 10^{-3}$ & $3.12\times 10^{-3}$ & $8.02\times 10^{-3}$ & $4$ & exact\\
\addlinespace[3pt]
\multirow{3}{*}{$128$} & \multirow{2}{*}{$0.01$} & RK4 & $1.07\times 10^{-3}$ & $2.12\times 10^{-4}$ & $1.60\times 10^{-4}$ & $3$ & exact\\*
 &  & Euler & $4.76\times 10^{-1}$ & $1.68\times 10^{-1}$ & $8.59\times 10^{-1}$ & $26$ & nonfinite after $64$\\*
 & $0.02$ & RK4 & $1.04\times 10^{-3}$ & $3.29\times 10^{-3}$ & $9.22\times 10^{-3}$ & $4$ & exact\\
\end{longtable}\endgroup

\subsection{Implementation checks}
\label{sec:numerical_checks}

The roundoff audit tests preservation of the finite-step exactness in \Cref{prop:unique_fixed_point}, rather than accumulation of perturbations in a contracting iteration. At $N = 64$, the unparenthesized full-interface update leaves an error of about 104 after $N$ corrections, whereas the native solver gives zero at exact completion; at $N = 32$ the corresponding values are 0.283 and zero. The native representation completes after $N$ fine batches and $N - 1$ corrections of the stored starts. In this audit, the discrepancy at exact completion comes from how the update is implemented (operation order and direct prefix assignment), not from an inherent limit of Parareal (\Cref{tab:roundoff}).

Refining the coarse solver at fixed $\Delta T$ on common states gives fitted source slopes 4.06 for RK4 and 1.08 for implicit Euler, consistent with \Cref{prop:actual_rk} while the fine error is smaller (\Cref{tab:refinement}). Holding the state samples fixed avoids confusing a refinement law with chaotic trajectory separation.

At $N = 16$, varying the tolerance and secant guard across four decades leaves the sweep count, one in every replicate, and the mean $\overline{W}_1$ unchanged in both systems. \Cref{tab:sensitivity_summary} reports all six sensitivity cases.

\Needspace{20\baselineskip}\begingroup\small\setlength{\tabcolsep}{10pt}\renewcommand{\arraystretch}{1.15}
\begin{longtable}{@{}lcc@{}}
\caption{\rev{Completion audit using native RK4 maps. The production solver completes in $N$ fine batches; the deliberately unparenthesized full-interface negative control performs $N$ corrections. The negative control's error is not a contraction-roundoff floor.}}\label{tab:roundoff}\\
\toprule
& \multicolumn{2}{c}{$\|U-U^*\|_2$}\\\cmidrule(lr){2-3}
 & $N=32$ & $N=64$\\
\midrule\endfirsthead
\multicolumn{3}{l}{\tablename~\thetable\ (continued)}\\
\toprule
& \multicolumn{2}{c}{$\|U-U^*\|_2$}\\\cmidrule(lr){2-3}
 & $N=32$ & $N=64$\\
\midrule\endhead
\multicolumn{3}{r}{Continued on next page}\\\endfoot
\bottomrule\endlastfoot
Production solver & $0$ & $0$\\
Negative control & $0.283$ & $104$\\
\end{longtable}\endgroup

\Needspace{20\baselineskip}\begingroup\small\setlength{\tabcolsep}{10pt}\renewcommand{\arraystretch}{1.15}
\begin{longtable}{@{}ccc@{}}
\caption{\rev{Actual derivative mismatch $\max\|D\mathcal F-D\mathcal G\|$ under coarse refinement on identical sampled states, at $H=\xi h$ with $h=10^{-3}$ (logistic) and $h=10^{-4}$ (Lorenz-63). The last row gives the least-squares slope of the logarithm of the mismatch against $\log H$.}}\label{tab:refinement}\\
\toprule
& \multicolumn{2}{c}{$\max\|D\mathcal F-D\mathcal G\|$}\\\cmidrule(lr){2-3}
$\xi$ & Logistic & Lorenz-63\\
\midrule\endfirsthead
\multicolumn{3}{l}{\tablename~\thetable\ (continued)}\\
\toprule
& \multicolumn{2}{c}{$\max\|D\mathcal F-D\mathcal G\|$}\\\cmidrule(lr){2-3}
$\xi$ & Logistic & Lorenz-63\\
\midrule\endhead
\multicolumn{3}{r}{Continued on next page}\\\endfoot
\bottomrule\endlastfoot
$5$ & $1.13\times 10^{-3}$ & $1.70\times 10^{-9}$\\*
$10$ & $2.55\times 10^{-3}$ & $2.65\times 10^{-8}$\\*
$20$ & $5.41\times 10^{-3}$ & $4.02\times 10^{-7}$\\*
$50$ & $1.41\times 10^{-2}$ & $1.33\times 10^{-5}$\\*
$100$ & $2.92\times 10^{-2}$ & $3.88\times 10^{-4}$\\
\midrule
Slope & $1.08$ & $4.06$\\
\end{longtable}\endgroup

\begin{table}[H]
  \centering
  \small
  \caption{Sensitivity of the secant rule at $N = 16$ with four replicate starts: three tolerances and three guards for each system. No variation changes the sweep count or the mean $\overline{W}_1$ in either system.}
  \label{tab:sensitivity_summary}
\setlength{\tabcolsep}{10pt}\renewcommand{\arraystretch}{1.15}
\begin{tabular}{@{}ccccc@{}}\toprule
Parameter & $\epsilon$ & $\tau_{\mathrm{sec}}$ & $s$ & $\overline{W}_1$\\
\midrule
\multicolumn{5}{c}{\textit{Lorenz-63}}\\*
\multirow{3}{*}{tolerance} & $10^{-7}$ & \multirow{3}{*}{$10^{-10}$} & $1$ & $5.52\times 10^{-4}$\\*
 & $10^{-9}$ &  & $1$ & $5.52\times 10^{-4}$\\*
 & $10^{-11}$ &  & $1$ & $5.52\times 10^{-4}$\\
\addlinespace[3pt]
\multirow{3}{*}{guard} & \multirow{3}{*}{$10^{-9}$} & $10^{-12}$ & $1$ & $5.52\times 10^{-4}$\\*
 &  & $10^{-10}$ & $1$ & $5.52\times 10^{-4}$\\*
 &  & $10^{-8}$ & $1$ & $5.52\times 10^{-4}$\\
\midrule
\multicolumn{5}{c}{\textit{Lorenz-96}}\\*
\multirow{3}{*}{tolerance} & $10^{-7}$ & \multirow{3}{*}{$10^{-10}$} & $1$ & $1.54\times 10^{-1}$\\*
 & $10^{-9}$ &  & $1$ & $1.54\times 10^{-1}$\\*
 & $10^{-11}$ &  & $1$ & $1.54\times 10^{-1}$\\
\addlinespace[3pt]
\multirow{3}{*}{guard} & \multirow{3}{*}{$10^{-9}$} & $10^{-12}$ & $1$ & $1.54\times 10^{-1}$\\*
 &  & $10^{-10}$ & $1$ & $1.54\times 10^{-1}$\\*
 &  & $10^{-8}$ & $1$ & $1.54\times 10^{-1}$\\
\bottomrule\end{tabular}
\end{table}

\section{Periodic control}
\label{sec:periodic_table}

The Lorenz-63 parameters are $(10, 100.5, 8/3)$, and the prescribed start is $(0,\allowbreak 3.6839592143970576,\allowbreak 73.30237399900534)$. The return and multiplier computations are numerical diagnostics, not a rigorous existence proof of the periodic orbit. The coarse and fine schemes, tolerances, and cost conventions are those of the main text. The computed period is $1.09623881358$, with return error $3.38\times 10^{-12}$, and the estimated fixed-period Floquet multipliers are $-0.137$, $-2.27\times 10^{-6}$, and $1.00$. The sampled one-chunk amplification is 158 and the largest sampled multi-chunk amplification is 14.8; this transient amplification explains why a neutral asymptotic exponent need not give a small chunk amplification. \Cref{tab:scaling_periodic} shows how the two weights respond on this dynamically active control.

\Needspace{20\baselineskip}\begingroup\small\setlength{\tabcolsep}{1.5pt}\renewcommand{\arraystretch}{1.15}\setlength{\LTleft}{0pt}\setlength{\LTright}{0pt}
\begin{longtable}{@{\extracolsep{\fill}}ccccccc@{\hspace{8pt}}cccc@{}}
\caption{\rev{Complete strong- and weak-scaling records for the periodic Lorenz-63 control. The unweighted interface error $\|e\|_2$ measures tracking of the serial fine discrete reference; $K_F$, Stop, and $\mathcal S$ follow the conventions in \Cref{sec:scaling}. The scaling grids share $N=T=128$, so those strong- and weak-scaling rows are the same run.}}\label{tab:scaling_periodic}\\
\toprule
& & & \multicolumn{4}{c}{Strong scaling ($T=128$)} & \multicolumn{4}{c}{Weak scaling ($\Delta T=1$)}\\
\cmidrule(lr){4-7}\cmidrule(lr){8-11}
$\xi$ & Rule & $N$ & $K_F$ & Stop & $\mathcal S$ & $\|e\|_2$ & $K_F$ & Stop & $\mathcal S$ & $\|e\|_2$\\
\midrule\endfirsthead
\multicolumn{11}{l}{\tablename~\thetable\ (continued)}\\
\toprule
& & & \multicolumn{4}{c}{Strong scaling ($T=128$)} & \multicolumn{4}{c}{Weak scaling ($\Delta T=1$)}\\
\cmidrule(lr){4-7}\cmidrule(lr){8-11}
$\xi$ & Rule & $N$ & $K_F$ & Stop & $\mathcal S$ & $\|e\|_2$ & $K_F$ & Stop & $\mathcal S$ & $\|e\|_2$\\
\midrule\endhead
\multicolumn{11}{r}{Continued on next page}\\\endfoot
\bottomrule\endlastfoot
\multirow{7}{*}{$10$} & \multirow{7}{*}{standard} & $2$ & $2$ & E & $0.952$ & $0$ & $2$ & E & $0.952$ & $0$\\*
 &  & $4$ & $3$ & C & $1.11$ & $5.64\times 10^{-12}$ & $3$ & C & $1.11$ & $8.14\times 10^{-13}$\\*
 &  & $8$ & $3$ & C & $1.67$ & $2.29\times 10^{-11}$ & $3$ & C & $1.67$ & $3.80\times 10^{-11}$\\*
 &  & $16$ & $3$ & C & $2.22$ & $3.23\times 10^{-11}$ & $3$ & C & $2.22$ & $1.83\times 10^{-11}$\\*
 &  & $32$ & $3$ & C & $2.67$ & $6.42\times 10^{-11}$ & $3$ & C & $2.67$ & $7.66\times 10^{-11}$\\*
 &  & $64$ & $3$ & C & $2.96$ & $8.70\times 10^{-11}$ & $3$ & C & $2.96$ & $9.08\times 10^{-11}$\\*
 &  & $128$ & $3$ & C & $3.14$ & $1.67\times 10^{-10}$ & $3$ & C & $3.14$ & $1.67\times 10^{-10}$\\*
\midrule
\multirow{7}{*}{$10$} & \multirow{7}{*}{$\psi$ ($w=1$)} & $2$ & $2$ & E & $0.952$ & $0$ & $2$ & E & $0.952$ & $0$\\*
 &  & $4$ & $2$ & C & $1.60$ & $1.57\times 10^{-10}$ & $2$ & C & $1.60$ & $4.49\times 10^{-12}$\\*
 &  & $8$ & $2$ & C & $2.42$ & $6.75\times 10^{-10}$ & $2$ & C & $2.42$ & $2.77\times 10^{-11}$\\*
 &  & $16$ & $2$ & C & $3.27$ & $1.76\times 10^{-9}$ & $2$ & C & $3.27$ & $9.55\times 10^{-11}$\\*
 &  & $32$ & $2$ & C & $3.95$ & $2.93\times 10^{-9}$ & $2$ & C & $3.95$ & $4.00\times 10^{-10}$\\*
 &  & $64$ & $2$ & C & $4.41$ & $3.62\times 10^{-9}$ & $2$ & C & $4.41$ & $1.28\times 10^{-9}$\\*
 &  & $128$ & $2$ & C & $4.69$ & $7.98\times 10^{-9}$ & $2$ & C & $4.69$ & $7.98\times 10^{-9}$\\*
\midrule
\multirow{7}{*}{$10$} & \multirow{7}{*}{$\psi$ (secant $w$)} & $2$ & $2$ & E & $0.952$ & $0$ & $2$ & E & $0.952$ & $0$\\*
 &  & $4$ & $2$ & C & $1.60$ & $1.57\times 10^{-10}$ & $2$ & C & $1.60$ & $4.49\times 10^{-12}$\\*
 &  & $8$ & $2$ & C & $2.42$ & $6.75\times 10^{-10}$ & $2$ & C & $2.42$ & $2.77\times 10^{-11}$\\*
 &  & $16$ & $2$ & C & $3.27$ & $1.76\times 10^{-9}$ & $2$ & C & $3.27$ & $9.55\times 10^{-11}$\\*
 &  & $32$ & $2$ & C & $3.95$ & $2.93\times 10^{-9}$ & $2$ & C & $3.95$ & $4.00\times 10^{-10}$\\*
 &  & $64$ & $2$ & C & $4.41$ & $3.62\times 10^{-9}$ & $2$ & C & $4.41$ & $1.28\times 10^{-9}$\\*
 &  & $128$ & $2$ & C & $4.69$ & $7.98\times 10^{-9}$ & $2$ & C & $4.69$ & $7.98\times 10^{-9}$\\*
\midrule
\multirow{7}{*}{$100$} & \multirow{7}{*}{standard} & $2$ & $2$ & E & $0.995$ & $0$ & $2$ & E & $0.995$ & $0$\\*
 &  & $4$ & $4$ & E & $0.985$ & $0$ & $4$ & E & $0.985$ & $0$\\*
 &  & $8$ & $7$ & C & $1.10$ & $8.21\times 10^{-11}$ & $5$ & C & $1.52$ & $1.48\times 10^{-10}$\\*
 &  & $16$ & $8$ & C & $1.79$ & $2.34\times 10^{-10}$ & $6$ & C & $2.37$ & $7.69\times 10^{-11}$\\*
 &  & $32$ & $9$ & C & $2.80$ & $1.50\times 10^{-10}$ & $7$ & C & $3.57$ & $1.22\times 10^{-9}$\\*
 &  & $64$ & $8$ & C & $5.02$ & $4.73\times 10^{-9}$ & $8$ & C & $5.02$ & $1.47\times 10^{-9}$\\*
 &  & $128$ & $10$ & C & $5.75$ & $2.16\times 10^{-9}$ & $10$ & C & $5.75$ & $2.16\times 10^{-9}$\\*
\midrule
\multirow{7}{*}{$100$} & \multirow{7}{*}{$\psi$ ($w=1$)} & $2$ & $2$ & E & $0.995$ & $0$ & $2$ & E & $0.995$ & $0$\\*
 &  & $4$ & $4$ & E & $0.985$ & $0$ & $4$ & E & $0.985$ & $0$\\*
 &  & $8$ & $7$ & C & $1.10$ & $8.21\times 10^{-11}$ & $4$ & C & $1.90$ & $1.12\times 10^{-8}$\\*
 &  & $16$ & $7$ & C & $2.04$ & $2.02\times 10^{-8}$ & $6$ & C & $2.37$ & $7.69\times 10^{-11}$\\*
 &  & $32$ & $8$ & C & $3.14$ & $7.20\times 10^{-9}$ & $6$ & C & $4.15$ & $8.57\times 10^{-8}$\\*
 &  & $64$ & $8$ & C & $5.02$ & $4.73\times 10^{-9}$ & $7$ & C & $5.71$ & $6.14\times 10^{-8}$\\*
 &  & $128$ & $8$ & C & $7.16$ & $1.03\times 10^{-6}$ & $8$ & C & $7.16$ & $1.03\times 10^{-6}$\\*
\midrule
\multirow{7}{*}{$100$} & \multirow{7}{*}{$\psi$ (secant $w$)} & $2$ & $2$ & E & $0.995$ & $0$ & $2$ & E & $0.995$ & $0$\\*
 &  & $4$ & $4$ & E & $0.985$ & $0$ & $4$ & E & $0.985$ & $0$\\*
 &  & $8$ & $5$ & C & $1.52$ & $4.05\times 10^{-6}$ & $3$ & C & $2.52$ & $2.08\times 10^{-5}$\\*
 &  & $16$ & $4$ & C & $3.52$ & $1.76\times 10^{-3}$ & $3$ & C & $4.68$ & $4.25\times 10^{-4}$\\*
 &  & $32$ & $3$ & C & $8.21$ & $9.78\times 10^{-2}$ & $3$ & C & $8.21$ & $6.09\times 10^{-3}$\\*
 &  & $64$ & $3$ & C & $13.2$ & $1.14\times 10^{-1}$ & $3$ & C & $13.2$ & $3.37\times 10^{-2}$\\*
 &  & $128$ & $3$ & C & $18.9$ & $4.78\times 10^{-1}$ & $3$ & C & $18.9$ & $4.78\times 10^{-1}$\\*
\end{longtable}\endgroup

\section{Outcomes with a poor coarse solver}
\label{sec:poorcoarse_details}

Here the Lorenz-63 coarse solver is explicit Euler at $H = 0.01$ or 0.02, and the fine solver is RK4, with eight replicate starts at each horizon. \Cref{tab:bad_coarse_full} extends \Cref{tab:bad_coarse} to both step sizes and all outcomes, showing the work required at $H = 0.01$, where some $N = 128$ runs already fail, and nonfinite runs throughout at $H = 0.02$. Statistics use only finite trajectories; the outcome column also counts failures.
\begin{table}[H]
  \centering
  \small
  \caption{Lorenz-63 control with an explicit Euler coarse solver: complete outcomes. Fine and coarse-only baseline means use all eight starts. Accepted means use each rule's finite runs, and paired intervals compare those runs with their matched coarse-only counterparts; these comparisons are conditional on finite completion.}
  \label{tab:bad_coarse_full}
\setlength{\tabcolsep}{10pt}\renewcommand{\arraystretch}{1.15}
\begin{tabular}{@{}ccccccc@{}}\toprule
& & & & & & 95\% paired CI\\
\cmidrule(lr){7-7}
$H$ & $N$ & Trajectory & Outcomes & $K_F$ & $\overline{W}_1$ & vs coarse-only\\
\midrule
\multirow{8}{*}{$0.01$} & \multirow{4}{*}{$16$} & fine baseline &  &  & $1.53$ & \\*
 &  & coarse-only &  &  & $2.78$ & \\*
 &  & $\psi$ ($w=e^{\lambda\Delta T}$) & 8E & $16$ & $0$ & [$-3.65$, $-1.85$]\\*
 &  & $\psi$ (secant $w$) & 2C, 6E & $12$--$16$ & $0.0806$ & [$-3.54$, $-1.81$]\\
\addlinespace[3pt]
 & \multirow{4}{*}{$128$} & fine baseline &  &  & $0.363$ & \\*
 &  & coarse-only &  &  & $1.58$ & \\*
 &  & $\psi$ ($w=e^{\lambda\Delta T}$) & 5C, 3F & $21$--$22$ & $1.26$ & [$-0.620$, $0.0117$]\\*
 &  & $\psi$ (secant $w$) & 6C, 2F & $10$--$13$ & $1.39$ & [$-0.410$, $0.0314$]\\
\midrule
\multirow{8}{*}{$0.02$} & \multirow{4}{*}{$16$} & fine baseline &  &  & $1.53$ & \\*
 &  & coarse-only &  &  & $2.29$ & \\*
 &  & $\psi$ ($w=e^{\lambda\Delta T}$) & 8F & -- & -- & --\\*
 &  & $\psi$ (secant $w$) & 8F & -- & -- & --\\
\addlinespace[3pt]
 & \multirow{4}{*}{$128$} & fine baseline &  &  & $0.363$ & \\*
 &  & coarse-only &  &  & $1.67$ & \\*
 &  & $\psi$ ($w=e^{\lambda\Delta T}$) & 8F & -- & -- & --\\*
 &  & $\psi$ (secant $w$) & 8F & -- & -- & --\\
\bottomrule\end{tabular}
\end{table}

\section{Long-chunk grids}
\label{sec:submitted_grid}

The long-chunk grids contain 150 distinct $(N, T)$ pairs: $N = 100$ with $T = 10, 20, \ldots, 500$; $T = 500$ with $N = 1, 3, \ldots, 99$ and $N = 100$; and $T = 5N$ for the same values of $N$. The common endpoint $(100, 500)$ is included once in the combined dataset. At each pair we run two coarsening ratios and three stopping rules, giving 900 runs. Where a step does not divide a chunk, the runs use full steps and one shortened terminal step.

All 900 runs remain finite, with status E in 119 and C in 781. Exact completion is reported when $K_F = N$, after $N - 1$ corrections, and reproduces the serial reference bitwise. On the fixed-$T = 500$ grid, 17 secant runs reach exact completion: eight at $\xi = 10$ and nine at $\xi = 100$, with chunk lengths at least $500/15 \approx 33.3$ and $500/17 \approx 29.4$, respectively. The longest secant acceptances after one or two corrections (two or three fine batches) are $500/29 \approx 17.24$ at $\xi = 10$ and $500/53 \approx 9.43$ at $\xi = 100$. These observations describe the sampled grid; they give neither a universal threshold nor an accuracy guarantee. \Cref{tab:submitted_grid_summary} records the family-wide count ranges and outcomes; \Cref{fig:submitted_grid} shows how the counts vary across the grid. \Cref{tab:submitted_grid_endpoint} retains the shared endpoint separately, where zero-correction Lyapunov acceptance at $\xi = 10$ shows the danger of strong discounting. This case is not recoverable from the family ranges alone.

\Needspace{20\baselineskip}\begingroup\small\setlength{\tabcolsep}{10pt}\renewcommand{\arraystretch}{1.15}
\begin{longtable}{@{}cccccc@{}}
\caption{\rev{Long-chunk grids: family summaries. Each run uses $K_F=s+1$ fine batches; C and E count stops by the criterion and by exact completion, and no run exhausts its budget or fails. The last column is the largest interface error over the family's runs. The common endpoint is included in each family summary but only once in the 900 distinct runs; row totals across families therefore exceed 900.}}\label{tab:submitted_grid_summary}\\
\toprule
& & & \multicolumn{2}{c}{Stops} & \\\cmidrule(lr){4-5}
$\xi$ & Rule & $K_F$ & C & E & $\max\|e\|_2$\\
\midrule\endfirsthead
\multicolumn{6}{l}{\tablename~\thetable\ (continued)}\\
\toprule
& & & \multicolumn{2}{c}{Stops} & \\\cmidrule(lr){4-5}
$\xi$ & Rule & $K_F$ & C & E & $\max\|e\|_2$\\
\midrule\endhead
\multicolumn{6}{r}{Continued on next page}\\\endfoot
\bottomrule\endlastfoot
\multicolumn{6}{c}{\textit{Growing horizon: $N=100$}}\\*
\multirow{3}{*}{$10$} & standard & $3$--$100$ & $49$ & $1$ & $8.58\times 10^{-7}$\\*
 & $\psi$ ($w=e^{\lambda\Delta T}$) & $1$--$2$ & $50$ & $0$ & $220$\\*
 & $\psi$ (secant $w$) & $2$ & $50$ & $0$ & $211$\\
\addlinespace[3pt]
\multirow{3}{*}{$100$} & standard & $5$--$99$ & $50$ & $0$ & $2.87\times 10^{-6}$\\*
 & $\psi$ ($w=e^{\lambda\Delta T}$) & $2$--$4$ & $50$ & $0$ & $217$\\*
 & $\psi$ (secant $w$) & $2$--$3$ & $50$ & $0$ & $224$\\
\midrule
\multicolumn{6}{c}{\textit{Strong scaling: $T=500$}}\\*
\multirow{3}{*}{$10$} & standard & $1$--$99$ & $12$ & $39$ & $9.30\times 10^{-7}$\\*
 & $\psi$ ($w=e^{\lambda\Delta T}$) & $1$--$2$ & $50$ & $1$ & $214$\\*
 & $\psi$ (secant $w$) & $1$--$15$ & $43$ & $8$ & $201$\\
\addlinespace[3pt]
\multirow{3}{*}{$100$} & standard & $1$--$99$ & $8$ & $43$ & $1.76\times 10^{-8}$\\*
 & $\psi$ ($w=e^{\lambda\Delta T}$) & $1$--$3$ & $50$ & $1$ & $213$\\*
 & $\psi$ (secant $w$) & $1$--$17$ & $42$ & $9$ & $213$\\
\midrule
\multicolumn{6}{c}{\textit{Weak scaling: $T=5N$}}\\*
\multirow{3}{*}{$10$} & standard & $1$--$99$ & $47$ & $4$ & $4.61\times 10^{-6}$\\*
 & $\psi$ ($w=e^{\lambda\Delta T}$) & $1$--$2$ & $50$ & $1$ & $214$\\*
 & $\psi$ (secant $w$) & $1$--$2$ & $50$ & $1$ & $201$\\
\addlinespace[3pt]
\multirow{3}{*}{$100$} & standard & $1$--$99$ & $44$ & $7$ & $2.70\times 10^{-6}$\\*
 & $\psi$ ($w=e^{\lambda\Delta T}$) & $1$--$3$ & $49$ & $2$ & $213$\\*
 & $\psi$ (secant $w$) & $1$--$3$ & $49$ & $2$ & $213$\\
\end{longtable}\endgroup

%
\Needspace{20\baselineskip}\begingroup\small\setlength{\tabcolsep}{10pt}\renewcommand{\arraystretch}{1.15}
\begin{longtable}{@{}ccccc@{}}
\caption{\rev{Common endpoint $N=100$, $T=500$. All six runs stop by their criterion. $\mathcal S=N/(K_F+\zeta K_G)$ is the modelled speed-up, with $\zeta$ the coarse-to-fine step ratio per chunk. At $\xi=100$ both weighted rules take the same fine batches, so they accept the same iterate and share one row. At $\xi=10$ the Lyapunov rule accepts the coarse interfaces with no correction: this is vacuous acceptance relative to pointwise accuracy, not a successful correction of the tail.}}\label{tab:submitted_grid_endpoint}\\
\toprule
$\xi$ & Rule & $K_F$ & $\|e\|_2$ & $\mathcal S$\\
\midrule\endfirsthead
\multicolumn{5}{l}{\tablename~\thetable\ (continued)}\\
\toprule
$\xi$ & Rule & $K_F$ & $\|e\|_2$ & $\mathcal S$\\
\midrule\endhead
\multicolumn{5}{r}{Continued on next page}\\\endfoot
\bottomrule\endlastfoot
\multirow{3}{*}{$10$} & standard & $99$ & $2.87\times 10^{-9}$ & $0.168$\\*
 & $\psi$ ($w=e^{\lambda\Delta T}$) & $1$ & $214$ & $9.17$\\*
 & $\psi$ (secant $w$) & $2$ & $201$ & $4.61$\\
\addlinespace[3pt]
\multirow{2}{*}{$100$} & standard & $99$ & $4.48\times 10^{-9}$ & $0.673$\\*
 & $\psi$ (either $w$) & $3$ & $213$ & $16.8$\\
\end{longtable}\endgroup

\begin{figure}[H]
  \centering
  \includegraphics[width=\textwidth]{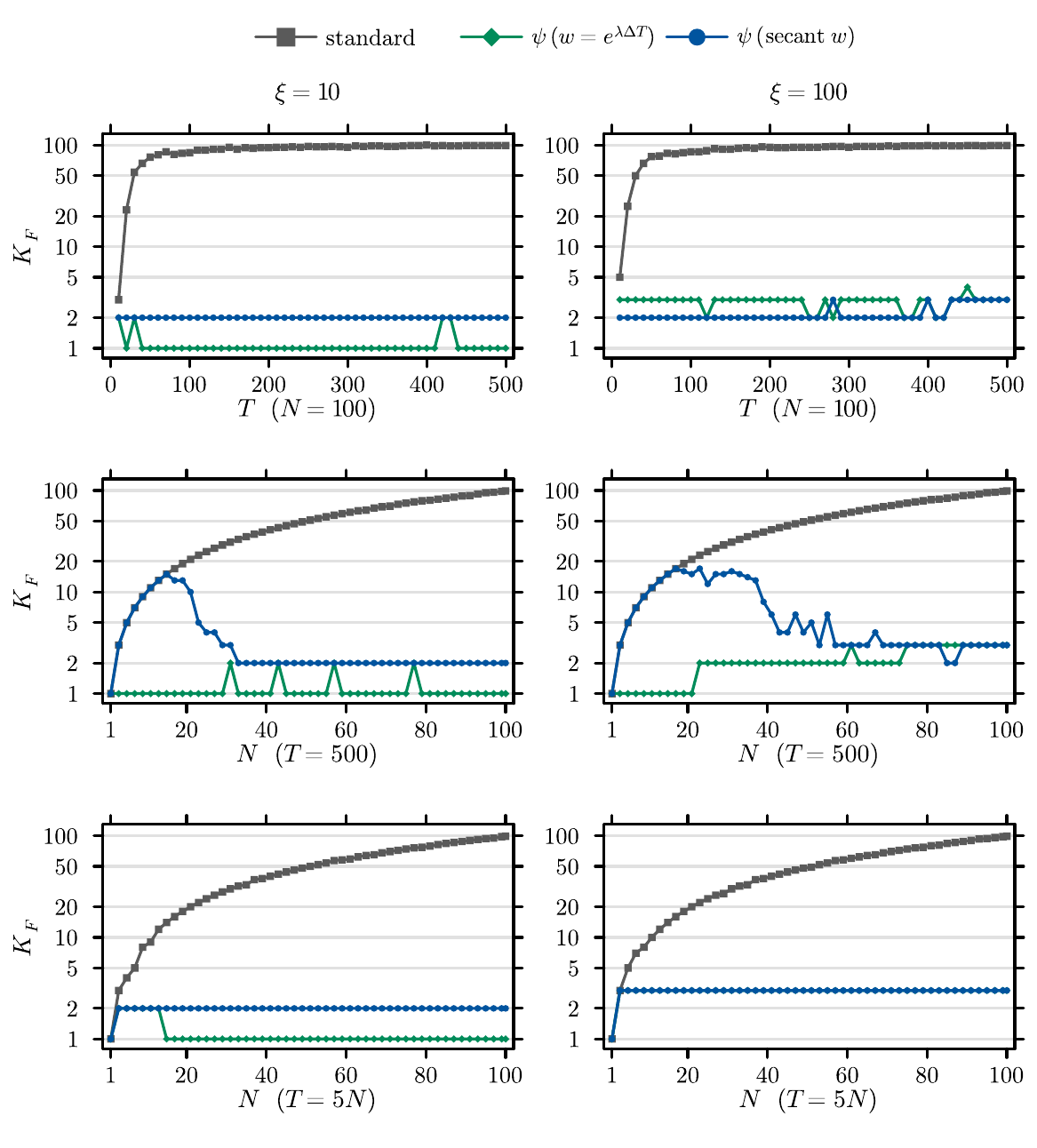}
  \caption{Runs on the long-chunk grids. Logarithmic vertical axes separate small weighted batch counts from the approximately $N$ standard counts while retaining the expensive long-chunk cases. Zero-correction acceptance and expensive long-chunk secant runs are visible; endpoint errors are tabulated separately.}
  \label{fig:submitted_grid}
\end{figure}

\section{Window length and the work model}
\label{sec:window_cost}

The work $\mathcal{W}$ in \Cref{eq:cost} is normalized by one fine chunk. With equal fine and coarse step costs, a fine chunk uses $L\xi$ steps, so the modelled parallel critical-path work of one window in elementary step units is
\begin{equation}
  \label{eq:sm_cost_window}
  \mathcal{W}_{\mathrm{win}} = L\xi\mathcal{W} = K_F L\xi + K_G L.
\end{equation}
At fixed $h, H, N$, reducing $L \ge 1$ lowers work per window but requires more windows to cover a prescribed horizon $T_{\mathrm{total}}$. For nonoverlapping windows with the same $s, K_F$, neglecting a partial last window,
\begin{equation}
  \label{eq:sm_cost_total}
  \mathcal{W}_{\mathrm{total}} \approx \frac{T_{\mathrm{total}}}{NLH}\mathcal{W}_{\mathrm{win}}
  = \frac{T_{\mathrm{total}}}{h}\left(\frac{K_F}{N} + \frac{K_G}{N\xi}\right).
\end{equation}
Thus $L$ cancels: any saving in this model comes through smaller sweep or batch counts. For overlapping windows shifted by $\Delta\tau < NLH$, the window count is instead approximately $T_{\mathrm{total}}/\Delta\tau$. At unchanged work per window, the cost increases by the factor $NLH/\Delta\tau$; any saving from reuse across windows must be accounted for separately. This is critical-path work, not processor-hours; it excludes communication, synchronization, restarts, and renewed derivative sampling, whose costs recur at every window.

Reducing $h, H$ at fixed $L, \xi, N$ instead keeps the step count per window unchanged, increases the number of windows, and changes the fine discretization. At fixed $\Delta T$, halving both steps doubles the steps per chunk; reducing only $H$ retains the fine target but increases coarse cost. Reducing $L$ reaches a floor at one coarse step and alone supplies no general contraction guarantee.

\paragraph{Local defects and transfer between windows}
\label{sec:window_defects}
Suppose every accepted window has $N_w$ chunks, uses fixed $w \ge 1$, and satisfies $\psi \le \epsilon$. Applied separately to each window, \Cref{rem:w_not_free} gives
\begin{equation}
  \label{eq:sm_window_defects}
  \max_{m, n}\|d_{m, n}\| \le N_w\epsilon w^{N_w},
\end{equation}
independently of the number of windows. For adaptive settings the corresponding requirement is a uniform bound on $N_{w, m}\epsilon_m w_m^{N_{w, m}}$. A window returned after budget exhaustion without acceptance does not satisfy this premise automatically.

For consecutive nonoverlapping windows whose fine maps share a common discretization, copying the accepted endpoint into the next start adds no state jump. With overlap, a transferred value at a common chunk boundary can likewise be copied exactly, but interpolation or a changed chunk schedule requires a separate transfer-consistency estimate. Local defect control alone proves neither closeness to the global fine orbit nor selection of its physical statistics. In particular, statistics averaged at equal relative positions across overlapping windows differ from a single chronological time average, and the windows are correlated.

\section{Computational records}
\label{sec:reproducibility}

All stochastic experiments, including calibration and confirmation, were regenerated under the fixed protocols, and all numerical records, tables, and figures were generated with Julia 1.13.1. Random starts and bootstrap resamples come from seeded MersenneTwister streams, and the same seed gives a different stream in Julia 1.10, for example, so the stochastic results reproduce exactly only with the version named here. The source snapshots use ArrowMacros 0.2.3, NSDEBase 0.4.0, and NSDERungeKutta and NSDETimeParallel 0.3.0. All numerical records were recomputed with native Runge--Kutta propagation and native Parareal, using the stopping and counting conventions in \Cref{alg:parareal_psi,eq:relative_stop}. The 672 scaling runs (600 distinct parameter configurations) comprise 336 archived strong- and weak-scaling records, selected without duplication for the article and supplementary tables, and 336 growing-horizon records across the four problems; \Cref{fig:KvsTLorenz} shows the Lorenz-63 subset. The 900-case long-chunk reconciliation is described in \Cref{sec:submitted_grid}; statistical, sensitivity, and poor-coarse protocols are in \Cref{sec:statistics_full,sec:numerical_checks,sec:poorcoarse_details}. \Cref{sec:functional_accuracy,sec:stopping_adequacy} describe the exploratory functional study, long references, calibration, and independent confirmation tests, retaining failed outcomes and forced histories. The archived logistic secant runs are not used in the paper's comparisons.

Tables and figures are generated from these records. The shared native chunk schedule and serial-reference convention are stated in \Cref{sec:experiments,sec:submitted_grid}; exact completion, refinement, and stopping sensitivity are checked in \Cref{sec:numerical_checks}.

\Needspace{7\baselineskip}
The scripts, source snapshots, numerical records, logs, tables, and figures are available at \url{https://github.com/giancarloantonucci/finite-budget-parareal}; a snapshot of the repository accompanies these supplementary materials. The records are Julia-serialised files and are reliably readable only with the versions named above. For other enquiries, quote SISC manuscript M186783 and contact:
\begin{quote}
  \email{giancarlo.antonucci@ukaea.uk}\\
  \email{PublicationsManager@ukaea.uk}
\end{quote}

\section{Notation}
\label{sec:notation}

\Cref{tab:notation} collects the symbols used by the convergence, stopping, and cost arguments. Hats denote sampled estimates, not validated bounds. A star denotes the serial fine solution. Double bars denote state, induced matrix, and trajectory norms; single bars denote scalar absolute values.
\small
\begin{longtable}{@{}p{.29\textwidth}p{.67\textwidth}@{}}%
  \caption{Notation for the article and supplement.}\label{tab:notation}\\
  \toprule
  Symbol & Meaning \\
  \midrule
  \endfirsthead
  \toprule
  Symbol & Meaning \\
  \midrule
  \endhead
  \multicolumn{2}{l}{\textit{Time grids and propagators}} \\*
  $d, u_0, f, J(t)$ & State dimension, prescribed initial state, vector field, and its Jacobian along a reference trajectory. \\
  $T, N, \Delta T$ & Horizon, number of chunks, and chunk length $T/N$. \\
  $h, H, \xi, L$ & Fine/coarse steps, ratio $H/h$, and coarse steps per chunk $\Delta T/H$. \\
  $\mathcal{F}, \mathcal{G}$ & Actual fine and coarse chunk maps. \\
  $F_{\mathrm{fr}}', G_{\mathrm{fr}}'$ & Frozen-Jacobian model products, not generally actual-map derivatives. \\
  $A_{\mathcal{F}}, b_{\mathcal{F}}, c_{\mathcal{F}}$ & Fine tableau; similarly for the coarse method. \\
  $p_{\mathcal{F}}, p_{\mathcal{G}}$ & Orders of the numerical methods, distinct from the trajectory exponent $p$. \\
  \midrule
  \multicolumn{2}{l}{\textit{Iteration and residuals}} \\*
  $U^s, U^*, \mathcal{X}$ & Interfaces after $s$ sweeps, serial fine interfaces, affine space with fixed $U_0 = u_0$. \\
  $\Phi, v$ & Parareal update map and correction, $\Phi(U) = U + v(U)$; $\Phi$ also denotes a generic iteration in the budget framework. \\
  $\mathcal{N}, \mathcal{D}, \Psi, M$ & State neighbourhood, invariant iteration domain, generic map, and common finite-termination count. \\
  $\mathcal{E}_{\mathcal{F}}, \mathcal{A}, \mathcal{A}_r$ & Stacked fine residual, asymptotic basin, and finite-budget acceptance set. \\
  $A_j, D_j$ & New-iterate coarse derivative $D\mathcal{G}(\Phi_j(U))$ and old-iterate derivative mismatch. \\
  \midrule
  \multicolumn{2}{l}{\textit{Derivative bounds and prediction}} \\*
  $S, \Lambda_{\mathcal{F}}, \Lambda_{\mathcal{G}}$ & Uniform source bound and fine/coarse amplification bounds on the required neighbourhood. \\
  $\Gamma_i$ & Uniform bounds on length-$i$ products of coarse transport derivatives. \\
  $\mu, \nu, C_1, C_2$ & Bounds on $J, J'$ and frozen-model curvature/drift coefficients. \\
  $\hat{S}_{\mathrm{map}}, \hat{S}_{\mathrm{fr}}$ & Inflated sampled actual-map source and leading frozen-model comparison value. \\
  $\hat{\beta}_{\mathrm{fr}}, s_{\mathrm{fr}}$ & Frozen-source factor and uncapped budget, retaining sampled actual-map transport. \\
  $w_{\mathrm{RK}}, \hat{\beta}_{\mathrm{RK}}, s_{\mathrm{RK}}$ & Coarse-node RK4 diagnostic weight, factor, and budget; defined in \Cref{sec:frozen_prediction}. \\
  $\beta, \beta_W, \hat{\beta}_W$ & Unweighted or generic contraction bound, weighted bound, and sampled weighted prediction. \\
  \midrule
  \multicolumn{2}{l}{\textit{Balls and iteration budgets}} \\*
  $B_\rho, \overline{B}_\rho$ & Open/closed balls in the stated trajectory norm on $\mathcal{X}$. \\
  $r, \varrho, R(s)$ & Inner target radius, contraction-ball radius, and certified outer radius $\min\{\varrho, r\beta^{-s}\}$. \\
  $R_{\max}(s)$ & Supremal radius of centred open balls in the finite-budget acceptance set; at least $R(s)$. \\
  $E_0, E_{\mathrm{tr}}, \hat{e}_0$ & Sufficient initial-error bound, transported-defect bound, and its sampled estimate. \\
  $s, s_b, s_{\mathrm{pred}}, s_{\mathrm{hit}}$ & Correction sweeps, available budget, predicted budget, and first observed target-error hit. \\
  $K_F$ & Number of parallel fine batches, including the final acceptance-testing batch under either stopping rule. \\
  $E_{\mathrm{start}}$ & Weighted error on stored chunk starts, defined in \Cref{eq:start_norm}. \\
  $q, Q_q$ & Q-order parameter and asymptotic error-ratio limsup; separate from a uniform finite-step inequality. \\
  $q_s$ & Consecutive-error ratio $E_{s+1}/E_s$ in \Cref{tab:prediction_coarse_ratios}. \\
  \midrule
  \multicolumn{2}{l}{\textit{Defects, weights, and accuracy}} \\*
  $e_n, d_n, a_n$ & Interface error, fine-map defect, and transported weighted-defect envelope. \\
  $\epsilon, \psi, \psi_\infty$ & Stopping tolerance, averaged proximity function used in the experiments, and maximum-defect theoretical variant. \\
  $w, W, \lambda, \tau_{\mathrm{sec}}$ & Weight base and diagonal matrix, asymptotic Lyapunov exponent, and secant guard. \\
  $\lambda_{\Delta T}^{\mathcal{F}}(U)$ & Maximal Euclidean finite-time exponent of the numerical chunk map, defined in \Cref{eq:ftle_chunk}. \\
  $\theta, H_{N, p}, C_N$ & Ratio $\Lambda_{\mathcal{F}}/w$, $\ell^p$ norm of $(1, \theta, \ldots, \theta^{N - 1})$, and averaged-defect conversion constant $C_N = NH_{N, p}$. \\
  $\epsilon_a, \epsilon_r$ & Absolute safeguard and relative tolerance for each interface update. \\
  $\delta$ & Unweighted pseudo-orbit defect tolerance. \\
  \midrule
  \multicolumn{2}{l}{\textit{Work and reported outcomes}} \\*
  $\mathcal{W}, \mathcal{S}, \overline{W}_1$ & Model work, model speedup, and mean marginal Wasserstein distance to the fine reference. \\
  $K_G$ & Sequential coarse calls after exact-prefix reuse, defined in \Cref{eq:cost}. \\
  $\mathcal{C}, \mathcal{Q}, J_F, \zeta$ & Modelled elapsed-cost ratio, total-work ratio, number of fine chunk solves, and coarse/fine cost ratio; see \Cref{eq:sm_cost_models}. \\
  $E_j(M), n_{\mathrm{rep}}, h_{\mathrm{ser}}$ & Mean absolute error of observable $j$ for method $M$, replicate count, and serial comparison step. \\
  $L_{0.95}, U_{0.95}$ & One-sided 95\% bootstrap lower and upper bounds. \\
  $n_\eta, \eta$ & Accurate-prefix length and error threshold; $\eta$ also denotes the error level in \Cref{sec:physical_target}. \\
  $\eta_j$ & Per-sweep perturbation in \Cref{rem:roundoff}. \\
  $N_w$ & Chunks per window in \Cref{sec:window_defects}. \\
  $n_{\mathrm{st}}, n_{\mathcal{F}}, n_{\mathcal{G}}$ & Runge--Kutta stage counts for a generic, fine, and coarse method. \\
  C/E/B/F & Criterion met, exact completion, budget exhausted, nonfinite failure. \\
  \bottomrule
\end{longtable}
\normalsize

\bibliographystyle{siamplain}
\bibliography{references}